\documentclass[a4paper, 10pt]{article}

\usepackage[utf8]{inputenc}    
\usepackage[T1]{fontenc}
\usepackage{float}
\usepackage{epsf}
\usepackage{mathrsfs}
\usepackage{lmodern}
\usepackage{graphicx}
\usepackage{subcaption}
\usepackage{caption}
\usepackage{wrapfig}	
\usepackage[english]{babel}
\usepackage{amsmath, amsfonts, amssymb} 
\usepackage[citecolor=blue]{hyperref} 
\usepackage{amsthm}
\hypersetup{colorlinks = true, urlcolor = blue,linkcolor=red, bookmarksopen = true}
\usepackage[nomarginpar]{geometry}
\usepackage{young}
\usepackage[usenames,dvipsnames]{xcolor}
\usepackage[normalem]{ulem}
\usepackage{titlesec}
\numberwithin{equation}{section}
\usepackage{amsrefs}
\newcommand{\norme}[1]{\left\Vert #1\right\Vert}

\newcommand{\norm}[1]{\left\Vert #1\right\Vert}
\newcommand{\D}{\mathrm{d}}
\newcommand{\I}{\mathrm{i}}

\theoremstyle{plain}
\newtheorem{theorem}{Theorem}[section]
\newtheorem{lemma}[theorem]{Lemma}
\newtheorem{corollary}[theorem]{Corollary}
\newtheorem{proposition}[theorem]{Proposition}
\theoremstyle{definition}
\newtheorem{remark}[theorem]{Remark}

\newtheorem{definition}[theorem]{Definition}
\usepackage{dsfont}
\usepackage{bm}

\titleformat{\paragraph}[runin]{\itshape\normalsize}{\theparagraph}{}{}
\titleformat{\subparagraph}[runin]{\itshape\normalsize}{\theparagraph}{}{}
\titleformat{\section}[block]{\normalfont\filcenter}{\Large\bf\thesection .}{.5em}{\Large\bf}
\titleformat{\subsection}[block]{\normalfont\bfseries}{\large\bf\thesubsection  .}{.5em}{\large \bf}

\def \equi#1{\mathrel{\mathop{\kern 0pt\sim}\limits_{#1}}}

\begin{document}
\addtocontents{toc}{\protect\setcounter{tocdepth}{1}}
\setcounter{secnumdepth}{3}

\title{\textbf{Eigenvectors distribution and quantum unique ergodicity for deformed Wigner matrices}}
\author{L. Benigni\\\vspace{-0.15cm}\footnotesize{\it{LPSM, Université Paris Diderot}}\\\footnotesize{\it{lbenigni@lpsm.paris}}}
\date{}
\maketitle
\abstract 
\small{We analyze the distribution of eigenvectors for mesoscopic, mean-field perturbations of diagonal matrices, in the bulk of the spectrum. Our results apply to a generalized $N\times N$ Rosenzweig--Porter model. We prove that the eigenvector entries are asymptotically Gaussian with a specific variance. 
 For a well spread initial spectrum, this variance profile universally follows a heavy-tailed Cauchy distribution. In the case of smooth entries, we also obtain a strong form of quantum unique ergodicity in the form of a strong concentration inequality for the mass of eigenvectors on a given set of coordinates. The proof relies on a priori local laws for this model as given in \cites{lee2016bulk,landon2017convergence,bourgade2017eigenvector}, and the eigenvector moment flow from \cites{bourgade2017eigenvector, bourgade2018random}.} 
 
 \vspace{1ex}
 \renewcommand{\abstract}{\begin{center} \textbf{R\'esum\'e}\end{center}}
 \begin{abstract}
 Nous analysons la distribution des vecteurs propres de perturbations m\'esoscopiques de matrices diagonales \`a l'int\'erieur du spectre. Nos r\'esultats s'appliquent a un mod\`ele g\'en\'eralis\'e de Rosenzweig--Porter. Nous prouvons que les entr\'ees des vecteurs propres sont asymptotiquement gaussiennes avec une variance explicite. Lorsque le spectre initial est bien \'etal\'e, ce profil de variance suit de mani\`ere universelle une distribution de Cauchy \`a queue lourde. Lorsque les entr\'ees sont lisses, nous obtenons aussi une forme forte d'unique ergodicit\'e quantique sous la forme d'une in\'egalit\'e de concentration sur la masse des vecteurs propres sur un domaine fix\'e de coordonn\'ees. La preuve se base sur des lois locales a priori donn\'ees dans \cites{lee2016bulk,landon2017convergence,bourgade2017eigenvector} et le flot des moments des vecteurs propres de \cites{bourgade2017eigenvector, bourgade2018random}
 \end{abstract}
 
 
 \vspace{1ex}
 \noindent \textbf{MSC}: Primary 60B20; secondary 58J51
{\hypersetup{linkcolor=black}\tableofcontents}
\section{Introduction}
\indent{}\normalsize In the study of large interacting quantum systems, Wigner conjectured that empirical results on energy levels are well approximated by statistics of eigenvalues of large random matrices. We refer to \cite{mehta2004random} for an overview and the mathematical formalization of the conjecture. This vision has not been shown for realistic correlated quantum systems but is regarded to hold for numerous models. For instance,  the  Bohigas$-$Giannoni$-$Schmit conjecture in quantum chaos \cite{bohigas1984characterization} connects 
eigenvalues distributions in the semiclassical limit to the Gaudin distribution for GOE statistics. These statistics also conjecturally appear for random Schr\"{o}dinger operators \cite{Anderson} in the delocalized phase.
Most of these hypotheses are unfortunately far from being proved with mathematical rigor. It is, however, possible to study systems given by large random matrices. One of the most important models of this type is the Wigner ensemble, random Hermitian or symmetric matrices whose elements are, up to the symmetry, independent and identically distributed zero-mean unit variance random variables. For this ensemble, local statistics of the spectrum only depend on the symmetry class and not on the laws of the elements (see \cites{erdos2011universality,Tao2011random,erdos2012bulk,erdos2015gap,bourgade2016fixed}). The Wigner$-$Dyson$-$Mehta conjecture was solved for numerous, more general mean-field models such as the generalized Wigner matrices, random matrices for which the laws of the matrix elements can have distinct variances (see \cite{erdos2017dynamical} and references therein).\\
\indent The statistics of eigenvectors were not present in Wigner's original study but localization, or delocalization, has been broadly studied in random matrix theory. For Wigner matrices, it has been shown in \cite{erdos2009local} that eigenvectors are completely delocalized in the following sense: denoting $u_1,\dots,u_N$ the $L^2$-normalized eigenvectors of an $N\times N$ Wigner matrix, we have with very high probability,
$$\sup_{\alpha}\vert u_i(\alpha)\vert\leqslant \frac{C(\log N)^{C}}{\sqrt{N}}.$$
Thus, eigenvectors cannot concentrate onto a set of size smaller than $N(\log N)^{-C}.$ See also \cite{vu2015random} for optimal bounds in some cases of Wigner matrices or \cite{erdos2012rigidity} for similar estimates for generalized Wigner matrices and  \cite{rudelson2015delocalization} for an improved bound which also holds for non-Hermitian matrices.  In the GOE and GUE cases, the distribution of the matrix is orthogonally invariant and eigenvectors are distributed according to the Haar measure on the orthogonal group. In particular, the entries of bulk eigenvectors are asymptotically normal:
\[
\sqrt{N}u_i(\alpha)\xrightarrow[N\rightarrow\infty]{}\mathcal{N},
\]
where $\mathcal{N}$ is a standard Gaussian random variable. Asymptotic normality was first proved for Wigner matrices in \cites{knowles2013eigenvector, Tao2012random} under a matching condition on the first four moments of the entries using comparison theorems introduced in \cites{Tao2011random, erdos2012bulk}. These conditions were later removed in \cite{bourgade2017eigenvector} where asymptotic normality holds for generalized Wigner matrices. Beyond mean-field models, conjectures of interest, for example for band matrices, are still yet to be proved. A sharp transition is conjectured to occur when the band width $W$ cross the critical value $\sqrt{N}$. For $W\ll\sqrt{N}$, eigenvectors are expected to be exponentially localized on $\mathcal{O}(W^2)$ sites and eigenvalue statistics are Poisson, while for $W\gg\sqrt{N}$ eigenvectors would be completely delocalized and one would get Wigner-Dyson-Mehta statistics for the eigenvalues. For the most recent works on this subject see \cites{schenker2009eigenvector, peled2019wegner} for localization results, \cite{bourgade2018random} for delocalization results, \cite{sodin2010spectral} for another transition occurring at the edge of the spectrum and \cite{bourgade2018band} for a recent review on the subject.  

\paragraph{}In this paper, we consider a generalized Rosensweig-Porter model, of  mean-field type, which also interpolates between delocalized and localized (or partially delocalized) phases, but always with GOE/GUE statistics. It is defined as a perturbation of a potential, consisting of a deterministic diagonal matrix, by a mean field noise, given by a Wigner random matrix, scaled by a parameter $t$. This model follows two distinct phase transitions. When $t\ll 1/N$, eigenvalue statistics coincide with $t=0$ and eigenvectors are localized on $\mathcal{O}(1)$ sites \cite{von2017local}, while when $t\gg 1$, local statistics fall in the Wigner-Dyson-Mehta universality class \cite{lee2016bulk} with fully delocalized eigenvectors \cite{lee2013local}. For $1/N\ll t\ll 1$, it has been shown in \cite{landon2017convergence} that eigenvalue statistics are in the Wigner-Dyson-Mehta universality class and in \cite{von2017non} that eigenvectors are not completely delocalized when the noise is Gaussian. In this intermediate phase, also called the \emph{bad metal regime} (see \cite{facoetti2016non} or \cite{ossipov2015eig} for instance), eigenstates are partially delocalized over $Nt$ sites, a diverging number as $N$ grows but a vanishing fraction of the eigenvector coordinates. The existence of this regime for more intricate models is only conjectured or even debated in the physics literature though progress has been made recently, for instance for the Anderson model on the Bethe lattice and regular graph in \cite{kravtsov2017non}. \\
\indent Our results give the asymptotic distribution of the eigenvectors for the Rosensweig-Porter model with a complete understanding of the bad metal regime. We show that bulk eigenvectors are asymptotically Gaussian with a specific, explicit variance depending on the initial potential, the parameter $t$ and  the position in the spectrum. For a well-spread initial condition, this variance is heavy-tailed and follows a Cauchy distribution. This shape was first unearthed in a non-rigorous way  in \cites{allez2014eigenvectors, allez2014eigenvector} for $W$ a matrix from the Gaussian ensembles, where the  
Gaussian distribution of eigenvectors (Corollary \ref{cor:Gaussian}) was conjectured. Note that eigenvector dynamics was also considered in \cite{allez2012eigenvector} and used for denoising matrices in \cites{bun2016rotational, bun2018overlaps}. In the case of Gaussian entries, the eigenvector distribution has been exhibited in the physics literature in \cite{facoetti2016non} using the resolvent flow and in \cite{ossipov2015eig} using supersymmetry techniques.\\
\indent Another strong form of delocalization of eigenfunctions is quantum ergodicity. It has been proved for the Laplace-Beltrami operator for a wide class of manifolds by Shnirel'man \cite{shnirel1974ergodic}, Colin de Verdière \cite{colin1985ergodicite} and Zelditch \cite{zelditch1987uniform} but also for regular graphs by Ananthamaran-Le Masson \cite{anantharaman2015quantum}. In \cite{rudnick1994behaviour}, Rudnick-Sarnack conjectured a stronger form of delocalization for eigenfunctions of the Laplacian called the quantum unique ergodicity. More precisely, denote $(\phi_k)_{k\geqslant 1}$ the eigenfunctions of the Laplace operator on any negatively curved compact Riemannian manifold $\mathcal{M}$, they then supposedly become equidistributed with respect to the volume measure $\mu$ in the following sense: for any open set $A\subset\mathcal{M}$
$$\int_A\vert\phi_k\vert^2\mathrm{d}\mu\xrightarrow[k\rightarrow\infty]{}\int_A \mathrm{d}\mu.$$
This convergence has been established for arithmetic surfaces (see \cites{holowinsky2012sieving, holowinsky2010mass, lindenstrauss2006invariant}).\\
\indent A probabilistic form of quantum unique ergodicity exists for eigenvectors of large random matrices. It first appeared in \cite{bourgade2017eigenvector} for generalized Wigner matrices. It is stated as a high-probability bound showing that eigenvectors are asymptotically flat in the following way: let $(u_k)_{1\leqslant k \leqslant N}$ be the eigenvectors of a $N\times N$ generalized Wigner matrix, then for any $k\in[\![1,N]\!]$, for any deterministic $N$-dependent set $I\in[\![1,N]\!]$ such that $\vert I\vert\rightarrow +\infty$ and any $\delta>0$,
$$\mathds{P}\left(
	\frac{N}{\vert I\vert}
	\left\vert
		\sum_{\alpha \in I}			
		\left(
			u_k(\alpha)^2-\frac{1}{N}
		\right)
	\right\vert
	>\delta
\right)
\leqslant \frac{N^{-\varepsilon}}{\delta^2}$$
for some $\varepsilon>0$ using the Bienaymé–Chebyshev inequality. 
Similar high-probability bounds were proved for different models of random matrices such as $d$-regular random graphs in \cite{bauerschmidt2016local}, or band matrices in \cites{bourgade2017universality, bourgade2018random}. In these last papers on band matrices, it was seen that quantum unique ergodicity is a useful property to study non mean-field models. In \cite{bourgade2018random}, a stronger form of probabilistic quantum unique ergodicity has been actually proved, showing that the eigenvectors mass is asymptotically flat with overwhelming probability (the probability decreases faster than any polynomial). Our result adapts the method introduced in \cite{bourgade2018random} to show a strong deformed quantum unique ergodicity for eigenvectors of a class of deformed Wigner matrices. Indeed, the probability mass is not flat but concentrates onto an explicit and deterministic profile with a quantitative error.

\paragraph{}The key ingredient for this analysis is the Bourgade-Yau eigenvector moment flow  \cite{bourgade2017eigenvector}, a multi-particle random walk in a random environment given by the trajectories of the eigenvalues.
This method was used for generalized Wigner matrices  \cite{bourgade2017eigenvector} and sparse random graphs \cite{bourgade2017huang}, and both settings correspond to equilibrium or close to equilibrium situations. 
Our main contribution consists in treating the non-equilibrium case, which implies additional difficulties made explicit in the next section.

\subsection{Main Results}
Consider a deterministic diagonal matrix $D=\mathrm{diag}(D_1,\dots,D_N)$. The eigenvalues (or diagonal entries) need to be regular enough on a window of size $r$ in the following way first defined in \cite{landon2017convergence}.

\begin{definition}\label{def:init}
Let $\eta_\star$ and $r$ be two $N$-dependent parameters satisfying
$$N^{-1}\leqslant \eta_\star \leqslant N^{-\varepsilon'},\quad N^{\varepsilon'}\eta_\star\leqslant r\leqslant N^{-\varepsilon'}$$
for some $\varepsilon'>0.$ A deterministic diagonal matrix $D$ is said to be $(\eta_\star,r)$-regular at $E_0$ if there exists $c_D>0$ and $C_D>0$ independent of $N$ such that for any $E\in[E_0-r,E_0+r]$ and $\eta_\star\leqslant\eta\leqslant 10$, we have
$$c_D\leqslant \mathrm{Im}\,m_D(E+\mathrm{i}\eta)\leqslant C_D,$$
where $m_D$ is the Stieltjes transform of $D$:
$$m_D(z)=\frac{1}{N}\sum_{k=1}^N\frac{1}{D_k-z}.$$
\end{definition}
We want to study the perturbation of such a diagonal matrix, notably the eigenvectors, by a mesoscopic Wigner random matrix. We will now suppose that $D$ is $(\eta_\star,\,r)$-regular at $E_0$, a fixed energy point. Letting $0<\kappa<1$, we will denote in the rest of the paper the spectral window as 
\begin{equation*}
\mathcal{I}^\kappa_{r}=[E_0-(1-\kappa)r,E_0+(1-\kappa) r].
\end{equation*} 
We remove a certain window of energy to avoid any possible complications at the edge.
We will also use the following domains: the first domain is used for the size of our deformation while the other is the spectral domain in which we perform our analysis. We need the perturbation to be mesoscopic but smaller than the energy window size $r$, define then for any small positive $\omega,$
$$\mathcal{T}_\omega=\left[N^{\omega}\eta_\star,N^{-\omega}r\right].$$
For the spectral domain, take first $t\in\mathcal{T}_\omega$ and note that we will consider only $\mathrm{Im}(z):=\eta$ smaller than $t$ but most results such as local laws holds up to macroscopic $\eta$. Let $\vartheta>0$ be an arbitrarily small constant and
$$\mathcal{D}_{r}^{\vartheta,\kappa}=\left\{z=E+\mathrm{i}\eta:\, E\in\mathcal{I}_{r}^\kappa,\, N^{\vartheta}/N\leqslant \eta\leqslant N^{-\vartheta}t\right\}.$$

Hereafter is our assumptions on our Wigner matrix.
\begin{definition}\label{def:wignernonsmooth}
A Wigner matrix $W$ is a $N\times N$ Hermitian/symmetric matrix satisfying the following conditions
\begin{enumerate}
\item[(i)]The entries $(W_{i,j})_{1\leqslant i\leqslant j\leqslant N}$ are independent.
\item[(ii)]For all $i,j$, $\mathds{E}[W_{i,j}]=0$ and $\mathds{E}[\vert W_{ij}\vert^2]=N^{-1}.$
\item[(iii)]For every $p\in\mathbb{N}$, there exists a constant $C_p$ such that $\norm{\sqrt{N}W_{ij}}_p\leqslant C_p$.
\end{enumerate}
\end{definition}
Let $W$ be a Wigner matrix and define the following $t$-dependent matrix for $t\in\mathcal{T}_\omega$
\begin{equation}\label{def:wigner}
W_t=D+\sqrt{t}W.
\end{equation}
The eigenvectors of $D$ are exactly the vectors of the canonical basis since the matrix is diagonal. However, if $t$ were of order one instead of being in $\mathcal{T}_\omega$, the local statistics of $W_t$ would become universal and would be given by local statistics from the Gaussian ensemble. In particular, the eigenvectors would be completely delocalized \cite{lee2016bulk}. Our model consists in looking at the diffusion of the eigenvectors on the canonical basis after a mesoscopic perturbation. Our main result is that the coordinates of bulk eigenvectors are time and position dependent Gaussian random variables.  Before stating our result, we first define the asymptotic distribution of the eigenvalues of the matrix $W_t$ which is the free convolution of the semicircle law (coming from $W$) and the empirical distribution of $D$. We will define this distribution through its Stieltjes transform $m_t(z)$ as the solution to the following self-consistent equation

\begin{equation}
m_t(z)=\frac{1}{N}\sum_{\alpha=1}^N\frac{1}{D_\alpha-z-tm_t(z)}.
\end{equation} 
It is known that this equation has a unique solution with positive imaginary part and is the Stieltjes transform of a measure with density denoted by $\rho_t$ (see \cite{biane1997free} for more details).  Define the quantiles $\left(\gamma_{i,t}\right)_{0\leqslant i\leqslant N}$ of this measure by
$$\int_{-\infty}^{\gamma_{i,t}}\rho_t(x)\mathrm{d}x=\frac{i}{N}.$$
We can also now define the set of indices in the spectral window corresponding to the indices such that the corresponding classical locations lies in the energy window $\mathcal{I}_r^{\kappa},$
\[
\mathcal{A}^\kappa_{r}=\left\{i\in[\![1,N]\!],\gamma_{i,t}\in \mathcal{I}^\kappa_{r}\right\}.
\]
We can now state our main results, denoting  $u_1(t),\dots,u_N(t)$ the $L^2$-normalized eigenvectors of $W_t$ (we will often omit the $t$-dependence for $\bm{u}$). We will define the following quantity, for $N^{-1}\ll \eta\ll t$ and $\eta_\star\ll t\ll r$:
\begin{equation}\label{defsigma}
\sigma_t^2(\mathbf{q},k,\eta)
=
\sum_{\alpha=1}^N\frac{q_\alpha^2t}{(D_\alpha-\gamma_{k_i,t}-t\mathrm{Re}\,m_t(\gamma_{k_i,t}+\I\eta))^2+\left(t\mathrm{Im}\,m_t(\gamma_{k_i,t}+\I\eta)\right)^2}.
\end{equation}
It is the asymptotic deterministic variance of our eigenvector projections. By taking $\mathbf{q}$ to be a canonical basis vector $\mathbf{e}_\alpha$, we see that $u_k(\alpha)$ has a variance of the form
$$\frac{1}{N}\frac{t}{(D_\alpha-\gamma_{k_i,t}-t\mathrm{Re}\,m_t(\gamma_{k_i,t}))^2+(t\mathrm{Im}\,m_t(\gamma_{k_i,t}))^2}.$$
For regularly spaced $D_\alpha$'s, this is heavy-tailed with Cauchy shape $\frac{t}{N(x^2+t^2)}$. It localizes the entries onto a subset of indices of size $Nt\ll N$: a fraction of the eigenvector coordinates vanishing as $N$ grows. Such a partial localization appears in \cite{von2017non} for $W$ GOE-distributed. 
\begin{theorem}\label{theo:resultGaus}(Gaussianity of bulk eigenvectors)
Fix $\kappa\in(0,1),$ $\omega\in (0,\varepsilon^\prime/10)$ where $\varepsilon^\prime$ is as in Definition \ref{def:init} and $m\in\mathbb{N}$. Let $t\in\mathcal{T}_\omega$ and  $I\subset \mathcal{A}^\kappa_{r}$ be a deterministic ($N$-dependent) set of $m$ elements. Let $W$ as in Definition \ref{def:wignernonsmooth} and $W_t$ as in \eqref{def:wigner}.  Write $I=\{k_1,\dots,k_m\}$, take a deterministic $\mathbf{q}\in\mathbb{R}^N$ such that $\norme{\mathbf{q}}_2=1$, and define for $i\in[\![1,m]\!],$ 
\begin{equation}
\sigma_t^2(\mathbf{q},k_i):=\lim_{\eta\downarrow 0} \sigma_t^2(\mathbf{q},k_i,\eta)
\end{equation}
Then we have
\begin{align}
\label{results}\left(\sqrt{\frac{N}{\sigma_t^2(\mathbf{q},k_i)}}\vert \langle \mathbf{q},u_{k_i}\rangle\vert\right)_{i=1}^m&\xrightarrow[N\rightarrow\infty]{}(\vert\mathcal{N}_i\vert)_{i=1}^m&\text{in the symmetric case,}\\
\label{resulth}\left(\sqrt{\frac{2N}{\sigma_t^2(\mathbf{q},k_i)}}\vert \langle \mathbf{q},u_{k_i}\rangle\vert\right)_{i=1}^m&\xrightarrow[N\rightarrow\infty]{}(\vert\mathcal{N}_i^{(1)}+i\mathcal{N}_i^{(2)}\vert)_{i=1}^m&\text{in the Hermitian case}
\end{align}
in the sense of convergence of moments, where all $\mathcal{N}_{i},$ $\mathcal{N}_{i}^{(1)}$ and $\mathcal{N}_{i}^{(2)}$ are independent Gaussian random variables with variance 1. The convergence is uniform in over the choice of sets $I\subset[\![1,N]\!]$ of size $m$.
\end{theorem}
One can deduce joint weak convergence of eigenvector entries from the previous convergence of moments because  $\mathbf{q}$ is arbitrary in $\mathbb{S}^{N-1}$ (see \cite{bourgade2017eigenvector}*{Section 5.3}). However, since the eigenvectors are defined up to a phase, we first need to define the following  equivalence relation: $u\sim v$ if and only if $u=\pm v$ in the symmetric case and $u=e^{\mathrm{i}\omega}v$ for some $\omega\in\mathbb{R}$ in the Hermitian case.

\begin{corollary}\label{cor:Gaussian}
Let $\kappa\in(0,1)$ and $m\in\mathbb{N}$, let $W$ as in Definition \ref{def:wignernonsmooth} and $W_t$ as in Definition \ref{def:wigner}.  Then for any deterministic  $k\in\mathcal{A}^\kappa_{r}$ and $J\subset [\![1,N]\!]$ such that $\vert J\vert =m$ we have
\begin{align}
\left(\sqrt{\frac{N}{\sigma_t^2(\mathbf{e_\alpha},k)}}u_k(\alpha)\right)_{\alpha\in J}&\xrightarrow[N\rightarrow \infty]{}\left(\mathcal{N}_i\right)_{i=1}^m&\text{in the symmetric case},\\
\left(\sqrt{\frac{2N}{\sigma_t^2(\mathbf{e_\alpha},k)^2}}u_k(\alpha)\right)_{\alpha\in J}&\xrightarrow[N\rightarrow \infty]{}\left(\mathcal{N}^{(1)}_i+\mathrm{i}\mathcal{N}^{(2)}_i\right)_{i=1}^m&\text{in the Hermitian case}
\end{align}
in the sense of convergence of moments modulo $\sim$, where all $\mathcal{N}_{i},$ $\mathcal{N}_{i}^{(1)}$ and $\mathcal{N}_{i}^{(2)}$ are independent Gaussian random variables with variance 1. In more precise terms, for any polynomial $P$ in $m$ variables there exists $\delta$ depending on $P$ such that, for $N$ large enough,
\[
\sup_{\substack{J\subset [\![1,N]\!],\,\vert J\vert=m\\ k\in \mathcal{A}^\kappa_r}}
\left\vert
	\mathds{E}\left[
		P\left(
			\left(
				\epsilon\sqrt{\frac{N}{\sigma_t^2(\mathbf{e_\alpha},k)}}
				u_k(\alpha)
			\right)_{\alpha\in J}
		\right)
	\right]
	-
	\mathds{E}\left[
		P\left(
			\left(
			\mathcal{N}_j
			\right)_{j=1}^m
		\right)
	\right]
\right\vert
\leqslant
N^{-\delta}
\]
where $\epsilon$ is taken uniformly at random in the set $\{-1,1\}.$ The convergence in the Hermitian case is similar by taking $\epsilon$ uniform on the circle.
\end{corollary}

This result states that the entries of bulk eigenvectors are asymptotically independent Gaussian random variables with variance ${\sigma_t}^2$ which
answers a conjecture from \cite{allez2014eigenvectors}*{Section 3.2}, stated in the more restrictive case where $W$ is GOE. The asymptotic normality of the eigenvectors gives the following weak form of quantum unique ergodicity.

\begin{corollary}\label{QUE}(Weak Quantum Unique Ergodicity)
Let $W$ as in \ref{def:wignernonsmooth} and $W_t$ as in Definition \eqref{def:wigner}. There exists $\vartheta>0$ such that for any $c>0$, there exists $C>0$ such that the following holds : for any $I\subset [\![1,N]\!]$ and $k\in \mathcal{A}^\kappa_{r}$, we have
\begin{equation}
\mathds{P}\left(\frac{Nt}{\vert I\vert}\left\vert\sum_{\alpha\in I}\vert u_k(\alpha)\vert^2-\frac{1}{N}\sum_{\alpha\in I}\sigma_t^2(\mathbf{e}_\alpha,k)\right\vert >c\right)\leqslant C(N^{-\vartheta} + \vert I\vert^{-1}).
\end{equation} 
\end{corollary} 

This high probability bound is not the strongest form of quantum unique ergodicity one can obtain for random matrices. Indeed, if we consider the Gaussian ensembles for which the eigenbasis is Haar-distributed on the orthogonal group and each eigenvectors is uniformly distributed on the sphere, one can get that  for any $\varepsilon$ and $D$ positive constants, for any $1\leqslant k\leqslant N$,
\[
\mathds{P}\left(
	\left\vert 
		\sum_{\alpha\in I}\vert u_k(\alpha)\vert^2-\frac{\vert I\vert}{N}
	\right\vert
	\geqslant N^\varepsilon\frac{\sqrt{I}}{N}
\right)\leqslant N^{-D}\quad\text{for any }N\text{ sufficiently large}
\]

In this paper, we will obtain a similar overwhelming probability bound on the probability mass of a single eigenvector with an explicit error for a more restrictive model of matrices: deformed random matrix with smooth entries given by the following definition or from a Gaussian divisible ensemble.

\begin{definition}\label{def:smoothwig}
A smooth Wigner matrix $W$ is a $N\times N$ Hermitian/symmetric matrix with the following conditions
\begin{itemize}
\item[(i)] The matrix entries $(W_{ij})_{1\leqslant i\leqslant j\leqslant N}$ are independent and identically distributed random variables following the distribution $N^{-1/2}\nu$ where $\nu$ has mean zero and variance $1$.
\item[(ii)]The distribution $\nu$ has a positive density $\nu(x)=e^{-\Theta(x)}$ such that for any $j$, there are constants $C_0$ and $C_1$ such that
\begin{equation}\label{eq:regucond}
\vert\Theta^{(j)}(x)\vert\leqslant C_0(1+x^2)^{C_1}
\end{equation} 
\item[(iii)] The tail of the distribution $\nu$ has a subexponential decay. In other words, there exists $C$ and $q$ two positive constants such that
\begin{equation}\label{eq:subexp}
\int_\mathbb{R}\mathds{1}_{\vert x\vert\geqslant y}\mathrm{d}\nu(x)\leqslant C\exp(-y^q)
\end{equation}
\end{itemize}
\end{definition}

We need the smoothness assumptions on $W$ in order to use the reverse heat flow techniques from \cites{erdos2010peche, erdos2011universality}. Indeed, our result is an overwhelming probability bound on the eigenvectors of $W_t$. 
We think, however, that this property holds for a larger matrix ensembles and that the smoothness property is simply technical.

In the following, since the eigenvectors are concentrated on $Nt$ sites, it is relevant to define the following notation for any set (which can be $N$-dependent) $A$, denote
$$\widehat{A}=\frac{\vert A\vert}{Nt}\wedge 1.$$
Indeed, having errors involving $\widehat{A}$ allows us to get bounds improving for $\vert A\vert\leqslant Nt$ but still holding for $\vert A\vert\gg Nt$. 
\begin{theorem}\label{theo:resultQUE}
Let $\kappa\in(0,1)$, $\omega$ a small positive constant, for instance $\omega<\varepsilon'/10$. Let $t\in\mathcal{T}_\omega$, $I\subset [\![1,N]\!]$ be a deterministic ($N$-dependent) set, $W$ as in Definition \ref{def:smoothwig} and $W_t$ as in \eqref{def:wigner}.
Define now
\begin{equation}\label{eq:defxi}
\Xi=\frac{\widehat{I}}{(Nt)^{1/3}}\quad\text{and}\quad\sigma_t^2(\alpha,k):=\sigma_t^2(\alpha,k,\eta_0)\quad\text{with}\quad N\eta_0=\frac{\widehat{I}^{\,\,2}}{\Xi^2}.
\end{equation}
Then we have, for any $\varepsilon>0$ (small) and $D>0$ (large) and for $k,\,\ell\in\mathcal{A}_{r}^\kappa$ with $k\neq\ell$, in the symmetric case
\begin{align}\label{eq:resultQUE}
\mathds{P}\left(\left\vert\sum_{\alpha\in I}\left(u_k(\alpha)^2-\frac{1}{N}\sigma_t^2(\alpha,k)\right)\right\vert+\left\vert \sum_{\alpha\in I}u_k(\alpha){u}_\ell(\alpha)\right\vert\geqslant N^\varepsilon \Xi\right)\leqslant N^{-D}
\end{align}
and in the Hermitian case,
\begin{align}
\mathds{P}\left(\left\vert\sum_{\alpha\in I}\left(\vert u_k(\alpha)\vert^2-\frac{1}{2N}\sigma_t^2(\alpha,k)\right)\right\vert+\left\vert \sum_{\alpha\in I}u_k(\alpha)\bar{u}_\ell(\alpha)\right\vert\geqslant N^\varepsilon \Xi\right)\leqslant N^{-D}.
\end{align}
\end{theorem}
\begin{remark}The choice of $\eta_0$ depends on our proof and is the one we should take to optimize our error $\Xi$. However, this error and the choice of $\eta_0$ do not seem optimal since we actually expect to have some form of Gausian fluctuations around this deterministic profile.
\end{remark}
\subsection{Method of Proof}
Our proof is based on the three-step strategy from \cites{erdos2010peche,erdos2011universality}  (see \cite{erdos2017dynamical} for a recent book presenting this method). The first step is to have an optimal, local control of the spectral elements of the matrix ensemble given by a local law on the resolvent. The second step is to obtain the wanted result for a relaxation of the model by a small Gaussian perturbation. Finally, the third and last step consists of removing this Gaussian part. We will give the proof of Theorem \ref{theo:resultGaus}, Corollary \ref{QUE}, Theorem \ref{theo:resultQUE} only in the symmetric case and refer the reader to \cites{bourgade2017eigenvector, bourgade2018random} for the tools needed in the Hermitian case.
\paragraph{First step: local laws for our model.} In \cite{landon2017convergence}, Landon-Yau showed a local law for the Dyson Brownian motion with a diagonal initial condition at all times. This result gives us an averaged local law on the Stieltjes transform but also an entrywise anisotropic local law for the resolvent. Since we want to consider any projection of the eigenvectors, we will also need a local law on the quadratic form $\langle \mathbf{q},G(z)\mathbf{q}\rangle$. This control of the resolvent for mesoscopic perturbation has been showed in \cite{bourgade2017huang}. Note that these results were done in the Gaussian case but can easily be generalized to the Wigner case with the right assumptions on moments. 
\paragraph{Second step: short time relaxation.}The second step consists of perturbing  $W_t$ by a small Gaussian component. We will obtain this perturbed model by making $W_t$ undergo the Dyson Brownian motion given by the following definition. 
\begin{definition}Here is our choice of Dyson Brownian motion.\newline
Let $B$ be a $N\times N$ symmetric matrix such that $B_{ij}$ for $i<j$ and $B_{ii}/\sqrt{2}$ are independent standard brownian motions. The $N\times N$ symmetric Dyson Brownian motion with initial condition $H_0$ is defined as 
\begin{equation}\label{dyson1}
H_s=H_0+\frac{1}{\sqrt{N}}B_s.
\end{equation}

\end{definition}

We also give the dynamics followed by the eigenvalues and the eigenvectors of such matrices.

\begin{definition}\label{def:dysondyn}
Let $\bm{\lambda_0}$ be in the simplex $\Sigma_N=\{\lambda_1<\dots<\lambda_N\}$, $\bm{u_0}$ be an orthogonal $N\times N$ matrix, and $B$ as in \eqref{dyson1}. Consider the dynamics
\begin{align}
\label{eq:dysoneigval}
\mathrm{d}\lambda_k&=\frac{\mathrm{d}B_{kk}}{\sqrt{N}}+\frac{1}{N}\sum_{\ell\neq k}\frac{\mathrm{d}s}{\lambda_k-\lambda_\ell},\\
\mathrm{d}u_k&=\frac{1}{\sqrt{N}}\sum_{\ell\neq k}\frac{\mathrm{d}B_{k\ell}}{\lambda_k-\lambda_\ell}u_\ell-\frac{1}{2N}\sum_{\ell\neq k}\frac{\mathrm{d}s}{(\lambda_k-\lambda_\ell)^2}u_k\label{dysonvect}
\end{align}
with initial condition $(\bm{\lambda_0},\bm{u_0})$,

\end{definition}

This eigenvector flow was first computed in different contexts such as \cite{anderson2010introduction} for GOE or GUE matrices, \cite{bru1989diffusions} for real Wishart processes and \cite{norris1986brownian} for Brownian motion on ellipsoids.
\begin{remark}
 If $\bm{\lambda_0}$ and $\bm{u_0}$ are the eigenvalues and eigenvectors of a fixed matrix $H_0$, then the solution to the dynamics from Definition \ref{def:dysondyn} have the same distribution for any time $s$ as  the eigenvalues and eigenvectors of
\begin{align*}
H_s=H_0+\sqrt{s}\mathrm{GOE}
\end{align*} 
with GOE being a matrix from the normalized Gaussian Orthogonal Ensemble (in the sense that its off-diagonal entries have variance $1/N$). In this paper, taking $W$ to be such a matrix in \eqref{def:wigner}, we study the eigenvectors of the Dyson Brownian motion with a diagonal initial condition after a mesoscopic time.
\end{remark}
We will then need to study the eigenvectors of $H_\tau$ for a small $N^{-1}\ll \tau\ll t$. The convergence of joint moments of eigenvectors projections will be obtained by the maximum principle technique introduced in \cite{bourgade2017eigenvector}. It is based on analyzing the dynamics followed by these moments. We will now recall notations and results on this eigenvector moment flow.\\
\indent Take $\mathbf{q}\in\mathbb{R}^N$ such that $\norme{\mathbf{q}}_2=1$ a fixed direction onto which we will project our eigenvectors. For $u_1^H,\dots,u_N^H$ the eigenvectors of the matrix \eqref{dyson1}, define
\begin{equation}\label{zk}
z_k(s)=\sqrt{N}\langle\mathbf{q},u_k^H(s)\rangle.
\end{equation}

Now for $m\in[\![1,N]\!]$, denote by $j_1,\dots,j_m$ positive integers and let $i_1,\dots,i_m$ in $[\![1,N]\!]$ be distinct indices. We will consider the following normalized polynomials
\begin{align}
\label{eq:polysym}Q^{j_1,\dots,j_m}_{i_1,\dots,i_m}&=\prod_{l=1}^mz_{i_l}^{2j_l}a(2j_l)^{-1}\quad\text{where}\quad a(n)=\prod_{k\leq n,k\, \text{odd}}k
\end{align}
Note that $a(2n)=\mathds{E}[\mathcal{N}^{2n}]$ with $\mathcal{N}$ a standard Gaussian random variables.\\ 
\indent Consider a configuration of particles $\bm{\xi}:[\![1,N]\!]\rightarrow\mathbb{N}$ where $\xi_j:=\bm{\xi}(j)$ is seen as the number of particles at the site $j$. We denote $\mathcal{N}(\bm{\xi})=\sum_j\xi_j$ the total number of particles in the configuration $\bm{\xi}.$\\
\indent Define $\bm{\xi}^{i,j}$ to be the configuration obtained by moving one particle from $i$ to $j$. If there is no particle in $i$ then $\bm{\xi}^{i,j}=\bm{\xi}$. It is clear that we can map $\{(i_1,j_1),\dots,(i_m,j_m)\}$ with distinct $i_k$'s and positive $j_k$'s summing to an $n>0$ to a configuration $\bm{\xi}$  with $\xi_{i_k}=j_k$ and $\eta_l=0$ if $l\notin \{i_1,\dots,i_m\}.$\\
Define now, given this map, 
\begin{equation}\label{flambda}
f_{\bm{\lambda},s}(\bm{\xi}):=\mathds{E}\left[{Q}^{j_1,\dots,j_m}_{i_1,\dots,i_m}\vert\bm{\lambda}\right],
\end{equation}
a $n$-th joint moment of the coordinates of $u_k$. The conditioning here is on the full path of eigenvalues from 0 to $\infty$. The next theorem gives the eigenvector moment flow that $f_{\bm{\lambda},s}$ undergoes. 

\begin{theorem}[\cite{bourgade2017eigenvector}*{Theorem 3.1}]\label{theo:eigmomflow}
Suppose that $\bm{u}$ is the solution of the symmetric Dyson vector flow \eqref{dysonvect} and $f_{\bm{\lambda},s}(\bm{\xi})$ is given by \eqref{flambda} with the polynomials $Q_s$. Then it satisfies the equation
\begin{equation}\label{discs}
\partial_sf_{\bm{\lambda},s}(\bm{\xi})=\frac{1}{N}\sum_{i\neq j}\frac{2\xi_i(1+2\xi_j)\left(f_{\bm{\lambda},s}(\bm{\xi}^{i,j})-f_{\bm{\lambda},s}(\bm{\xi})\right)}{(\lambda_i-\lambda_j)^2}.
\end{equation}

\end{theorem}
\begin{center}
\begin{figure}[!ht]
\hspace{15.5em}\includegraphics[scale=.5]{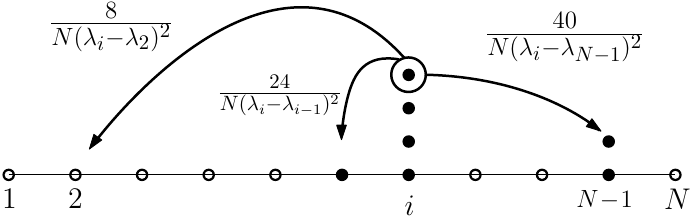}
\caption{Example of the symmetric eigenvector moment flow with a configuration of 7 particles.}
\end{figure}
\end{center}

Now that we have the expression of the eigenvector moment flow, we can give an heuristic for the apparition of a Cauchy profile in the variance \eqref{defsigma}. Indeed, the single particle case $m=1$ gives us the variance of an entry of an eigenvector. To understand this result, consider the diagonal entries of the matrix $D$ to be the quantiles of the semicircle law for instance, it is then interesting to consider the following continuous dynamics, define the operator $K$ acting on smooth functions on $[-2,2]$ as
\begin{equation}
(Kf)(x)=\int_{-2}^2\frac{f(x)-f(y)}{(x-y)^2}\mathrm{d}\rho(y).
\end{equation}
The differential equation $\partial_tf=Kf$ can be seen as a deterministic and continuous equivalent of \eqref{discs} because of the rigidity property of the Dyson Browian motion eigenvalues.  We then get the following lemma from \cite{bourgade2016fixed}

\begin{lemma}[\cite{bourgade2016fixed}]
Let $f$ be smooth with all derivatives uniformly bounded. For any $x,y\in (-2,2),$ denote $x=2\cos\theta,\,\,y=2\cos{\phi}$ with $\theta,\phi\in(0,\pi)$. Then
\begin{equation}
\left(e^{-tK}f\right)(x)=\int p_t(x,y)f(y)\mathrm{d}\rho(y)
\end{equation}

where the kernel is given by
\begin{equation}
p_t(x,y):=\frac{1-e^{-t}}{\left\vert e^{i(\theta+\phi)}-e^{-t/2}\right\vert^2\left\vert e^{i(\theta-\phi)}-e^{-t/2}\right\vert^2}.
\end{equation}

\end{lemma}

Now at our small time-scale, we have
$$p_t(x,y)\asymp \frac{t}{(x-y)^2+t^2}.$$
Hence \eqref{theo:resultGaus} where $m=1$ can be considered as a result of stochastic homogenization in a non-equilibrium setting when we consider the dynamics \eqref{discs} in the bulk.
\paragraph{}For Theorem \ref{theo:resultQUE}, we will study another observable which follows the same dynamics as in Theorem \ref{theo:eigmomflow}. This new observable has been analyzed in \cite{bourgade2018random} to obtain universality for a class of band matrices.
Define now the centered eigenvectors overlaps for symmetric matrices,
\begin{eqnarray}
&\displaystyle{p_{ij}=\sum_{\alpha\in I}u_i(\alpha)u_j(\alpha),\quad i\neq j\in I},\\
&\displaystyle{p_{ii}=\sum_{\alpha\in I} u_i(\alpha)^2-C_0,\quad i\in I}
\end{eqnarray}
where $\mathbf{u}$ are the eigenvectors of $H_{s}$ and $C_0$ is any constant in the sense that it does not depend on $i$ but can depend on $N$.

Now for $\bm{\xi}$ a configuration of $n$ particles on $N$ sites, define the following set
$$\mathcal{V}_{\bm{\xi}}=\{(i,a),\,\,1\leqslant i\leqslant N,\,\, 1\leqslant a\leqslant 2\eta_i\}.$$ 
The set $\mathcal{V}$ will be a set of vertices. Consider now $\mathcal{G}_{\bm{\xi}}$ the set of perfect matchings on $\mathcal{V}_{\bm{\xi}}$. For any edge on $G$, $e=\{(i,a),(j,b)\}$, define $p(e)=p_{ij}$, $P(G)=\prod_{e\in\mathcal{E}(G)}p(e)$ and finally

\begin{equation}\label{lem:eq:perfobsdeformed}
F_{\bm{\lambda},s}(\bm{\xi})=\frac{1}{\mathcal{M}(\bm{\xi})}\mathds{E}\left[\sum_{G\in\mathcal{G}_{\bm{\xi}}}P(G)\middle\vert\bm{\lambda}\right]
\end{equation}
where $\mathcal{M}(\bm{\xi})=\prod_{i=1}^N(2\xi_i)!!$, with $(2m)!!$ being the number of perfect matchings of the complete graph on $2m$ vertices. Note that this quantity depend on the eigenvalues trajectories $\bm{\lambda}$.

\begin{figure}[H]
\centering
\begin{subfigure}[t]{.5\textwidth}
  \centering
  \includegraphics[width=.7\linewidth]{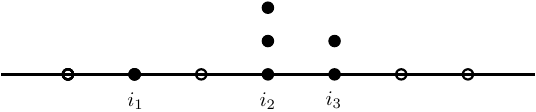}
  \caption{A configuration $\bm{\xi}$ with $\mathcal{N}(\bm{\xi})=6$ particles}
\end{subfigure}%
\hspace{-0em}\begin{subfigure}[t]{.5\linewidth}
	\centering
  \includegraphics[width=.6\linewidth]{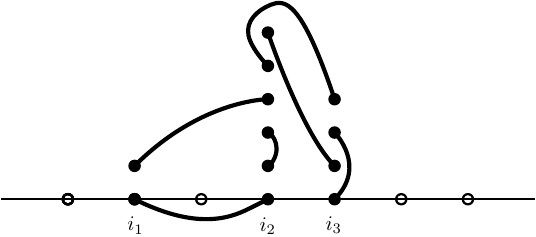}
  \captionof{figure}{\raggedright An example of a perfect matching \hbox{$G\in\mathcal{G}_{\bm{\xi}}$} with \hbox{ $P(G)=p_{i_1i_2}^2p_{i_2i_2}p_{i_2i_3}^2p_{i_3i_3}$}}
\end{subfigure}
\end{figure}

The previous quantity follows the same dynamics \eqref{discs} as $f_{\bm{\lambda},s}(\bm{\xi}).$

\begin{theorem}[\cite{bourgade2018random}*{Theorem 2.6}] Suppose that $\mathbf{u}$ is the solution of the symmetric Dyson vector flow \eqref{dysonvect} and $F_{\bm{\lambda},s}(\bm{\xi})$ is given by \eqref{lem:eq:perfobsdeformed}.  Then it satisfies the equation
\begin{equation}\label{eq:discs2deformed}
\partial_s F_{\bm{\lambda},s}(\bm{\xi})=\frac{1}{N}\sum_{i\neq j}\frac{2\xi_i(1+2\xi_j)\left(F_{\bm{\lambda},s}(\bm{\xi}^{i,j})-F_{\bm{\lambda},s}(\bm{\xi})\right)}{(\lambda_i-\lambda_j)^2}.
\end{equation}
\end{theorem}

\paragraph{Third step: invariance of local statistics.} The third and last step will be to obtain the result for the matrix $W_t$ without any Gaussian part. We will do so using a variant of the dynamical method introduced in \cite{bourgade2017eigenvector}*{Appendix A} which will show the continuity of resolvent statistics along the trajectory. This method was also used in \cite{huang2015bulk} to study sparse matrices. We will see that we need to take $\tau$ of order larger than $N^{-1}$ but smaller than $\sqrt{t/N}$ to use the continuity argument. In order to remove the Gaussian component for Theorem \ref{theo:resultQUE}, we will use the reverse heat flow and will need smoothness of the entries of our original matrix. Note that a moment matching scheme between the two matrix ensembles, which holds for any time $N^{-1}\ll \tau\ll 1$, could be used to obtain the invariance of local statistics. However it can be of later interest to obtain the continuity estimate up to time $\sqrt{t/N}$.

\paragraph{}The next section will state the local laws proved in different papers (\cite{landon2017convergence} and \cite{bourgade2017huang}). The third section is dedicated to prove Theorem \ref{theo:resultGaus}, Corollary \ref{QUE} and Theorem \ref{theo:resultQUE} for a short time relaxation of the matrix $W_t$. We use a maximum principle on $f_s$ or $F_s$, a basic tool for the analysis of parabolic equations. However, since we want a local result, remember that the variance depends on the position of the spectrum, we need to localize the maximum principle. We finish with an induction on the number of particles in the multi-particle random walk or a mutli-scale argument. For the third step, we will need a continuity result for the Dyson Brownian motion which will be shown in Subsection \ref{sec:cont} and we will give the reverse heat flow technique in Subsection \ref{sec:reverse}. We will then conclude by combining the three steps in Section \ref{sec:proof}.
\paragraph{Acknowledgments.} The author would like to kindly thank his advisor Paul Bourgade for many insightful and helpful discussions about this work and anonymous referees for many helpful comments on this article.

\section{Local laws}\label{sec:local}
In this section, we focus on the different local laws result for $W_t$. These local laws are high probability bounds, for simplicity we will now introduce the following notation for stochastic domination. For
\begin{equation*}
X=\left(X_N(u), N\in\mathbb{N}, u\in U_N\right),\quad
Y=\left(Y_N(u), N\in\mathbb{N}, u\in U_N\right).
\end{equation*}
two families of nonnegative random variables depending on $N$ (note that $U_N$ can also depend on $N$), we will say that $X$ is {stochastically dominated} by $Y$ uniformly in $\omega$, and write $X\prec Y$, if for all $\tau>0$ and $D>0$ we have
\begin{equation*}
\sup_{u\in U_N}\mathds{P}\left(X_N(u)>N^\tau Y_N(u)\right)\leqslant N^{-D}
\end{equation*}
for $N$ large enough. If we have $\vert X\vert\prec Y$ for some family $X$, we will write $X=\mathcal{O}_\prec\left(Y\right)$.

\paragraph{}Define now the resolvent of $W_t$ and its normalized trace, 
\begin{equation}
G(z)=(W_t-z)^{-1}=\sum_{k=1}^N\frac{\vert u_k\rangle\langle u_k\vert}{\lambda_k-z}
,\quad 
\varsigma_t(z)=\frac{1}{N}\mathrm{Tr}\, G(z)=\frac{1}{N}\sum_{k=1}^N\frac{1}{\lambda_k-z}
\end{equation}
and denote $G_{ij}(z)$ the $(i,j)$ entry of the resolvent matrix. In the rest of the section we will omit the dependence in $t$ of the resolvent since we are not looking at its dynamics.
\subsection{Anisotropic local law for deformed Wigner matrices}
\indent An averaged local law was proved in \cite{landon2017convergence}. The proof relies on Schur's complement formula, large deviations bounds and interlacing formula in order to first state a weak local law on the resolvent entries and the Stieltjes transform. The result then follows from a fluctuation averaging lemma in order to go from the scale $(N\eta)^{-1/2}$ to $(N\eta)^{-1}$. We first give the definition of the limiting Stieltjes transform as the solution $m_t(z)$ such that $\mathrm{Im}\, m_t(z)>0$ on the upper half plane of the following equation
\begin{equation}\label{eq:stieljes}
m_t(z)=\frac{1}{N}\sum_{\alpha=1}^N\frac{1}{D_\alpha-z-tm_t(z)}=\frac{1}{N}\sum_{\alpha=1}^N g_\alpha(t,z)
\end{equation}
where we defined
$$g_\alpha(t,z):=\frac{1}{D_\alpha-t-m_t(z)}$$

We will also need the following lemma on the Stieltjes transform claiming that its imaginary part is of order one.
\begin{lemma}[\cite{landon2017convergence}*{Lemma 7.2}]\label{lem:stieljes}
Let $\vartheta, \omega>0$  small constants and $\kappa\in(0,1)$. Take $z\in\mathcal{D}_{r}^{\vartheta,\kappa}$, for $N$ large enough, the following bounds holds for $t\in\mathcal{T}_\omega$,
$$c\leqslant\mathrm{Im}\,m_t(z)\leqslant C.$$
Moreover
$$ct\leqslant\vert D_\alpha-z-tm_t(z)\vert\leqslant C.$$
Note that the constants above do not depend on any parameter.
\end{lemma}
 
Here is the averaged local law taken from \cite{landon2017convergence}.
\begin{theorem}[ \cite{landon2017convergence}*{Theorem 3.3}]\label{local}
Let $W_t$ be as in Definition \ref{def:wigner}, $\vartheta>0$ and  $\kappa\in(0,1)$,
\begin{equation}
\vert \varsigma_t(z)-m_t(z)\vert\prec\frac{1}{N\eta}
\end{equation} 
 uniformly in $z=E+i\eta\in \mathcal{D}_{r}^{\vartheta,\kappa}$
\end{theorem} 

The proof of Theorem \ref{local} also gives the following entrywise local law from \cite{landon2017convergence} also properly stated in \cite{bourgade2017huang}.
\begin{theorem}[\cite{bourgade2017huang}*{Theorem 2.4}]\label{entrylocal}
Let $W_t$ be as in Definition \ref{def:wigner} and $\vartheta>0$, $\kappa\in(0,1)$. Uniformly in $z=E+i\eta\in \mathcal{D}_{r}^{\vartheta,\kappa},$ we have for the diagonal entries 
\begin{equation}
\left\vert G_{\alpha\alpha}(z)-g_\alpha(t,z)\right\vert\prec \frac{t}{\sqrt{N\eta}}\left\vert g_\alpha(t,z)\right\vert^2,
\end{equation}
and for the off-diagonal entries
\begin{equation}
\left\vert G_{\alpha\beta}(z)\right\vert\prec \frac{1}{\sqrt{N\eta}}\min\{\vert g_\alpha(t,z)\vert,\vert g_\beta(t,z)\vert\}.
\end{equation}
\end{theorem}

In order to study $\langle \mathbf{q},u_k\rangle$, we will need the following local law for $\langle\mathbf{q},G(z)\mathbf{q}\rangle$ proved in \cite{bourgade2017huang},
\begin{theorem}[\cite{bourgade2017huang}*{Theorem 2.1}]\label{theo:isolocal}
Let $\vartheta>0$, $\kappa\in(0,1)$ and $\mathbf{q}$ a $L^2$-normalized vector of $\mathbb{R}^N$, we have
 \begin{equation}\label{eq:anisolocal}
 \left\vert\langle\mathbf{q},G(z)\mathbf{q}\rangle-\sum_{\alpha=1}^N
 {q_\alpha^2g_\alpha(t,z)}\right\vert\prec\frac{1}{\sqrt{N\eta}}\mathrm{Im}\left(\sum_{\alpha=1}^N{q_\alpha^2 g_\alpha(t,z)}\right)
 \end{equation}
 uniformly in $z=E+i\eta\in \mathcal{D}_{r}^{\vartheta,\kappa}$
 \end{theorem}
 ~\\ This theorem also gives us control of the resolvent as a bilinear form by polarization. We will give the proof of this corollary for completeness. 
\begin{corollary}\label{coro:aniso}  Let $\vartheta>0$, $\kappa\in(0,1)$, let $\mathbf{v}$ and $\mathbf{w}$ two $L^2$-normalized vectors of $\mathbb{R}^N$, we have
\begin{equation}
\label{eq:anisoloca}\left\vert\langle\mathbf{v},G(z)\mathbf{w}\rangle-\sum_{\alpha=1}^Nv_\alpha w_\alpha g_\alpha(t,z)\right\vert\prec\frac{1}{\sqrt{N\eta}}\sqrt{\mathrm{Im}\left(\sum_{\alpha=1}^N{v_\alpha^2g_\alpha(t,z)} \right)\mathrm{Im}\left(\sum_{\alpha=1}^N{w_\alpha^2g_\alpha(t,z)} \right)}
\end{equation}
uniformly in $z=E+i\eta\in \mathcal{D}_{r}^{\vartheta,\kappa}$
 \end{corollary} 
\begin{proof}Let $\mu\in\mathbb{R}$,  a parameter fixed later. Consider
\begin{equation}\label{eq:polari}\langle\left(\mathbf{v}+\mu
\mathbf{w}\right),G\left(\mathbf{v}+\mu\mathbf{w}\right)\rangle
=\langle \mathbf{v},G\mathbf{v}\rangle+\mu^2\langle\mathbf{w},G\mathbf{w}\rangle
+2\mu\langle \mathbf{v},G\mathbf{w}\rangle,
\end{equation}
by linearity and symmetry of the resolvent $G$. On one hand, using Theorem \ref{theo:isolocal} on the first two terms of the right hand side of \eqref{eq:polari}, we get the equation
\begin{multline}\label{eq:polari1}
\langle\left(\mathbf{v}+\mu
\mathbf{w}\right),G\left(\mathbf{v}+\mu\mathbf{w}\right)\rangle
=2\mu\langle\mathbf{v},G\mathbf{w}\rangle+\sum_{\alpha=1}^Nv_\alpha^2
g_\alpha(t,z)+\mu^2\sum_{\alpha=1}^Nw_\alpha^2g_\alpha(t,z)\\+
\mathcal{O}_\prec\left(\frac{1}{\sqrt{N\eta}}\left(\mathrm{Im}\left(\sum_{\alpha=1}^Nv_\alpha^2g_\alpha(t,z)
+\mu^2\sum_{\alpha=1}^Nw_\alpha^2g_\alpha(t,z)\right)\right)\right).
\end{multline}
On the other hand, using Theorem \ref{theo:isolocal} on the left hand side of \eqref{eq:polari}, we obtain
\begin{multline}\label{eq:polari2}
\langle\left(\mathbf{v}+\mu
\mathbf{w}\right),G\left(\mathbf{v}+\mu\mathbf{w}\right)\rangle=
\sum_{\alpha=1}^N(v_\alpha+\mu w_\alpha)^2g_\alpha(t,z)
\\+\mathcal{O}_\prec\left(\frac{1}{\sqrt{N\eta}}\mathrm{Im}\left(\sum_{\alpha=1}^N(v_\alpha+\mu w_\alpha)^2g_\alpha(t,z)\right)\right).
\end{multline}
Finally, combining \eqref{eq:polari1} and \eqref{eq:polari2} and choosing
$$\mu=\frac{\mathrm{Im}\left(\sum_{\alpha=1}^Nw_\alpha v_\alpha g_\alpha(t,z)\right)}{\mathrm{Im}\left(\sum_{\alpha=1}^Nw_\alpha^2g_\alpha(t,z)\right)},$$
we get the final result.
\end{proof} 
We will also need the following rigidity result from \cite{landon2017convergence}.
\begin{theorem}[\cite{landon2017convergence}*{Theorem 3.5}]\label{theo:rigidity}
	Let $\omega>0$ be a small constant and $\kappa\in(0,1)$. For any $t\in\mathcal{T}_\omega$, we have
	\[
	\vert
	\lambda_k-\gamma_{k,t}
	\vert
	\prec
	\frac{1}{N}
	\]
uniformly in $k\in\mathcal{A}^\kappa_r$
\end{theorem}
 This control of the resolvent allows us to give an upper bound for the moments of the eigenvectors of $W_t$. Defining, for a fixed $\mathbf{q}\in\mathbb{S}^{N-1}$,
\begin{equation}\label{eq:defphi}
\varphi_{t}(\bm{\xi})=\mathds{E}\left[\prod_{k=1}^N\frac{(\sqrt{N}\langle\mathbf{q},u_k\rangle)^{2\xi_k}}{a(2\eta_k)}\middle\vert\bm{\lambda}\right]
\end{equation}
  with $u_1,\dots,u_N$ the eigenvectors of $W_t$ and $\bm{\lambda}$ its eigenvalues, we have the following corollary. 
 
\begin{corollary}\label{coro:normephi}
Let $\kappa\in(0,1),$ $\vartheta>0$ and $\bm{\xi}:\mathcal{A}_{r}^\kappa\to\mathbb{N}$
\begin{equation}
\varphi_t(\bm{\xi})
\prec
\prod_{k=1}^N\sigma_t^{2\xi_k}(\mathbf{q},k,\eta)=:\sigma_t^2(\mathbf{q},\bm{\xi},\eta).
\end{equation}
uniformly in $N^{-1+\vartheta}\leqslant\eta\leqslant N^{-\vartheta}t.$
\end{corollary}

\begin{proof}
Fix any $\vartheta\in(0,1)$ such that $\eta=N^{-1+\vartheta}\ll t$ and $k\in\mathcal{A}_{r}^\kappa$, we have the following first high probability bound with $z_k=\lambda_k+\mathrm{i}\eta$
\begin{multline*}
\frac{1}{\eta}\left(\sqrt{N}\langle\mathbf{q},u_k\rangle\right)^2
=
\frac{\left(\sqrt{N}\langle\mathbf{q},u_k\rangle\right)^2\eta}{(\lambda_k-\mathrm{Re}(z_k))^2+\eta^2}
\leqslant 
N\mathrm{Im}\langle\mathbf{q},G(z_k)\mathbf{q}\rangle\\
= N\mathrm{Im}\left(\sum_{\alpha=1}^N\frac{q_\alpha^2}{D_\alpha-z_k-tm_t(z_k)}\right)+\mathcal{O}_\prec\left(\frac{1}{\sqrt{N\eta}}\mathrm{Im}\left(\sum_{\alpha=1}^N\frac{q_\alpha^2}{D_\alpha-z_k-tm_t(z_k)}\right)\right).
\end{multline*}
We can then write,
$$\left(\sqrt{N}\langle\mathbf{q},u_k\rangle\right)^2\prec {N\eta}\mathrm{Im}\left(\sum_{\alpha=1}^N\frac{q_\alpha^2}{D_\alpha-z_k-tm_t(z_k)}\right)\prec\sigma_t^2(\mathbf{q},k,\eta).$$
where we used the definition of $\prec$ and that $\vartheta$ is as small as we want and the smoothness of $m_t(z)$.
We finish the proof by definition of $\varphi_{\bm{\lambda},t}$.
\end{proof}
%
\paragraph{}The local law gives us a strong control on both the eigenvalues and eigenvectors of our matrix ensemble. As $W_t$ will undergo the Dyson Brownian motion, these quantities will still be controlled through a local law up to a small error coming from the time of the relaxation.
We will now define the event of good eigenvalue paths ($\bm{\lambda}(t+s))_{s\in(0,\tau)}$) in the sense that all the estimates and bound from the previous section such as local laws from Theorems \ref{local}, \ref{theo:isolocal} and \ref{theo:rigidity} hold. First denote the resolvent of $H_s$ and its normalized trace by
\begin{equation}\label{eq:defstielt}
G(s,z)=(H_s-z)^{-1}\quad\text{and}\quad \varsigma(s,z)=\frac{1}{N}\mathrm{Tr}\,H_s=\frac{1}{N}\sum_{k=1}^N\frac{1}{\lambda_k(t+s)-z}.
\end{equation}
\begin{definition}\label{def:goodpath}
Let $\varepsilon,\, \vartheta>0$ and $\kappa\in(0,1)$. An eigenvalue configuration $\bm{\lambda}$ is \emph{good} is the following holds with overwhelming probability conditioning on $\bm{\lambda}(\tau_0)=\bm{\lambda}$ for $N$ large enough,
\begin{enumerate}
\item $\sup_{0\leqslant s\leqslant \tau} \vert \varsigma(s,z)-m_{t+s}(z)\vert \leqslant N^\varepsilon(N\eta)^{-1}$ uniformly in $z\in\mathcal{D}_r^{\vartheta,\kappa}$.
\item $\sup_{0\leqslant s\leqslant \tau} \vert\langle \mathbf{q}, G(s,z)\mathbf{q}\rangle - \sum_{\alpha=1}^N q_\alpha^2g_\alpha(t+s,z)\vert\leqslant N^{2\varepsilon}(N\eta)^{-1/2}\mathrm{Im}\sum_{\alpha=1}^Nq_\alpha^2g_\alpha(t+s,z)$ uniformly in $z\in\mathcal{D}_r^{\vartheta,\kappa}$.
\item $\sup_{0\leqslant s\leqslant \tau} \vert \lambda_i(t+s)-\gamma_{i,t+s}\vert\leqslant  N^\varepsilon N^{-1}$ uniformly in $i\in\mathcal{A}_r^\kappa$
\end{enumerate}
\end{definition}
Note that by the considerations in this section and a continuity argument, we see that \emph{good} eigenvalue paths occur with overwhelming probability so that we can condition on having such a path in the following section. Note that we will make $W_t$ undergo the Dyson Brownian motion for a small time $\tau$ so that the classical location $\gamma_{i,t}$ and the deterministic counterpart to resolvent entries $g_\alpha$ have small variations. This is the statement of the next lemma.
\begin{lemma}
Let $\varepsilon,\, \vartheta>0$ and $\kappa\in(0,1)$. Conditionally on a \emph{good} eigenvalue path as in Definition \ref{def:goodpath}, we have the following control of the resolvent,
\[
\begin{gathered}
\sup_{0\leqslant s\leqslant \tau} \left\vert \varsigma(s,z)-m_t(z)\right\vert\leqslant \frac{N^\varepsilon}{N\eta}+
\frac{\tau}{t},
\\
\sup_{0\leqslant s\leqslant \tau} \left\vert G(s,z)_{\alpha\alpha}-\frac{1}{D_\alpha-z-tm_t(z)}\right\vert
\leqslant \left(\frac{N^{2\varepsilon}}{\sqrt{N\eta}}+\frac{\tau}{t}\right)\vert g_\alpha(t,z)\vert
\end{gathered}
\]
uniformly in $z\in\mathcal{D}_r^{\vartheta,\kappa}.$ For the rigidity results for eigenvalues we have
\[
\sup_{0\leqslant s\leqslant \tau}
\left\vert \lambda_k(t+s)-\gamma_{k,t}\right\vert\leqslant N^\varepsilon\left(\frac{1}{N}+\tau\right)
\]
uniformly in $k\in \mathcal{A}_r^\kappa$.
\end{lemma}
\begin{proof}
The first result comes from the fact that $\vert \partial_t m_t(z)\vert \leqslant N^\varepsilon/t$ which can be deduced from the time evolution of $m_t$ from \cite{landon2017convergence}*{Lemma 7.6}
\[
\partial_tm_t(z)=\partial_z\left(m_t(z)(m_t(z)+z)\right)
\]
combined with the estimates $\vert \partial_z m_t(z)\vert\leqslant C/t$ and $\vert m_t(z)\vert\leqslant \log N.$ The other error term comes from the local law holding for a \emph{good} eigenvalue path.
The proof of the second bound comes from the following simple identity
\[
\partial_tg_\alpha(t,z)=\frac{m_t(z)+t\partial_tm_t(z)}{(D_i-z-tm_t(z))^2}\quad\text{which gives}\quad
\vert \partial_tg_\alpha(t,z)\vert \leqslant \frac{N^\varepsilon}{t}\vert g_\alpha(t,z)\vert.
\]
using the fact that $\vert m_t(z)\vert\leqslant C\log N$ and $\vert\partial_tm_t(z)\vert\leqslant N^\varepsilon/t.$ For the rigidity estimate, we combine the estimate $\vert\partial_t\gamma_{i,t}\vert\leqslant C\log N$ which can be found in \cite{landon2017convergence}*{Lemma 7.6} as well with the rigidity coming from the \emph{good} eigenvalue path.
\[
\]
\end{proof}

\section{Short time relaxation}\label{localmaxprinc}
In this section, we are going to prove Theorems \ref{theo:resultGaus} and \ref{theo:resultQUE} for the Dyson Brownian motion starting from $W_t$ using maximum principle. Note that in this section we will omit the subscript $\bm{\lambda}$ for simplicity.

\paragraph{}Recall the dynamics of the eigenvector moment flow with $n$ particles for $H_\tau$ with $H_0=W_t$.
\begin{equation}\label{eq:defdyna}
\left\{
	\begin{array}{ll}
		\displaystyle{\partial_\tau f_{\tau}(\bm{\xi})=\frac{1}{N}\sum_{i\neq j}\frac{2\eta_i(1+2\eta_j)\left(f_\tau(\bm{\xi}^{i,j})-f_\tau						(\bm{\xi})\right)}{(\lambda_i(t+\tau)-\lambda_j(t+\tau))^2}=:(\mathscr{B}_\tau f_\tau)(\bm{\xi})},\\
		f_0(\bm{\xi})=\varphi_t(\bm{\xi})
	\end{array}
\right.
\end{equation}  
with $\varphi_t(\bm{\xi})$ is defined in \eqref{eq:defphi}. where we noted $\lambda_i(t+\tau)$ the eigenvalues of $H_\tau$. Note that in the case of a single particle in $k$, we can write the dynamics
\begin{equation}\label{one}
\left\{
	\begin{array}{ll}
	\partial_\tau f_{\tau}(k)&=\displaystyle{\frac{2}{N}\sum_{j=1}^N\frac{f_{\tau}(j)-f_{\tau}(k)}{(\lambda_j(t+\tau)-\lambda_k(t+\tau))^2}},\\
	f_0(k)&=\varphi_t(k).
	\end{array}
	\right.
\end{equation}

We cut the dynamics into two parts : the short range where most of the information will be and the long range. This decomposition in this context was first introduced in \cite{erdos2015gap}.  Letting $1\ll \ell\ll N$ be a parameter that we will choose later, we then define
\begin{align}
\label{eq:defshort}&(\mathscr{S}(\tau) f_\tau)(\bm{\xi})=\displaystyle{\frac{1}{N}\sum_{\vert j-k\vert\leqslant \ell}^N\frac{2\eta_i(1+2\eta_j)\left(f_\tau(\bm{\xi}^{i,j})-f		(\bm{\xi})\right)}{(\lambda_i-\lambda_j)^2}},\\
\label{eq:deflong}&(\mathscr{L}(\tau) f_\tau)(k)=\displaystyle{\frac{1}{N}\sum_{\vert j-k\vert> \ell}^N\frac{2\eta_i(1+2\eta_j)\left(f_\tau(\bm{\xi}^{i,j})-f_\tau						(\bm{\xi})\right)}{(\lambda_i-\lambda_j)^2}}.
\end{align}
Denote by $U_\mathscr{S}(s,\tau)$ the semigroup associated with $\mathscr{S}$ from time $s$ to $\tau$ :
\begin{equation}
\partial_\tau U_\mathscr{S}(s,\tau)=\mathscr{S}(\tau) U_\mathscr{S}(s,\tau)
\end{equation} 
for any $s\leqslant \tau$. We will denote in the same way $U_\mathscr{B}$. 
It has been proved in \cites{bourgade2017eigenvector, bourgade2017huang} that the parabolic short range dynamics has a finite speed of propagation in the following sense: define the following distance on the set of configurations with $n$ particles
\begin{equation}\label{eq:distance}
d(\bm{\eta},\bm{\xi})=\sum_{k=1}^n\vert x_k-y_k\vert
\end{equation}
where $(x_1,\dots, x_n)$ are the positions of the particles in nondecreasing order of $\bm{\eta}$ and $y_\alpha$ of $\bm{\xi}$. The following lemma then states that if two configurations are far from each other, the short-range dynamics started at one and evaluated at the other is exponentially small with high probability.

\begin{lemma}[\cite{bourgade2017huang}*{Corollary 3.3}]\label{finitespeed}
Choose $\ell\geqslant N\tau$, let $\varepsilon>0$ be a small constant and $\kappa\in(0,1)$. Conditioning on a good eigenvalue path $(\bm{\lambda}(s))_s,$  uniformly, for any function $h$ on configurations of $n$ particles and a configuration $\bm{\xi}$ outside of the support of $h$ in the sense that $d(\bm{\xi},\bm{\eta})\geqslant N^\varepsilon \ell$ for any configuration $\bm{\eta}$ inside the support of $h$, we have
\begin{equation}
\sup_{0\leqslant s\leqslant s^\prime \leqslant t}U_\mathscr{S}(s,s^\prime)h(\bm{\xi})
\leqslant
N^n \Vert h\Vert_\infty e^{-cN^\varepsilon}
\end{equation}
for any $D>0$.
\end{lemma}

In this section, we condition on an event occurring with overwhelming probability so that we can state the results deterministically. We state it here as the following lemma.
\begin{lemma}
Let $\omega, \mathfrak{a}, \mathfrak{b}, \vartheta$ and $\varepsilon$ be small positive constants and $D$ as in Definition \ref{def:init}. Let $t\in\mathcal{T}_\omega,$ and $W_t$ as in \eqref{def:wigner}. Let $\kappa\in(0,1)$, $\tau\in[N^{-1+\mathfrak{a}}, N^{-\mathfrak{a}}t]$ and $\ell\in[\tau N^{1+\mathfrak{b}},N^{-\mathfrak{b}}t].$ The dynamics $(H_s)_{0\leqslant s\leqslant \tau}$ induces a measure on the space of eigenvalues and eigenvectors $(\bm{\lambda}(t+s),\mathbf{u}(t+s))_{0\leqslant s\leqslant \tau}$.  The event $A$ of trajectories defined by the following holds with overwhelming probability:
\begin{enumerate}
\item $\sup_{0\leqslant s\leqslant \tau} \vert \varsigma(s,z)-m_{t}(z)\vert \leqslant N^\varepsilon(N\eta)^{-1}+\tau t^{-1}$ uniformly in $z\in\mathcal{D}_r^{\vartheta,\kappa}$.
\item $\sup_{0\leqslant s\leqslant \tau} \vert\langle \mathbf{q}, G(s,z)\mathbf{q}\rangle - \sum_{\alpha=1}^N q_\alpha^2g_\alpha(t,z)\vert\leqslant (N^{2\varepsilon}(N\eta)^{-1/2}+\tau t^{-1})\mathrm{Im}\sum_{\alpha=1}^Nq_\alpha^2g_\alpha(t,z)$ uniformly in $z\in\mathcal{D}_r^{\vartheta,\kappa}$.
\item $\sup_{0\leqslant s\leqslant \tau} \vert \lambda_i(t+s)-\gamma_{i,t}\vert\leqslant  N^\varepsilon (N^{-1}+\tau)$ uniformly in $i\in\mathcal{A}_r^\kappa$.
\item For any function $h$ on configurations of $n$ particles and a configuration $\bm{\xi}$ supported outside of the support of $h$ in the sense that $d(\bm{\xi},\bm{\eta})\leqslant N^\varepsilon\ell$ for any configuration $\bm{\eta}$ inside the support of $h$ we have
\[
\sup_{0\leqslant s\leqslant s^\prime \leqslant t}U_\mathscr{S}(s,s^\prime)h(\bm{\xi})
\leqslant
N^n \Vert h\Vert_\infty e^{-cN^\varepsilon}
\]
\end{enumerate}
\end{lemma}

The following lemma gives us a bound on the difference between the short-range and long-range dynamics basically stating that most of the information lies in the short-range dynamics.

\begin{lemma}\label{shortrange}
Fix $\ell \ll Nt$ and consider $\bm{\xi_0}$ to be a configuration of $n$ particles supported on $\mathcal{A}_r^\kappa$ then for any eigenvalue paths $(\bm{\lambda}(t+s),\bm{u}(t+s))_{0\leqslant s\leqslant \tau}$ in $A$, we have for any $\vartheta>0$ and \linebreak$N^{-1+\vartheta}\leqslant \eta\leqslant N^{-\vartheta}\ell/N$,
\begin{equation}
\left\vert
	(U_\mathscr{B}(0,\tau)-U_\mathscr{S}(0,\tau))\varphi_t(\bm{\xi})
\right\vert
\leqslant
N^{(n+4)\varepsilon}\frac{N\tau}{\ell}\sigma_t^2(\mathbf{q},\bm{\xi},\eta).
\end{equation}
\end{lemma}
\begin{proof}
Let $\eta\in[N^{-1+\vartheta}, N^{-\vartheta}t]$. Using Duhamel's formula we can write
\[
\left\vert
	(U_\mathscr{B}(0,\tau)-U_\mathscr{S}(0,\tau))\varphi_t(\bm{\xi})
\right\vert
=
\left\vert
	\int_0^\tau
	U_\mathscr{S}(s,\tau)\mathscr{L}(s)f_s(\bm{\xi})\D s
\right\vert
\]
Now, by definition of the operator $\mathscr{L}(s)$ we have that 
\[
\mathscr{L}(s)f_s(\bm{\xi})
=
\sum_{j,k, \vert j-k\vert>\ell}
2\eta_j(1+2\eta_k)
\frac{f_s(\bm{\xi^{j,k}})-f_s(\bm{\xi})}{N(\lambda_j-\lambda_k)^2}
\]
so that we can bound, using Corollary \ref{coro:normephi} since $\bm{\xi}$ is supported on $\mathcal{A}_r^\kappa$,
\[
\left\vert \mathscr{L}(s)f_s(\bm{\xi})\right\vert
\leqslant2n(1+2n)
\sum_{j,k:\vert j-k\vert>\ell}
\frac{N^{n\varepsilon}\sigma_t^2(\mathbf{q},\bm{\xi},\eta)+\vert f_s(\bm{\xi^{k,j}})\vert}{N(\lambda_j-\lambda_k)^2}.
\]
Now, say the configuration is supported on $p$ sites denoted $(k_1,\dots,k_p)$, then one can write
\begin{equation}\label{eq:longrange}
\left\vert \mathscr{L}(s)f_s(\bm{\xi})\right\vert
\leqslant
C_n\sum_{i=1}^p
\sum_{j:\vert j-k_i\vert >\ell}
\frac{N^{n\varepsilon}\sigma_t^2(\mathbf{q},\bm{\xi},\eta)+\vert f_s(\bm{\xi^{k_i,j}})\vert}{N(\lambda_j-\lambda_{k_i})^2}.
\end{equation}
Now, since $\bm{\xi}$ is supported on $\mathcal{A}_r^\kappa$, we have that by Corollary \ref{coro:normephi} and denoting $\bm{\xi}\setminus k_i$ the configuration $\bm{\xi}$ where we removed a particule from the site $k_i$,
\begin{equation}\label{eq:exceptj}
f_s(\bm{\xi^{k_i,j}})
\leqslant
N^{(n-1)\varepsilon}
\sigma_t^2(\mathbf{q},\bm{\xi}\setminus k_i,\eta)f_s(j).
\end{equation}
Consider now $\eta_q = 2^q\ell/N$ for $q=[0, \lceil \log_2 (N/\ell)\rceil]$, then we can bound
\begin{multline}\label{eq:jsupell1}
\sum_{j:\vert j-k_i\vert>\ell}
\frac{f_s(j)}{N(\lambda_j-\lambda_{i_p})^2}
\leqslant
\sum_{j:\vert j-k_i\vert >\ell}
\hspace{-1em}
\sum_{q=0}^{\lceil \log_2 (N/\ell)\rceil}
\frac{1}{\eta_q}
\frac{\langle{\mathbf{q},\,{u_j}^2}\rangle\eta_q}{(\lambda_{k_i}-\lambda_j)^+\eta_q^2}\\
\leqslant
\hspace{-1em}
\sum_{q=0}^{\lceil \log_2 (N/\ell)\rceil}
\frac{N}{2^q\ell}
\mathrm{Im}
\langle\mathbf{q},\,G(s,\lambda_{k_i}+\I \eta_q)\mathbf{q}\rangle.
\end{multline}
We can now use the anisotropic local law since $\lambda_{k_i}$ lies in the spectral window and since we are on the event $A$,
\begin{multline}\label{eq:jsupell2}
\mathrm{Im} \langle \mathbf{q}, G(s,\lambda_{k_i}+\I \eta_q)\mathbf{q}\rangle
\leqslant
\mathrm{Im} \langle \mathbf{q}, G(s,\lambda_{k_i}+\I \eta)\mathbf{q}\rangle
\leqslant
N^\varepsilon\mathrm{Im}\,m_t(\lambda_{k_i}+\I\eta)\sigma_t(\mathbf{q},k_i,\eta)
\\\leqslant
N^{2\varepsilon}\sigma_t^2(\mathbf{q},k_i,\eta).
\end{multline}
where we used the fact that $\eta\ll \ell/N\leqslant \eta_q$ and that $m_t(z)$ is bounded in the spectral window. 
Combining the estimates \eqref{eq:jsupell1} and \eqref{eq:jsupell2} we obtain that 
\[
\sum_{j:\vert j-k_i\vert>\ell}
\frac{f_s(j)}{N(\lambda_j-\lambda_{k_i})^2}
\leqslant
N^{2\varepsilon}
\frac{N}{\ell}\sigma_t^2(\mathbf{q},k_i,\eta).
\]
Injecting this bound in \eqref{eq:longrange} with \eqref{eq:exceptj}, we obtain that 
\[
\left\vert
	\mathscr{L}(s)\varphi_t(\bm{\xi})
\right\vert
\leqslant
N^{(n+2)\varepsilon}
\frac{N}{\ell}
\sigma_t^2(\mathbf{q},\bm{\xi},\eta)
\]
where we used the fact that, by the same argument as in \eqref{eq:jsupell1},
\[
\sum_{\vert j-k_i\vert>\ell}\frac{1}{N(\lambda_j-\lambda_{k_i})^2}\lesssim \frac{N}{\ell}.
\]
Now, we can use that $U_\mathscr{S}$ is a contraction combined with the finite speed of propagation from Lemma \ref{finitespeed} so that we can write that 
\begin{equation}\label{eq:contractionshort}
\left\vert
	U_\mathscr{S}(s,\tau)\mathscr{L}(s)f_s(\bm{\xi})
\right\vert
\leqslant
N^\varepsilon\sup_{\bm{\eta}:d(\bm{\eta},\bm{\xi})\leqslant N^\varepsilon\ell}
\left\vert
	\mathscr{L}(s)f_s(\bm{\eta})
\right\vert
\leqslant
N^{(n+3)\varepsilon}\frac{N}{\ell}\sup_{\bm{\eta}:d(\bm{\eta},\bm{\xi})\leqslant N^\varepsilon\ell}\sigma_t^2(\mathbf{q},\bm{\eta},\eta).
\end{equation}
However, since we have that $\bm{\eta}$ is close to $\bm{\xi}$ we can use regularity of $\sigma_t^2(\mathbf{q},\bm{\eta},\eta)$. Indeed, if one looks at the function 
\[
\psi(x)
=
\sum_{\alpha=1}^N
\frac{q_\alpha^2t}{(D_\alpha-x-t\mathrm{Re}\,m_t(x))^2+(t\mathrm{Im}\,m_t(x))^2}
\]
Then we have that 
\[
\partial_x\psi(x)
=
\sum_{\alpha=1}^N
\frac{q_\alpha^2t
	(2(1+t\partial_x\mathrm{Re}\,m_t(x))
	(D_\alpha-x-t\mathrm{Re}\,m_t(x))
	-2t^2\partial_x\mathrm{Im}\,m_t(x)\mathrm{Im}\,m_t(x)
}{\left[(D_\alpha-x-t\mathrm{Re}\,m_t(x))^2+(t\mathrm{Im}\,m_t(x))^2\right]^2}
\]
So that we can obtain the bound using the fact that $\vert\partial_x m_t(x)\vert\leqslant N^\varepsilon/t$,
\[
\vert \partial_x\psi(x)\vert
\leqslant
\frac{N^\varepsilon}{t}
\psi(x).
\]
We can use this bound in order to obtain the following variation formula for $\sigma_t^2(\mathbf{q},k,\eta),$
\begin{equation}\label{eq:variationsigma}
\left\vert
	\sigma_t^2(\mathbf{q},\bm{\xi},\eta)-\sigma_t^2(\mathbf{q},\bm{\eta},\eta)
\right\vert
\leqslant
N^\varepsilon\frac{d(\bm{\xi},\bm{\eta})}{Nt} \sigma_t^2(\mathbf{q},\bm{\xi},\eta).
\end{equation}
In \eqref{eq:contractionshort}, one can see that the supremum is only taken over configurations close to each other, namely such that $d(\bm{\xi},\bm{\eta})\leqslant N^{\varepsilon} \ell$. Since we have $\ell\ll Nt$, by \eqref{eq:variationsigma}, $\sigma_t^2(\mathbf{q},\bm{\xi},\eta)$ varies slowly and we can finally bound
\[
\left\vert
	U_\mathscr{S}(s,\tau)\mathscr{L}(s)f_s(\bm{\xi})
\right\vert
\leqslant
N^{(n+4)\varepsilon}
\frac{N}{\ell}\sigma_t^2(\mathbf{q},\bm{\xi},\eta)
\]
\end{proof}

The two previous lemmas will be very useful tools to prove Theorem 1.4. Indeed, the finite speed of propagation in Lemma \ref{finitespeed} allows us to localize our problem but is a property of the short range dynamics, Lemma \ref{shortrange} then tells us that most of the information of the global dynamics is in this short range part.

\subsection{Analysis of the moment observable}\label{subsec:moment}
To prove Theorem \ref{theo:resultGaus} we will prove the following intermediary proposition,
\begin{proposition}\label{prop:moment}
Conditionally on $(\bm{\lambda},\bm{u})\in A$, let $\kappa\in(0,1)$, $\varepsilon>0$ and $n$ be an integer. If $\mathbf{q}\in\mathbb{S}^{N-1},$ for any $\bm{\xi}:\mathcal{A}^\kappa_{r}\to\mathbb{N}$ such that $\mathcal{N}(\bm{\xi})=n$, there exists a $p$ depending on $n$ such that we have
\begin{equation}\label{resultbulk}
f_{\tau}(\bm{\xi})=\sigma_t^2(\mathbf{q},\bm{\xi},N^{-\varepsilon}\tau)+\mathcal{O}\left(N^{p\varepsilon}\left(\frac{1}{\sqrt{N\tau}}+\left(\frac{\tau}{t}\right)^{1/3}\right)\sigma_t^2(\mathbf{q},\bm{\xi},N^{-\varepsilon}\tau)\right).
\end{equation}
where $\sigma_t(\mathbf{q},k,\tau)$ is given by \eqref{defsigma}.
\end{proposition}

The $1/3$ exponent that we give here in the error is not optimal. We are not able to reach an optimal error because of the strong dichotomy we do between the short range and the long range dynamics and also the localization technique where more parameters need to be tuned. Using a multi-scale partition of the dynamics could improve the error term. Note also the choice of the parameter $\eta$ corresponds to $N^{-\varepsilon}\tau$ which optimize our error term.
\paragraph{}Let $\varepsilon>0$ be a small constant. Recall that $t\in\mathcal{T}_\omega$ and $\tau\ll t$. First, the following lemma gives us a local law for $f_\tau$, in the case of a single particle, deduced from the isotropic local law for $W_t$ in Theorem \ref{theo:isolocal}.
\begin{lemma}\label{lem:localtau}
For $z\in\mathcal{D}_{r}^{\vartheta,\kappa}$, we have
\begin{multline*}
\mathrm{Im}\left(\frac{1}{N}\sum_{k=1}^N\frac{f_\tau(k)}{\lambda_k(t+\tau)-z}\right)=\mathrm{Im}\left(\sum_{\alpha=1}^N\frac{q_\alpha^2}{D_{\alpha}-z-tm_t(z)}\right)\\+\mathcal{O}\left(N^\varepsilon\left(\frac{1}{\sqrt{N\eta}}+\frac{\tau}{t}\right)\mathrm{Im}\left(\sum_{\alpha=1}^N\frac{q_\alpha^2}{D_\alpha-z-tm_t(z)}\right)\right).
\end{multline*}
\end{lemma}
\begin{proof}
See first that, by definition of $f_\tau$, we have
$$\frac{1}{N}\sum_{k=1}^N\frac{f_\tau(k)}{\lambda_k(t+\tau)-z}=\mathds{E}\left[\langle\mathbf{q},G^{H_\tau}(z)\mathbf{q}\rangle\middle\vert\bm{\lambda}\right]$$
where $G^{H_\tau}:=(H_\tau-z)^{-1}$ is the resolvent of $H_\tau$. Now, the law of $H_\tau$ is $D+\sqrt{t}W +\sqrt{\tau}\mathrm{GOE}\overset{(d)}{=}D+\sqrt{t+\tau}W'$ for some $W'$ a Wigner matrix. We can use Theorem \ref{theo:isolocal} for this matrix and write
\begin{multline*}\mathrm{Im}\left(\langle\mathbf{q},G^{H_\tau}\mathbf{q}\rangle\right)=\mathrm{Im}
\left(\sum_{\alpha=1}^N\frac{q_\alpha^2}{D_\alpha-z-(t+\tau)m_{t+\tau}(z)}\right)\\+\mathcal{O}\left(N^\varepsilon\frac{1}{\sqrt{N\eta}}\mathrm{Im}
\left(\sum_{\alpha=1}^N\frac{q_\alpha^2}{D_\alpha-z-(t+\tau)m_{t+\tau}(z)}\right)\right)
\end{multline*}
with $m_{t+\tau}(z)$ the solution with positive imaginary part of the following self-consistent equation,
$$m_{t+\tau}(z)=\frac{1}{N}\sum_{\alpha=1}^N\frac{1}{D_\alpha-z-(t+\tau) m_{t+\tau}(z)}.$$
Note that we have, for $z\in\mathcal{D}_{r}^{\vartheta,\kappa}$,  $0<\mathrm{Im}(m_{t+\tau}(z))\leqslant C$ for some constant $C$ so that, by a Taylor expansion in $\tau\ll t$ (remember that $\eta\ll t$ for $z\in\mathcal{D}_{r}^{\vartheta,\kappa}),$ we obtain
\begin{multline*}
\mathrm{Im}
\left(\sum_{\alpha=1}^N\frac{q_\alpha^2}{D_\alpha-z-(t+\tau)m_{t+\tau}(z)}\right)=\mathrm{Im}
\left(\sum_{\alpha=1}^N\frac{q_\alpha^2}{D_\alpha-z-tm_{t+\tau}(z)}\right)\\
+\mathcal{O}\left(N^\varepsilon\frac{\tau}{t}\mathrm{Im}
\left(\sum_{\alpha=1}^N\frac{q_\alpha^2}{D_\alpha-z-tm_{t}(z)}\right)\right).
\end{multline*}

We get the final result with the bound $\vert \partial_tm_t(z)\vert\leqslant C(\log N)/t.$ 
\end{proof}
 Let {$\bm{\xi}_0\subset \mathcal{A}^\kappa_{r}$} be a fixed configuration, we want to use a maximum principle on a window centered around $\bm{\xi}_0$ of size $w$ according to the distance \eqref{eq:distance}. Since we make a small perturbation $\tau\ll t$, in order to notice the dynamics in this window, we need to have $w\gg N\tau$. Furthermore, we want to look in the part of the spectrum where the eigenvector will be of typical size $1/\sqrt{Nt}$, to localize the dynamics in this small part of the spectrum. We then need to take $w\ll Nt.$ 
 
We define the following flattening and averaging operator, for $a>0$
\begin{equation}
\left(\text{Flat}_{\bm{\xi}_0}^af\right)(\bm{\xi})=
\left\{
\begin{array}{ll}
f(\bm{\xi})\,\,&\text{if}\,\,d(\bm{\xi},\bm{\xi_w})\leqslant a,\\
\displaystyle{\sigma_t^2(\mathbf{q},\bm{\xi_0})}\,\,\,&\text{if}\,d(\bm{\xi},\bm{\xi_w})> a
\end{array}
\right.
\end{equation}
and
\begin{equation}
(\text{Av}_{\xi_0}f)(\bm{\xi})=\frac{2}{w}\int_{w/2}^w(\text{Flat}_{\xi_0}^af)(\bm{\xi})\mathrm{d}a.
\end{equation}
Notice that for every $\bm{\xi}$, there exists $a_{\bm{\xi}}\in[0,1]$ such that
\begin{equation}\label{notationak}
\text{Av}_{\xi_w}f(\bm{\xi})=a_{\bm{\xi}}f(\bm{\xi})+(1-a_{\bm{\xi}})\sigma_t^2(\mathbf{q},\bm{\xi_0}).
\end{equation}
\subsubsection{Proof of Proposition \ref{prop:moment} : Case of a single particle}
To show the result \eqref{resultbulk} by induction, we will first prove in the case of one particle with the dynamics \eqref{one}. We first prove under the hypotheses as in Proposition \ref{prop:moment} in the case where $n=1$,
 \begin{equation}\label{goal}
f_\tau(k)=\sigma_t^2(\mathbf{q},k,\eta)+\mathcal{O}\left(N^{\varepsilon}\left(\frac{1}{\sqrt{N\tau}}+\left(\frac{\tau}{t}\right)^{1/3}\right)\sigma_t^2(\mathbf{q},k,\eta)\right).
\end{equation}

To do so, we want to use a localized maximum principle centered around $k_w$ which is the position of the particle for the configuration $\bm{\xi_0}$ in this case. However, we need to know that the maximum stays in that window, that is why we first flatten and average $f_t(k)$ and use it as an initial condition for the dynamics \eqref{discs}. We will then make the short range dynamics work on $f_t(k)$ during a time $\tau\ll t$. Since we use the short range dynamics for a time $\tau$ we will be able to use the finite speed of propagation \eqref{finitespeed} and we should choose $Nt\gg\ell\geqslant N\tau$ for the range cut-off. We will take an explicit value at the end of the proof. The different parameters and scaling is illustrated in Figure \ref{fig:illusproof}.

\begin{figure}[!ht]
	\centering
	\includegraphics[scale=.8]{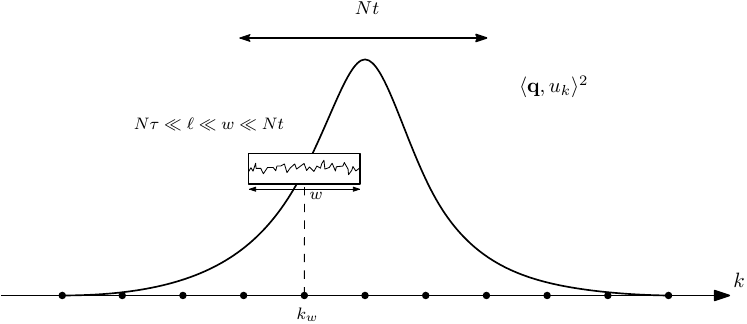}
	\caption{We see here a sketch of the variance profile plotted in the spectral dimension: the projection $\mathbf{q}$ is fixed and we plot the profile as a function of the eigenvector $u_k$. We then localize the dynamics onto the small window plotted here: the window is small enough so that the eigenvector can be seen as ``flat'' but large enough so that the short-range dynamics will not involve indices outside of this window.}
	\label{fig:illusproof}
\end{figure}
\paragraph{}Consider $g_{\tau},$
\begin{align*}
\partial_{\tau} g_{\tau}(k)=\frac{1}{N}\sum_{\vert j-k\vert\leqslant \ell}\frac{g_{\tau}(j)-g_{\tau}(k)}{(\lambda_j(t+\tau)-\lambda_k(t+\tau))^2}\quad\text{with}\quad
g_{0}(k)=(\text{Av}_{k_w}\varphi_{t})(k).
\end{align*}
First note that, in order to prove \eqref{goal}, it is enough to show that
\begin{equation}\label{eq:gtau}
g_{\tau}(k)=\sigma_t^2(\mathbf{q},k,\eta)+\mathcal{O}\left(\left(\frac{\ell}{w}+\frac{N\tau}{\ell}+\frac{1}{\sqrt{N\eta}}+\frac{\tau}{\eta}+\frac{w}{Nt}+\frac{N\eta}{\ell}\right)\sigma_t^2(\mathbf{q},k,\eta)\right)
\end{equation}
where the parameter $\eta$, the spectral resolution, will be chosen so that $N^{-1}\ll\eta\ll t$.
Indeed we have, 
\begin{align}
\hspace{-0.8cm}\vert f_{\tau}(k)-g_{\tau}(k)\vert&=\left\vert (U_\mathscr{B}(0,\tau)\varphi_t)(k)-(U_{\mathscr{S}}(0,\tau)\text{Av}_{k_w}\varphi_{t})(k)\right\vert\\
\label{suffice1}&=\left[(U_{\mathscr{B}}(0,\tau)-U_{\mathscr{S}}(0,\tau))\varphi_t\right](k) +[U_\mathscr{S}(0,\tau)(\mathrm{Id}-\text{Av}_{k_w})\varphi_t](k).
\end{align}
By Lemma \ref{shortrange}, we get
\begin{equation}
[(U_{\mathscr{B}}(0,\tau)-U_{\mathscr{S}}(0,\tau))\varphi_t](k)\leqslant
N^{(n+4)\varepsilon}\frac{N\tau}{\ell}\sigma_t^2(\mathbf{q},k,\eta).
\end{equation}
Now, for the second term in \eqref{suffice1}. Since $(\varphi_{t}-\text{Av}_{k_w}\varphi_{t})(k)=0$ for $k\in [\![k_w-w/2,k_w+w/2]\!],$ and taking $w\gg \ell N^\varepsilon$, looking at $k\in[\![k_w-w/3,k_w+w/3]\!]$ for instance, Lemma \ref{finitespeed} tells us that the term is exponentially small. Thus, we obtain
$$f_\tau(k)=g_\tau(k)+\mathcal{O}\left(N^{(n+4)\varepsilon}\frac{N\tau}{\ell}\sigma_t^2(\mathbf{q},k,\eta)\right).$$

We first prove the following equation which can be seen as an averaged version of the result.  We now show the following lemma which is analogous to \cite{bourgade2017eigenvector}*{Lemma 7.3}
\begin{lemma}
For $k_0\in[\![k_w-u,k_w+u]\!],$ set $z^{(k_0)}=\gamma_{k_0,t}+\mathrm{i}\eta\in\mathcal{D}_{r}^{\vartheta,\kappa}$
\begin{multline}\label{g}
\left\vert 
	\mathrm{Im}\left(
		\frac{1}{N}\sum_{k=1}^N\frac{g_{\tau}(k)}{\lambda_k(t)-z^{(k_0)}}
	\right)
	- 
	\mathrm{Im}(m_t(z^{(k_0)}))\sigma_t^2(\mathbf{q},k_w,\eta)
\right\vert
\\
\leqslant
N^{(n+6)\varepsilon}\left(\frac{\ell}{w}+\frac{N\tau}{\ell}+\frac{1}{\sqrt{N\eta}}+\frac{\tau}{\eta}+\frac{w}{Nt}+\sqrt{\eta}\right)\sigma_t^2(\mathbf{q},k_w,\eta).
\end{multline}
\end{lemma}

\begin{proof}
First, decompose the left hand side term into three different terms :
\begin{align}
&\label{541}\left\vert \mathrm{Im}\left(\frac{1}{N}\sum_{k=1}^N\frac{(U_\mathscr{S}(0,\tau)\text{Av}_{k_w}\varphi_{t})(k)-(\text{Av}_{k_w}U_\mathscr{S}(0,\tau)\varphi_{t})(k)}{\lambda_k-z^{(k_0)}}\right)\right\vert\\
\label{542}+&\left\vert \mathrm{Im}\left(\frac{1}{N}\sum_{k=1}^N\frac{(\text{Av}_{k_w}U_\mathscr{S}(0,\tau)\varphi_t)(k)-\text{Av}_{k_w}f_\tau(k)}{\lambda_k-z^{(k_0)}}\right)\right\vert\\
\label{543}+&\left\vert \mathrm{Im}\left(\frac{1}{N}\sum_{k=1}^N\frac{\text{Av}_{k_w}f_\tau(k)}{\lambda_k-z^{(k_0)}}\right)-\mathrm{Im}(m_t(z^{(k_0)}))\sigma_t^2(\mathbf{q},k_w,\eta)\right\vert.
\end{align}

To bound \eqref{541}, we write
\begin{equation}\label{eq:integrand}
(U_\mathscr{S}(0,\tau)\text{Av}_{k_w}\varphi_{t})(k)-(\text{Av}_{k_w}U_\mathscr{S}(0,\tau)\varphi_{t})(k)=\frac{2}{w}\int_{w/2}^w\mathfrak{U}_{k_w}^a(\varphi_t)(k)\mathrm{d}a
\end{equation}
with
\begin{equation*}
\mathfrak{U}_{k_w}^a(\varphi_t)(k)=(U_\mathscr{S}(0,\tau)\text{Flat}_{k_w}^a\varphi_{t})(k)-(\text{Flat}_{k_w}^aU_\mathscr{S}(0,\tau)\varphi_{t})(k).
\end{equation*}
Look now at what happens around $k_w-a$, the other boundary of the window $k_w+a$ can be bounded exactly the same way.
By finite speed of propagation from Lemma \ref{finitespeed}, for $k<k_w-a-\ell N^\varepsilon$, we easily get
$$ \left(U_\mathscr{S}(0,\tau)\text{Flat}_{k_w}^a\varphi_t\right)(k)=\text{Flat}_{k_w}^a\left(U_\mathscr{S}(0,\tau)\varphi_t\right)(k)+\mathcal{O}\left(N^ne^{-N^{\varepsilon/2}}\right).$$ 
The same equality is true for $k>k_w-a+\ell N^\varepsilon$ using the same argument.\\
\indent For $k_w-a-\ell N^\varepsilon\leqslant k\leqslant k_w-a+\ell N^\varepsilon$, since the operator $U_\mathscr{S}$ is bounded in $\ell^\infty$, we have
\begin{equation}
\left\vert\mathfrak{U}_{k_w}^a(\varphi_t)(k)\right\vert\leqslant 2\sup_{j:\vert j-k_w\vert\leqslant N^\varepsilon\ell}\varphi_t(j)\leqslant N^\varepsilon \sigma_t^2(\mathbf{q},k_w,\eta)
\end{equation}
where we used Corollary \ref{coro:normephi} combined with \eqref{eq:variationsigma}. Now, in the integrand in \eqref{eq:integrand} there is a set of measure at most $N^\varepsilon\ell$ which is not exponentially small which gives, combined with Theorem \ref{local},
that we can bound \eqref{541} by 
\begin{equation}\label{322}
\eqref{541}\leqslant N^\varepsilon \frac{\ell}{w}\mathrm{Im}\,m_t(z^{(k_0)}))\sigma^2_t(\mathbf{q},k_w)\leqslant N^\varepsilon \frac{\ell}{w}\sigma_t^2(\mathbf{q},k_w,\eta),
\end{equation}
where we used that $\mathrm{Im}(m_t(z^{(k_0)}))$ is bounded in the spectral window from Lemma \ref{lem:stieljes}.

To bound \eqref{542}, noting that $f_\tau=U_\mathscr{B}(0,\tau)\varphi_t$,
\begin{multline*}
\vert(\text{Av}_{k_w}U_\mathscr{S}(0,\tau)\varphi_t)(k)-(\text{Av}_{k_w}U_\mathscr{B}(0,\tau)\varphi_t)(k)\vert\leqslant \left\vert\left[\left(U_\mathscr{S}(0,\tau)-U_\mathscr{B}(0,\tau)\right)\varphi_t\right](k)\right\vert
\\\leqslant N^{5\varepsilon}\frac{N\tau}{\ell}\sigma^2_t(\mathbf{q},k,\eta)
\end{multline*}
where we applied Lemma \ref{shortrange}.\\
Thus we have
\begin{align}\label{323}
\eqref{542}\leqslant N^{5\varepsilon}\frac{CN\tau}{\ell}\sigma_t^2(\mathbf{q},k,\eta)\mathrm{Im}\left(\frac{1}{N}\sum_{k=1}^N\frac{1}{\lambda_k-z^{(k_0)}}\right)\leqslant N^{5\varepsilon}\frac{N\tau}{\ell}\sigma_t^2(\mathbf{q},k,\eta)
\end{align}
where we used the averaged local law from Theorem \ref{local} and the fact that in the spectral window we have $\mathrm{Im}\left(m_t(z^{(k_0)})\right)\leqslant C.$

\paragraph{}To bound \eqref{543}, we want to use the local laws. Recalling that $z^{(k_0)}=\gamma_{k_0}+\mathrm{i}\eta,$ then
\begin{multline*}
\mathrm{Im}\left(\frac{1}{N}\sum_{k=1}^N\frac{(\text{Av}_{k_w}U_\mathscr{B}(0,\tau)\varphi_t)(k)}{\lambda_k-z^{(k_0)}}\right)=\mathrm{Im}\left(\frac{1}{N}\sum_{\vert k-k_0\vert\leqslant N\sqrt{\eta}}\frac{(\text{Av}_{k_w}U_\mathscr{B}(0,\tau)\varphi_t)(k)}{\lambda_k-z^{(k_0)}}\right)\\+\mathcal{O}\left(\sqrt{\eta}\sigma_t^2(\mathbf{q},k_0,\eta)\right).
\end{multline*}
where we used the fact that similarly to the proof of Lemma \ref{shortrange}, we have that, for any threshold $\ell_0$,
\begin{equation}\label{eq:partialsum}
\mathrm{Im}\left(
	\frac{1}{N}\sum_{\vert k-k_0\vert> \ell_0}\frac{(\text{Av}_{k_w}f_{\tau})(k)}{\lambda_k-z^{(k_0)}}
\right)
=
\mathcal{O}\left(
	N^\varepsilon\frac{N\eta}{\ell_0}
	\sigma_t^2(\mathbf{q},k_0,\eta)
\right)
\end{equation}
Taking $\ell_0=N\sqrt{\eta}$ gives here the bound.
If we use the notation \eqref{notationak}, we obtain
\begin{align}
&\mathrm{Im}\left(
	\frac{1}{N}\sum_{k=1}^N\frac{(\text{Av}_{k_w}f_{\tau})(k)}{\lambda_k-z^{(k_0)}}
\right)
=\\
&=\mathrm{Im}\left(
	\frac{1}{N}\sum_{\vert k-k_0\vert\leqslant N\sqrt{\eta}}\frac{a_kf_{\tau}(k)+(1-a_k)\sigma_t^2(\mathbf{q},k_w,\eta)}{\lambda_k-z^{(k_0)}}
\right)
+
\mathcal{O}\left(
	\sqrt{\eta}\sigma_t^2(\mathbf{q},k_0,\eta)
\right)\\
&\label{336}= 
a_{k_0}\mathrm{Im}\left(
	\frac{1}{N}\sum_{k=1}^N\frac{f_{\tau}(k)}{\lambda_k-z^{(k_0)}}
\right)
+
(1-a_{k_0})\sigma_t^2(\mathbf{q},k_w,\eta)\mathrm{Im}\left(
	\frac{1}{N}\sum_{k=1}^N\frac{1}{\lambda_k-z^{(k_0)}}
\right)\\
\label{337}&+
\mathrm{Im}\left(
	\frac{1}{N}\sum_{\vert k-k_0\vert\leqslant N\sqrt{\eta}}^N\frac{(a_k-a_{k_0})f_{\tau}(k)+(a_{k_0}-a_k)\sigma_t^2(\mathbf{q},k_w,\eta)}{\lambda_k-z^{(k_0)}}
\right)
+
\mathcal{O}\left(
	\sqrt{\eta}\sigma_t^2(\mathbf{q},k_0,\eta)
\right).
\end{align}

Note now that $$\mathrm{Im}\left(\sum_{\alpha=1}^N\frac{q_\alpha^2}{D_\alpha-z^{(k_0)}-tm_t(z^{(k_0)})}\right)=\left(\mathrm{Im}(m_t(z^{(k_0)}))+\frac{\eta}{t}\right)\sigma_t^2(\mathbf{q},k_0,\eta),$$
so that, by Lemma \ref{lem:localtau} and Theorem \ref{local}, we obtain
\begin{multline*}
\eqref{336}=a_{k_0}\mathrm{Im}(m_t(z^{(k_0)}))
\sigma_t^2(\mathbf{q},k_0, \eta)
+(1-a_{k_0})\sigma_t^2(\mathbf{q},k_w,\eta)\mathrm{Im}(m_t(z^{(k_0)}))\\
+\mathcal{O}\left(
	N^\varepsilon\left(
		\frac{1}{\sqrt{N\eta}}+\frac{1}{N\eta}+\frac{\tau}{t}+\sqrt{\eta}
	\right)
	\sigma_t^2(\mathbf{q},k_w,\eta)\right)
\end{multline*}
Note that we do not keep the error $\eta/t$ as we take $\eta\ll \tau$ so that $\tau/t$ is of larger order. Now using the deterministic bound \eqref{eq:variationsigma}, we obtain that since $k\in[\![k_w-w,k_w+w]\!]$,
\[
\eqref{336}
=
\mathrm{Im}(m_t(z^{(k_0)}))\sigma_t^2(\mathbf{q},k_w,\eta)
+
\mathcal{O}\left(
	N^\varepsilon\left(
		\frac{1}{\sqrt{N\eta}}+\frac{\tau}{t}+\frac{w}{Nt}+\sqrt{\eta}
	\right)\sigma_t^2(\mathbf{q},k_w,\eta)
\right)
\]
Finally, with the elementary property $\vert a_i-a_k\vert\leqslant \frac{C\vert i-k\vert}{N},$ we get that $\eqref{337}\leqslant\! C{\sqrt{\eta}\sigma_t^2(\mathbf{q},k_w,\eta)}.$ Putting these estimates together, we obtain
 \begin{equation}
\label{324}\eqref{543}=\mathcal{O}\left(
	N^\varepsilon\left(
		\frac{1}{\sqrt{N\eta}}+\frac{\tau}{t}+\frac{w}{Nt}+\sqrt{\eta}
	\right)\sigma_t^2(\mathbf{q},k_w,\eta)
\right)
 \end{equation}
 
Combining  \eqref{322},\eqref{323} and \eqref{324}, we get the final result 
 \begin{multline*}
\left\vert \mathrm{Im}\left(\frac{1}{N}\sum_{k=1}^N\frac{g_{\tau}(k)}{\lambda_k-z^{(k_0)}}\right)- \mathrm{Im}(m_t(z^{(k_0)}))\sigma_t^2(\mathbf{q},k_w,\eta)\right\vert\\
=
\mathcal{O}\left(N^\varepsilon\left(\frac{\ell}{w}+\frac{N\tau}{\ell}+\frac{1}{\sqrt{N\eta}}+\frac{\tau}{t}+\frac{w}{Nt}+\sqrt{\eta}\right)\sigma_t^2(\mathbf{q},k_w,\eta)\right).
 \end{multline*}
 
\end{proof}

 Now, we just need to prove that \eqref{eq:gtau}. Let $k_m$ be the index such that $g_{\tau}(k_m)=\max_{k}g_{\tau}(k)$ and $z=\lambda_{k_m}+i\eta.$ If we have that 
\begin{align}\label{-10}
\left\vert g_{\tau}(k_m)-\sigma_t^2(\mathbf{q},k_m,\eta)\right\vert\leqslant N^{-10},
\end{align}
there is nothing to prove. Now if the left hand side is greater than $N^{-10},$ by finite speed of propagation, $k_m$ is in the interval $[\![k_w-u,k_w+u]\!].$ Indeed, if it is not then the difference in \eqref{-10} would be exponentially small. We then have
\begin{align}
\partial_{\tau} &g_{\tau}(k_m)=\frac{1}{N}\sum_{\substack{\vert j-k_m\vert\leqslant \ell\\j\neq k_m}}\frac{g_{\tau}(j)-g_{\tau}(k_m)}{(\lambda_j-\lambda_{k_m})^2}\nonumber\\
&\leqslant \frac{1}{N\eta}\sum_{\substack{\vert j-k_m\vert\leqslant \ell\\j\neq k_m}}\frac{\eta g_{\tau}(j)}{(\lambda_j-\lambda_{k_m})^2+\eta^2}-\frac{g_{\tau}(k_m)}{N\eta}\sum_{\substack{\vert j-k_m\vert\leqslant \ell\\j\neq k_m}}\frac{\eta}{(\lambda_j-\lambda_{k_m})^2+\eta^2}\nonumber\\
&\leqslant \frac{1}{\eta}\mathrm{Im}\left(\frac{1}{N}\sum_{k=1}^N\frac{g_{\tau}(k)}{\lambda_k-z}\right)-\frac{g_{\tau}(k_m)}{\eta}\mathrm{Im}(m_t(z))+\mathcal{O}\left(\frac{N^\varepsilon}{\eta}\left(\frac{N\eta}{\ell}+\frac{1}{N\eta}\right)\sigma_t^2(\mathbf{q},k_m,\eta)\right)\nonumber\\
&\label{348}\leqslant \frac{C}{\eta}\left(\sigma_t^2(\mathbf{q},k_w,\eta)-g_{\tau}(k_m)\right)
\\&+\mathcal{O}\left(\frac{N^\varepsilon}{\eta}\left(\left(\frac{N\eta}{\ell}+\frac{1}{N\eta}\right)\sigma_t^2(\mathbf{q},k_m,\eta)+\left(\frac{\ell}{w}+\frac{N\tau}{\ell}+\frac{1}{\sqrt{N\eta}}+\frac{\tau}{t}+\frac{w}{Nt}+\sqrt{\eta}\right)\sigma_t^2(\mathbf{q},k_w,\eta)\right)\right)
\end{align}
where we used in the first inequality that $g_{\tau}(k_m)$ is the maximum, in the second inequality that extending the sum to all $j$ adds an error $(N^{1+\varepsilon}\eta/\ell\sigma_t^2(\mathbf{q},k_m,\eta))$ using \eqref{eq:partialsum} and Theorem \ref{local} combined with the estimate from Lemma \ref{lem:stieljes}. Finally in the last inequality we used \eqref{g}, $c\leqslant\mathrm{Im}(m_t(z))\leqslant C$ in the spectral window and that the rigidity errors that appears from changing $\lambda_{k_m}$ into $\gamma_{k_m,t}$ from Theorem \ref{theo:rigidity} are smaller than the other terms. 
Injecting our variance $\sigma_t^2(\mathbf{q},k_w,\eta)$ in \eqref{348} which does not depend of $\tau$,
\begin{multline*}
\partial_{\tau}\left(g_{\tau}(k_m)-\sigma_t^2(\mathbf{q},k_w,\eta)\right)\leqslant -\frac{C}{\eta}\left(g_{\tau}(k_m)-\sigma_t^2(\mathbf{q},k_w,\eta)\right)\\
+\mathcal{O}\left(\left(\frac{N\eta}{\ell}+\frac{\ell}{w}+\frac{N\tau}{\ell}+\frac{1}{\sqrt{N\eta}}+\frac{\tau}{t}+\frac{w}{Nt}+\sqrt{\eta}\right)\frac{N^\varepsilon}{\eta}\sigma_t^2(\mathbf{q},k_w,\eta)\right).
\end{multline*}
Thus, writing 
$S_{\tau}=g_{\tau}(k_m)-\sigma_t(\mathbf{q},k_w)^2$
we get
\begin{equation*}
 \partial_{\tau} S_{\tau}\leqslant -\frac{C}{\eta}S_{\tau}+\mathcal{O}\left(\left(\frac{N\eta}{\ell}+\frac{\ell}{w}+\frac{N\tau}{\ell}+\frac{1}{\sqrt{N\eta}}+\frac{\tau}{t}+\frac{w}{Nt}+\sqrt{\eta}\right)\frac{N^\varepsilon}{\eta}\sigma_t^2(\mathbf{q},k_w,\eta)\right).
\end{equation*}

Note that $S_\tau$ is not necessary differentiable as the maximum is not necesarily unique for instance, but one can get the result by instead considering 
\[
\partial_\tau S_\tau = \limsup_{s\rightarrow\tau}\frac{S_s-S_\tau}{s-\tau}.
\]
 Using Gronwall's lemma, we have if $\eta\ll \tau$,
  $$S_{\tau}=\mathcal{O}\left(N^\varepsilon\left(\frac{N\eta}{\ell}+\frac{\ell}{w}+\frac{N\tau}{\ell}+\frac{1}{\sqrt{N\eta}}+\frac{\tau}{t}+\frac{w}{Nt}+\sqrt{\eta}\right)\sigma_t^2(\mathbf{q},k_w,\eta)+N^{-C}\right)$$
 for any $C$. We can do the same reasoning with the infimum. Finally taking the following parameters
\begin{equation}\label{eq:parameter}
\eta = N^{-\varepsilon}\tau,\quad w=\left(N\tau(Nt)^2\right)^{1/3},\quad \ell = \sqrt{N\tau u}
\end{equation} 
we get the result for a single particle. Note that with these choices of parameters we have the correct relations: $N^{-1}\ll\eta\ll\tau\ll \ell\ll w\ll Nt.$
\subsubsection{Proof of Proposition \ref{prop:moment}: Case of $\texorpdfstring{\bm{n}}{}$ particles} In the previous part of the proof, we looked only at the second moment $\mathds{E}\left[N\langle \mathbf{q},u_k(t)\rangle^2\vert\bm{\lambda}\right]$ which corresponds to a single particle in the site $k$. Now, we will do the proof of \eqref{resultbulk} by induction on the number of particles.\\
We can first define the same objects as the single particle case: we will consider the short range dynamics for a small time $\tau\ll t$ with initial condition an average of the eigenvectors moment of $W_t$ localized onto a specific window. More precisely define, with $\bm{\xi}_0$ being the configuration with $n$ particles that lies at the center of our window of size $w$ in the sense of the distance \eqref{eq:distance},
\begin{align}
\partial_{\tau}g_{\tau}(\bm{\xi})&=\frac{1}{N}\sum_{\substack{\vert i-j\vert\leqslant \ell\\i\neq j}}\frac{g_{\tau}(\bm{\xi}^{ij})-g_{\tau}(\bm{\xi})}{(\lambda_i-\lambda_j)^2},\\
g_0(\bm{\xi})&=\left(\text{Av}_{\bm{\xi_0}}f_t\right)(\bm{\xi}).
\end{align}

By the same reasoning as for the one particle case, using Lemmas \ref{shortrange} and \ref{finitespeed} with $n$ particles, we get
\begin{equation}
\vert f_{\tau}(\bm{\xi})-g_{\tau}(\bm{\xi})\vert\leqslant N^{(n+4)\varepsilon} \frac{CN\tau}{\ell}\sigma_t^2(\mathbf{q},\bm{\xi},\eta).
\end{equation}

To reason by induction on the number of particles, we need to show the following equation, similar to \eqref{g} in the case of one particle. For $k_r\in \mathcal{A}^\kappa_{r}$, define $z^{(k_r)}=\gamma_{k_r}+i\eta$ and let $\bm{\xi}$ a configuration of $n$ particles with at least one particle in $k_r$, we need to show
\begin{equation}\label{npartlem}
\begin{split}
\mathrm{Im}\left(\frac{1}{N}\sum_{k=1}^N\frac{g_{\tau}\left(\bm{\xi}^{k_r,k}\right)}{\lambda_k(t+\tau)-z^{(k_r)}}\right)-&\left(a_{\bm{\xi}}f_{\tau}(\bm{\xi}\setminus k_r)\sigma_t(\mathbf{q},k_r,\eta)^2+(1-a_{\bm{\xi}})\sigma_t^2(\mathbf{q},\bm{\xi_0},\eta)\right)\\=&\mathcal{O}\left(N^{(n+4)\varepsilon}\left(\frac{\ell}{w}+\frac{N\tau}{\ell}+\frac{1}{\sqrt{N\eta}}+\frac{\tau}{t}\right)\sigma_t^2(\mathbf{q},\bm{\xi_0},\eta)\right).
\end{split}
\end{equation}
where $\bm{\xi}\setminus k_r$ denote the configuration where we removed one particle in $k_r$ from $\bm{\xi}.$\\
\indent We apply the same decomposition in three terms as in the single particle case, the first two terms can be bounded the same way and we can bound the left hand side of \eqref{npartlem} by
\begin{multline}\label{eq:lasteq}
\left\vert\mathrm{Im}\left(\frac{1}{N}\sum_{k=1}^N
\frac{\left(\text{Av}_{\xi_w}f_{\tau}\right)\left(\bm{\xi}^{k_r,k}\right)}{\lambda_k-z^{(k_r)}}\right)\right.\\
\left.-\left(a_{\bm{\xi}}f_t(\bm{\xi}\setminus k_r)\mathrm{Im}(m_t(z^{(k_r)}))\sigma_t^2(\mathbf{q},k_r,\eta)+(1-a_{\bm{\xi}})\mathrm{Im}(m_t(z^{(k_r)}))\sigma_t^2(\mathbf{q},\bm{\xi_0},\eta)\right)\right\vert
\\+
\mathcal{O}\left(
	N^{(n+4)\varepsilon}
	\left(
		\frac{\ell}{u}+\frac{N\tau}{\ell}
	\right)
	\sigma_t^2(\mathbf{q},\bm{\xi_0},\eta)
\right)
\end{multline}

Now we need to see that,
\begin{align*}
\text{Av}_{\bm{\xi}_0}f_{\tau}(\bm{\xi}^{k_r,k})&=a_{\bm{\xi}^{k_r,k}}f_{\tau}(\bm{\xi}^{k_r,k})+(1-a_{\bm{\xi}^{k_r,k}})\sigma_t^2(\mathbf{q},\bm{\xi_0},\eta)\\
&=\left(a_{\bm{\xi}}f_{\tau}(\bm{\xi}^{k_r,k})+(1-a_{\bm{\xi}})\sigma_t^2(\mathbf{q},\bm{\xi_0},\eta)\right)\\
&+\left((a_{\bm{\xi}^{k_r,k}}-a_{\bm{\xi}})f_{\tau}({\bm{\xi}^{k_r,k}})+(a_{\bm{\xi}}-a_{\bm{\xi}^{k_r,k}})\sigma_t^2(\mathbf{q},\bm{\xi_0},\eta)\right)\\
&=\left(a_{\bm{\xi}}f_{\tau}(\bm{\xi}^{k_r,k})+(1-a_{\bm{\xi}})\sigma_t^2(\mathbf{q},\bm{\xi_0},\eta)\right)+\mathcal{O}\left(\frac{d(\bm{\xi}^{k_r,k},\bm{\xi})\sigma_t^2(\mathbf{q},\bm{\xi_0},\eta)}{N}\right).
\end{align*}
We can use the same decomposition into $\vert k-k_r\vert\leqslant N\sqrt{\eta}$ and the averaged local law from Theorem \ref{local} to get
\begin{multline}\label{410}
\mathrm{Im}\left(\frac{1}{N}\sum_{k=1}^N
\frac{\text{Av}_{\xi_w}f_{\tau}\left(\bm{\xi}^{k_r,k}\right)}{\lambda_k-z^{(k_r)}}\right)=a_{\bm{\xi}}\mathrm{Im}\left(\frac{1}{N}\sum_{k=1}^N\frac{f_{\tau}(\bm{\xi}^{k_r,k})}{\lambda_k-z^{(k_r)}}\right)\\+(1-a_{\bm{\xi}})\mathrm{Im}(m_t(z^{(k_r)}))\sigma_t^2(\mathbf{q},\bm{\xi_0},\eta)+\mathcal{O}\left(\frac{N^\varepsilon}{N\eta}\sigma_t^2(\mathbf{q},\bm{\xi_0},\eta)\right).
\end{multline}
Consider the sum in the right hand side, recall that there is at least one particle in $k_r$ and denote $k_1,\dots,k_m$ with $m\leqslant n,$ the sites where there is at least one particle in the configuration $\bm{\xi}$. Recall that $z^{(k_r)}=\gamma_{k_r,t}+i\eta$,
\begin{align}
\frac{1}{N}&\sum_{k=1}^N\frac{\eta f_{\tau}(\bm{\xi}^{k_r,k})}{(\gamma_{k_r,t}-\lambda_k)^2+\eta^2}=\frac{1}{N}\sum_{k\notin\{k_1,\dots k_m\}}\frac{\eta f_{\tau}(\bm{\xi}^{k_r,k})}{(\gamma_{k_r,t}-\lambda_k)^2+\eta^2}+\mathcal{O}\left(\frac{N^{n\varepsilon}}{N\eta}\sum_{i=1}^m\sigma_t^2(\mathbf{q},\bm{\xi}^{k_r,k_i},\eta)\right)
\end{align}
where we used Corollary \ref{coro:normephi} on the indices we removed from the sum. Now we have the following equality for the first sum by definition of $f_\tau$,
\begin{multline}
\label{413}
\frac{1}{N}\sum_{k\notin\{k_1,\dots k_m\}}\frac{\eta f_{\tau}(\bm{\xi}^{k_r,k})}{(\gamma_{k_r,t}-\lambda_k)^2+\eta^2}
\\=\mathds{E}\left[\!\left.\prod_{\substack{1\leqslant r'\leqslant m \\r'\neq r}}\hspace{-.2cm}\frac{z_{r'}^{2j_{r'}}}{a(2j_{r'})}\frac{z_r^{2(j_r-1)}}{a(2(j_r-1))}\hspace{-.1cm}\sum_{j\notin\{k_1,\dots,k_m\}}\hspace{-1em}\frac{\eta z_j^2}{N((\gamma_{k_r,t}-\lambda_j)^2+\eta^2)}\right\vert\bm{\lambda}\right]
\end{multline}
By Lemma \ref{lem:localtau}, we have,
\begin{align}
\frac{1}{N}&\sum_{j\notin\{k_1,\dots,k_m\}}\frac{\eta z_j^2}{(\gamma_{k_r,t}-\lambda_j)^2+\eta^2}=\frac{1}{N}\sum_{k=1}^N\frac{\eta z_k^2}{(\gamma_{k_r,t}-\lambda_k)^2+\eta^2}+\mathcal{O}\left(\frac{N^\varepsilon}{N\eta }\sum_{j=1}^m\sigma^2_t(\mathbf{q},k_j,\eta)\right),\\
&=\mathrm{Im}\left(\frac{1}{N}\sum_{k=1}^N\frac{z_k^2}{\lambda_k-z^{(k_r)}}\right)+\mathcal{O}\left(\frac{N^\varepsilon}{N\eta }\sum_{j=1}^m\sigma^2_t(\mathbf{q},k_j,\eta)\right),\\
&\label{417}=\mathrm{Im}(m_t(z^{(k_r)}))\sigma_{t}^2(\mathbf{q},k_r)+\mathcal{O}\left(N^\varepsilon\left(\frac{1}{\sqrt{N\eta}}+\frac{\tau}{t}\right)\sigma_t^2(\mathbf{q},k_r)+\frac{N^\varepsilon}{N\eta }\sum_{j=1}^m\sigma^2_t(\mathbf{q},k_j,\eta)\right).
\end{align}

Combining \eqref{417} and \eqref{413} and using the bounds on the variations of $\sigma_t$ \eqref{eq:variationsigma}, we get
\begin{multline}\label{418}
\frac{1}{N}\sum_{k=1}^N\frac{\eta f_{\tau}(\bm{\xi}^{k_r,k})}{(\gamma_{k_r,t}-\lambda_k)^2+\eta^2}=\mathrm{Im}(m_t(z^{(k_r)}))\sigma_t(\mathbf{q},k_r,\eta)^2f_{\tau}(\bm{\xi}\setminus k_r)
\\+\mathcal{O}\left(N^\varepsilon\left(\frac{1}{\sqrt{N\eta}}+\frac{\tau}{t}\right)\sigma_t^2(\mathbf{q},\bm{\xi},\eta)\right)
\end{multline}

Finally, combining \eqref{418} and \eqref{410}, we have
\begin{multline*}
\mathrm{Im}\left(\frac{1}{N}\sum_{k=1}^N
\frac{\text{Av}_{\xi_0}f_{\tau}\left(\bm{\xi}^{k_r,k}\right)}{\lambda_k-z^{(k_r)}}\right)= a_{\bm{\xi}}f_t(\bm{\xi}\setminus k_r)\mathrm{Im}(m_t(z^{(k_r)}))\sigma_t^2(\mathbf{q},k_r,\eta)\\+(1-a_{\bm{\xi}})\mathrm{Im}(m_t(z^{(k_r)}))\sigma_t^2(\mathbf{q},\bm{\xi_0},\eta)+\mathcal{O}\left(N^\varepsilon\left(\frac{1}{\sqrt{N\eta}}+\frac{\tau}{t}\right)\sigma_t^2(\mathbf{q},\bm{\xi},\eta)\right)
\end{multline*}
which, combined with \eqref{eq:lasteq}, gives us \eqref{npartlem}.

We now follow the same proof as in the case of one particle : we state a maximum principle on the flattened and averaged moment. First define
\begin{equation}
\bm{\xi_m}=\max_{\substack{\bm{\xi}\\\mathcal{N}(\bm{\xi})=n}}g_{\tau}(\bm{\xi}),
\end{equation}
and let $k_1,\dots,k_m$ be the positions of the particles of the configuration $\bm{\xi_m}$ with $m\leqslant n.$ We are going to use our induction hypothesis in the maximum principle inequalities by \eqref{npartlem}.
\begin{align}
\partial_{\tau}&g_{\tau}(\bm{\xi_m})\leqslant\frac{C}{N}\sum_{\substack{\vert i-j\vert\leqslant l \\ i\neq j}}\frac{g_{\tau}(\bm{\xi_m}^{i,j})-g_{\tau}(\bm{\xi_m})}{(\lambda_i-\lambda_j)^2}\nonumber\\
&\leqslant \frac{C}{N}\sum_{r=1}^m\left(\frac{1}{\eta}\sum_{\substack{\vert j-k_r\vert\leqslant l \\j\neq k_r}}\frac{\eta g_{\tau}(\bm{\xi_m}^{k_r,j})}{(\lambda_j-\lambda_{k_r})^2+\eta^2}-\frac{g_{\tau}(\bm{\xi_m})}{\eta}\sum_{\substack{\vert j-k_r\vert\leqslant l\\j\neq k_r}}\frac{\eta}{(\lambda_j-\lambda_{k_r})^2+\eta^2}\right)\nonumber\\
&\leqslant \frac{C}{\eta}\sum_{r=1}^m\left[\mathrm{Im}\left(\frac{1}{N}\sum_{j=1}^N\frac{g_{\tau}(\bm{\xi_m}^{k_r,j})}{\lambda_j-z^{(k_r)}}\right)-{g_{\tau}(\bm{\xi_m})}\mathrm{Im}\left(s_\tau(z^{(k_r)})\right)\right]+\mathcal{O}\left(\frac{N\eta}{\ell}\frac{\sigma_t^2(\mathbf{q},\bm{\xi_m},\eta)}{\eta}\right)\nonumber\\
\label{avantind}&\leqslant \frac{C}{\eta}\sum_{r=1}^m\left({a_{\bm{\xi_m}}}f_{\tau}(\bm{\xi_m}\setminus k_r)\mathrm{Im}(m_t(z^{(k_r)}))\sigma_t^2(\mathbf{q},k_r,\eta)+{(1-a_{\bm{\xi_m}})}\mathrm{Im}(m_t(z^{(k_r)}))\sigma_t^2(\mathbf{q},\bm{\xi_0},\eta)\right)\\
&-\frac{C}{\eta}\sum_{r=1}^mg_{\tau}(\bm{\xi_m})\mathrm{Im}(m_t(z^{(k_r)}))+
\mathcal{O}\left(
	\frac{N^{(n+4)\varepsilon}}{\eta}
	\left(
		\frac{1}{\sqrt{N\eta}}+\frac{\tau}{t}+\frac{\ell}{w}+\frac{N\tau}{\ell}+\frac{N\eta}{\ell}
	\right)
	\sigma^2_t(\mathbf{q},\bm{\xi_m},\eta)
\right).\nonumber
\end{align}

Now, we use the induction assumption on $f_{\tau}(\bm{\xi_m}\setminus k_r)$ which is a $(n-1)$th moment and obtain
\begin{align}\label{induction}
f_{\tau}(\bm{\xi_m}\setminus k_r)\sigma_t(\mathbf{q},k_r,\eta)^2=\sigma_t^2(\mathbf{q},\bm{\xi_m},\eta)+\mathcal{O}\left(\left(\frac{1}{\sqrt{N\tau}}+\left(\frac{\tau}{t}\right)^{1/3}\right)\sigma^2_t(\mathbf{q},\bm{\xi_m},\eta)\right).
\end{align}

Besides, we can easily see that, since $d(\bm{\xi_w},\bm{\xi_m})\leqslant 2w$, from \eqref{eq:variationsigma},
\begin{equation}\label{sigma2}
\left\vert\sigma_t^2(\mathbf{q},\bm{\xi_0},\eta)^2-\sigma_t^2(\mathbf{q},\bm{\xi_m},\eta)^2\right\vert\leqslant\frac{N^\varepsilon w}{Nt}\sigma_t^2(\mathbf{q},\bm{\xi_m},\eta).
\end{equation} 

Now, injecting \eqref{induction} and \eqref{sigma2} in \eqref{avantind}, we get
\begin{multline*}
\partial_{\tau}\left(g_{\tau}(\bm{\xi_m}) - \sigma_t^2(\mathbf{q},\bm{\xi_0},\eta)\right)\leqslant-\frac{C}{\eta}\sum_{r=1}^m\left(g_{\tau}(\bm{\xi_m})-\sigma_t^2(\mathbf{q},\bm{\xi_0},\eta)\right)\\
+\mathcal{O}\left(\left(\frac{N\tau}{\ell}+\frac{\ell}{w}+\frac{1}{\sqrt{N\eta}}+\frac{\tau}{t}+\frac{N\eta}{\ell}+\frac{w}{Nt}+\frac{1}{\sqrt{N\tau}}+\left(\frac{\tau}{t}\right)^{1/3}\right)\frac{N^{(n+6)\varepsilon}}{\eta}\sigma_t^2(\mathbf{q},\bm{\xi_m},\eta)\right).
\end{multline*}

Doing the same reasoning as in the proof for one particle, we get, by applying Gronwall's lemma,
\begin{multline*}
g_{\tau}(\bm{\xi_m})=\sigma_t^2(\mathbf{q},\bm{\xi_m},\eta)\\
+\mathcal{O}\left(\left(\frac{N\tau}{\ell}+\frac{\ell}{w}+\frac{1}{\sqrt{N\eta}}+\frac{\tau}{t}+\frac{N\eta}{\ell}+\frac{w}{Nt}+\left(\frac{1}{\sqrt{N\tau}}+\left(\frac{\tau}{t}\right)^{1/3}\right)\right)\sigma_t^2(\mathbf{q},\bm{\xi_m},\eta)\right)
\end{multline*}

We can again do the same reasoning with the infimum and choosing the parameters as in \eqref{eq:parameter} the claim from Proposition \ref{prop:moment} follows.
\subsection{Analysis of the perfect matching observable}\label{subsec:perfect}
In this subsection we will again condition on the event $A$ of good eigenvalue paths where the local laws and the finite speed of propagation holds. Consider now a deterministic set of indices $I\subset[\![1,N]\!]$. Note that in the definition of the centered overlaps $p_{ii}$, we can only center by a constant not depending on $i$. However in Theorem \ref{theo:resultQUE}, one can see that the expectation of the probability mass of the $i$th-eigenvector on $I$ clearly depends on $i$. Thus, we will need to localize our perfect matching observables onto a window of size $w$ chosen later and show that these $p_{ii}$ are, up to an error depending on $w$, centered around the same constant. The size of the window $w$ will be taken so that $N\tau \ll w\ll Nt$ similarly to the previous section. More precisely, we will fix an integer $i_0\in\mathcal{A}_r^\kappa$ and consider the set of indices 
\[
\mathcal{A}_{w}^\kappa(i_0) = \left\{ i\in[\![1,N]\!],\,\gamma_{t,i}\in[\gamma_{t,i_0}-(1-\kappa)w,\gamma_{t,i_0}+(1-\kappa)w]\right\}
\]
so that we will take for our centered diagonal overlaps
\begin{equation}\label{eq:pii}
p_{ii}=\sum_{\alpha\in I}u_i(\alpha)^2-C_0\quad\text{with}\quad C_0=\frac{1}{N}\sum_{\alpha\in I}\sigma_t^2(\mathbf{e}_{\alpha},i_0),
\end{equation}
the overlaps for $i\neq j$ will not change.  First consider these overlaps for the matrix $W_t$ and define, similarly to the previous subsection
\[
\Phi_{t}(\bm{\xi})=\frac{1}{\mathcal{M}(\bm{\xi})}\mathds{E}\left[\sum_{G\in\mathcal{G}_{\bm{\xi}}}P(G)\middle\vert\bm{\lambda}\right].
\]
This quantity corresponds the perfect matching observable for our initial matrix $W_t$ and we make it undergo the dynamics \eqref{eq:discs2deformed} so that we define 
\[
F_s(\bm{\xi})=U_{\mathscr{B}}(0,s)\Phi_t(\bm{\xi})
\] 
where $\mathscr{B}$ is defined in \eqref{eq:defdyna}. We now prove the result from Theorem \ref{theo:resultQUE} for a Gaussian divisible ensemble for $p_{i_0i_0}$. We will need the following technical lemma allowing us to bound the $p_{ij}$ by the perfect matching observables.
\begin{lemma}[\cite{bourgade2018random}*{Lemma 3.6}]\label{lem:holderdeformed}
Take an even integer $n$, there exists $C>0$ depending on $n$ such that for any $i<j$ and any time $s$ we have
\begin{equation}
\mathds{E}\left[p_{ij}(s)^n\middle\vert\bm{\lambda}\right]\leqslant C\left(F_s(\bm{\xi}^{(1)})+F_s(\bm{\xi}^{(2)})+F_s(\bm{\xi}^{(3)})\right)
\end{equation}
where $\bm{\xi}^{(1)}$ is the configuration of $n$ particles in the site $i$, $\bm{\xi}^{(2)}$ n particles in the site $j$, and $\bm{\xi}^{(3)}$ an equal number of particles between the site $i$ and $j$. 
\end{lemma}
We will also use repeatedly the following bound on the eigenvectors which comes from Corollary \ref{coro:normephi}
\[
\sum_{k\in I}u_k(\alpha)^2\leqslant N^\varepsilon\widehat{I}.
\]
The purpose of this section is to prove Theorem \ref{theo:resultQUE} for another matrix ensemble: a deformed Wigner matrix perturbed by a small Gaussian component. More precisely, we state it as the following theorem.

\begin{theorem}\label{theo:resultint}
Consider $\mathfrak{a}$ and $\omega$ two small positive constants and $\kappa\in(0,1)$, take $t\in\mathcal{T}_\omega$, $D$ a deterministic diagonal matrix given by Definition \ref{def:init} and $W$ a Wigner matrix given by Definition \ref{def:wignernonsmooth}. Let $\tau\in\mathcal{T'}_{\mathfrak{a}}:=[N^{-1+\mathfrak{a}},\,N^{-\mathfrak{a}}t]$, then if $u_1,\dots,u_N$ are the eigenvectors of the matrix $$H_\tau = D+\sqrt{{t}}W+\sqrt{\tau}\mathrm{GOE},$$
define the error
\[
\Xi(\tau)=\frac{\widehat{I}}{\sqrt{N\tau}}+\widehat{I}\frac{\tau}{t},
\]
 we have, for any $k,\ell\in \mathcal{A}_{r}^\kappa$ with $k\neq\ell$ and any $\varepsilon>0$ and $D>0$,
\[
\mathds{P}\left(
		\left\vert
		\sum_{\alpha\in I}
		\left(
			u_k(\alpha)^2-\frac{1}{N}\sigma_t^2(\alpha,k,\tau)
		\right)
	\right\vert
	+
	\left\vert
		\sum_{\alpha\in I}u_k(\alpha)u_\ell(\alpha)
	\right\vert
	\geqslant
	N^\varepsilon\Xi(\tau)
\right)\leqslant N^{-D}.
\] 
\end{theorem}
As one can see in the statement of Theorem \ref{theo:resultint}, we will need the small time $\tau$ corresponding to the size of the Gaussian perturbation to be of order smaller than $t$ so that the eigenvalues and eigenvectors barely changed during that time. We will later choose a specific $\tau$ and optimize all our different parameters when using the reverse heat flow technique to remove this small Gaussian component. One of these parameters will be a cut-off for the dynamics as in Subsection \ref{subsec:moment}. Indeed, in order to use a maximum principle on the dynamics, we will split it in the same way: a short-range dynamics with generator $\mathscr{S}$ that will contain most of the information and a long-range part with generator $\mathscr{L}$ we need to control as in \cite{erdos2015gap} where $\mathscr{S}$ and $\mathscr{L}$ are defined respectively in \eqref{eq:defshort} and \eqref{eq:deflong}.
Lemma \ref{finitespeed} will help us localize the dynamics onto a small set of configurations. Now the following lemma says that most of the information of the dynamics is given by the short-range, bounding the difference between $\mathscr{B}$ and $\mathscr{S}$. It is analogous to Lemma 3.5 in \cite{bourgade2018random}. First define
$$S_I^{(u,v)}=\sup_{\bm{\xi}\subset  I,u\leqslant s\leqslant v}F_s(\bm{\xi}).$$

\begin{lemma}\label{lem:shortlong}
For any intervals $J_{\mathrm{in}}\subset \mathcal{A}_w(i_0)$ and $J_{\mathrm{out}}=\{i,\,d(i,J_{\mathrm{in}})\leqslant N^\varepsilon \ell\}\subset\mathcal{A}_{w}^\kappa(i_0)$ since we will take $\ell\ll Nw$, any configuration $\bm{\xi}$ such that $\mathcal{N}(\bm{\xi})=n$ supported on $J_{\mathrm{in}}$ and any $N^{-1}\ll\tau\ll w$ we have
\begin{equation}
\left\vert((U_\mathscr{B}(0,\tau)-U_\mathscr{S}(0,\tau))\Phi_t(\bm{\xi})\right\vert\leqslant N^\varepsilon \frac{N\tau}{\ell}\left(S_{J_{\mathrm{out}}}^{(0,\tau)}+\widehat{I}\frac{w}{t}\left(S_{J_{\mathrm{out}}}^{(0,\tau)}\right)^{\frac{n-1}{n}}+\frac{\widehat{I}}{\ell}\left(S_{J_{\mathrm{out}}}^{(0,\tau)}\right)^{\frac{n-2}{n}}\right)
\end{equation}
where $F_\tau$ is the perfect matching observable defined in \ref{lem:eq:perfobsdeformed}.
\end{lemma}
This bound is used in this form so we can obtain information on a box in space by extracting information from a larger box. Iterating this bound will give us Theorem \ref{theo:resultint}.
\begin{proof}
We will follow the proof from \cite{bourgade2018random}. Define the following flattening operator. For $f$ a function on configurations of $n$ particles and $\bm{\xi}$ such a configuration,
\begin{equation}
(\mathrm{Flat}_af)(\bm{\xi})=\left\{\begin{array}{ll}
										f(\bm{\xi})\quad\text{if}\quad\bm{\xi}\subset\left\{i,\,d(i,J_{\mathrm{in}})\leqslant a\right\},\\
										0\quad\mathrm{otherwise}.
									\end{array}
								\right.							
\end{equation}
 We make the functions vanish outside of a certain interval. We use now Duhamel's formula and write
 $$\left((U_\mathscr{S}(0,\tau)-U_\mathscr{B}(0,\tau)\Phi_t\right)(\bm{\xi})=\int_0^{\tau}U_{\mathscr{S}}(s,\tau)\mathscr{L}(s)F_s(\bm{\xi})\D s.$$
 
 Now, see that, by definition of the flattening operator and the fact that $\bm{\xi}$ is supported on $J_{\mathrm{in}},$
 $$d(\mathrm{Supp}\left(\mathscr{L}(s)F_s-\mathrm{Flat}_{N^\varepsilon \ell}(\mathscr{L}(s)F_s)\right),\bm{\xi})\geqslant N^{\varepsilon}\ell.$$
 With this bound, we can use the finite speed of propagation from Lemma \ref{finitespeed} and obtain, using that $U_\mathscr{S}$ is a contraction in $\ell^\infty$,
 $$\left\vert U_\mathscr{S}(s,\tau)\mathscr{L}(s)F_s(\bm{\xi})\right\vert\leqslant \max_{\tilde{\bm{\xi}}\subset J_{\mathrm{out}}}\left\vert\left(\mathrm{Flat}_{N^\varepsilon \ell}(\mathscr{L}(s)F_s)\right)(\tilde{\bm{\xi}})\right\vert+\mathcal{O}\left(e^{-cN^\varepsilon/2}\right). $$
Thus we need to control $\vert(\mathscr{L}(s)F_s)(\tilde{\bm{\xi}})\vert$ for $\tilde{\bm{\xi}}=\left\{(i_1,j_1),\dots(i_m,j_m)\right\}$ a configuration of $n\geqslant m$ particles supported in $J_{\mathrm{out}}$. 

We have, by definition of $\mathscr{L},$
$$\left\vert\mathscr{L}(s)F_s(\tilde{\bm{\xi}})\right\vert\leqslant C_n \left\vert\sum_{1\leqslant p\leqslant m}\sum_{\vert i_p-k\vert\geqslant \ell}\frac{F_s(\tilde{\bm{\xi}}^{i_p,k})}{N(\lambda_{i_p}-\lambda_k)^2}\right\vert+C_n\left\vert F_s(\tilde{\bm{\xi}})\right\vert\sum_{\substack{1\leqslant p\leqslant m}}\sum_{\vert i_p-k\vert>\ell}\frac{1}{N(\lambda_{i_p}-\lambda_k)^2}.$$
For the second term in the previous inequality, we can use rigidity estimates from Theorem \ref{theo:rigidity} and a dyadic decomposition and see that
\begin{equation}\label{eq:dyadic}
\sum_{k,\,\vert i_p-k\vert >\ell}\frac{1}{N(\lambda_{i_p}-\lambda_k)^2}\leqslant N^\varepsilon \frac{N}{\ell},
\end{equation}
so that we have the bound
\begin{equation}\label{eq:Lbound}
\left\vert \mathscr{L}(s)F_s(\tilde{\bm{\xi}})\right\vert\leqslant N^\varepsilon \left\vert\sum_{1\leqslant p \leqslant m}\sum_{k,\vert i_p-k\vert\geqslant \ell}\frac{F_s(\tilde{\bm{\xi}}^{i_p,k})}{N(\lambda_{i_p}-\lambda_k)^2}\right\vert+N^\varepsilon\frac{N}{\ell}S_{J_{\mathrm{out}}}^{(0,\tau)}.
\end{equation}
For the first sum in \eqref{eq:Lbound}, we will first restrict it to the sites $k$ such that there are no particles in the configuration $\tilde{\bm{\xi}}$, so that we will have $\tilde{{\eta}}^{i_p, k}_k=1$. Note that, if we have $\tilde{{\eta}}_k\neq 0$ then by definition of $\tilde{\bm{\xi}}$ supported on $J_{\mathrm{out}}$, $\tilde{\bm{\xi}}^{i_p,k}$ is also supported on $J_{\mathrm{out}}.$ This gives us the bound
\begin{align*}
\sum_{k,\,\vert i_p-k\vert>\ell}\frac{F_s(\tilde{\bm{\xi}}^{i_p,k})}{N(\lambda_{i_p}-\lambda_k)^2}&\leqslant\sum_{\substack{k,\,\vert i_p-k\vert>\ell\\\tilde{\eta}_k=0}}\frac{F_s(\tilde{\bm{\xi}}^{i_p,k})}{N(\lambda_{i_p-\lambda_k)^2}}+S_{J_{\mathrm{out}}}^{(0,\tau)}\sum_{\substack{k,\,\vert i_p-k\vert >\ell\\\tilde{\eta}_k\neq 0}}\frac{1}{N(\lambda_{i_p}-\lambda_k)^2}.\\
&\leqslant \sum_{\substack{k,\,\vert i_p-k\vert>\ell\\\tilde{\eta}_k=0}}\frac{F_s(\tilde{\bm{\xi}}^{i_p,k})}{N(\lambda_{i_p}-\lambda_k)^2}+CN^\varepsilon\frac{N}{\ell^2}S_{J_{\mathrm{out}}}^{(0,\tau)}
\end{align*}
where in the second inequality we used the rigidity of the eigenvalues, which gives us $(\lambda_{i_p}-\lambda_k)^2\geqslant C(\ell/N)^2$ for $\vert i_p-k\vert >\ell$, and the fact that there is at most $m$ sites $k$ such that $\tilde{\eta}_k\neq 0$. By definition of the perfect matching observables, we can write

$$\sum_{\substack{k,\,\vert i_p-k\vert>\ell\\\tilde{\eta}_k=0}}\frac{F_s(\tilde{\bm{\xi}}^{i_p,k})}{N(\lambda_{i_p}-\lambda_k)^2}=C(n)\mathds{E}\left[\sum_{\substack{k,\,\vert i_p-k\vert>\ell\\\tilde{\eta}_k=0}}\sum_{G\in\mathcal{G}_{\tilde{\bm{\xi}}^{i_p,k}}
}\frac{\prod_{e\in\mathcal{E}(G)}p(e)}{N(\lambda_{i_p}-\lambda_k)^2}\middle|\bm{\lambda}\right].$$
In order to control this term, we will consider two different types of perfect matchings. Define the following partition of $\mathcal{G}_{\bm{\xi}}$ into two subsets
\begin{align}
\label{eq:defg1}\mathcal{G}^{(1)}_{\bm{\xi}}&=\left\{G\in\mathcal{G}_{\bm{\xi}},\,\{(k,1),(k,2)\}\in\mathcal{E}(G)\right\},\\
\label{eq:defg2}\mathcal{G}^{(2)}_{\bm{\xi}}&=\left\{G\in\mathcal{G}_{\bm{\xi}},\,\{(k,1),(k,2)\}\notin\mathcal{E}(G)\right\}.
\end{align}

\begin{figure}[H]
\centering
\begin{subfigure}[t]{.5\textwidth}
  \centering
  \includegraphics[width=.7\linewidth]{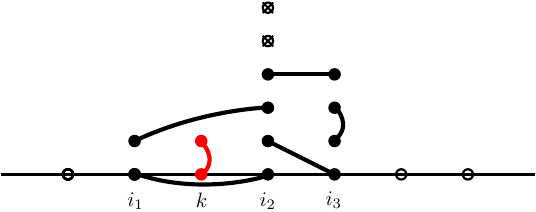}
  \caption{A perfect matching from $\mathcal{G}^{(1)}_{\bm{\xi}^{i_2,k}}$}
\end{subfigure}%
\begin{subfigure}[t]{.5\textwidth}
  \centering
  \includegraphics[width=.7\linewidth]{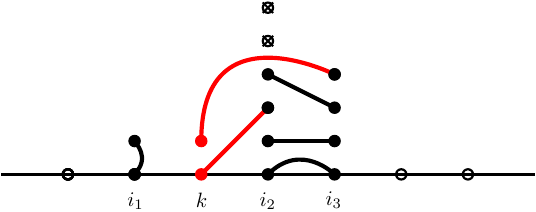}
  \caption{A perfect matching from $\mathcal{G}^{(2)}_{\bm{\xi}^{i_2,k}}$}
\end{subfigure}
\end{figure}

We will begin by bounding the contribution from \eqref{eq:defg1}. Note first that, for $G\in\mathcal{G}^{(1)}_{\bm{\xi}}$, we have
\begin{equation}\label{eq:prodg1}\prod_{e\in\mathcal{E}(G)}p(e)=p_{kk}\times Q_1((p(e))_{e\in\mathcal{E}(G)})
\end{equation}
\text{with}
$$Q_1((p(e))_{e\in\mathcal{E}(G)})=\prod_{e\in\mathcal{E}(G)\setminus\{(k,1),(k,2)\}}p(e).$$
See also that $Q_1$ is a monic monomial of degree $n-1$ so that we can use Lemma \ref{lem:holderdeformed} and obtain
\begin{equation}\label{eq:boundq1}
\mathds{E}\left[\sum_{G\in\mathcal{G}_{\bm{\xi}\setminus i_p}}Q_1((p(e))_{e\in\mathcal{E}(G)})\middle|\bm{\lambda}\right]\leqslant C\sup_{\substack{0\leqslant s \leqslant \tau\\\bm{\xi}\subset J_{\text{out}},\,\mathcal{N}(\bm{\xi})=n-1}}\left\vert F_s(\bm{\xi})\right\vert\leqslant C\left(S_{J_{\mathrm{out}}}^{(0,\tau)}\right)^{\frac{n-1}{n}}.
\end{equation}

Combining \eqref{eq:prodg1} and \eqref{eq:boundq1}, we now only need to bound
$$\sum_{\substack{k,\,\vert k-i_p\vert>l,\\\tilde{\eta}_k=0}}\frac{p_{kk}}{N(\lambda_{i_p}-\lambda_k)^2}=\sum_{k,\,\vert k-i_p\vert>\ell}\frac{p_{kk}}{N(\lambda_{i_p}-\lambda_k)^2}+\mathcal{O}\left(\frac{N}{\ell^2}\right).
$$
In order to bound the sum from the right hand side of the previous equation, first define the following functions, for $\vert z-\lambda_{i_p}\vert\leqslant N^{-\varepsilon}\ell/N,$
\begin{align*}
f(z)&=\sum_{k,\,\gamma_{t,k}\notin [E_1^-,E_2^+]}\frac{p_{kk}}{N(z-\lambda_k)},\\
g(z)&=\sum_{k,\,\gamma_{t,k}\notin [E_1^-,E_1^+]\cup[E_2^-,E_2^+]}\frac{p_{kk}}{N(z-\lambda_k)}
\end{align*} 
where $E_1=\gamma_{t,i_p-\ell},$ $E_1^-=\gamma_{t,i_p-\ell-N^\varepsilon}$, $E_1^+=\gamma_{t,i_p-\ell+N^{\varepsilon}},$ $E_2=\gamma_{t,i_p+\ell},$ $E_2^-=\gamma_{t,i_p+\ell-N^\varepsilon}$, $E_2^+=\gamma_{t,i_p+\ell+N^{\varepsilon}}.$ Let also $\Gamma$ be the rectangle with vertices $E_1\pm \mathrm{i}\ell/N$ and $E_2\pm \mathrm{i}\ell/N$. We therefore want to bound, up to a $N^\varepsilon$ term,
$$\sum_{k,\,\vert k-i_p\vert>\ell}\frac{p_{kk}}{N(\lambda_{i_p}-\lambda_k)^2}=\partial_z f(z)\Bigr\vert_{ z=\lambda_{i_p}}$$
Now consider $\mathcal{C}_{i_p}$ the circle centered in $\lambda_{i_p}$ with radius $N^{-\varepsilon}\frac{\ell}{N}$, then by Cauchy's formula, we can write,
$$\partial_zf(\lambda_{i_p})=\frac{1}{2\mathrm{i}\pi}\int_{\mathcal{C}_{i_p}}\frac{f(z)}{(z-\lambda_{i_p})^2}\mathrm{d}\xi.$$
By using another Cauchy integral formula on the contour $\Gamma$ for $f$ and seeing that for $\lambda_{\mathrm{int}}$, $z$ inside the contour and $\lambda_{\mathrm{ext}}$ outside of the contour we have, by a residue calculus,
$$\int_\Gamma \frac{\mathrm{d}\xi}{(\xi - \lambda_{\mathrm{int}})(\xi-\lambda_{i_p})}=\int_\Gamma \frac{\mathrm{d}\xi}{(\overline{\xi}-\lambda_\mathrm{ext})(\xi-\lambda_{i_p})}=0$$.
Note that we also have
\[
\frac{1}{2\pi}\int_\gamma\frac{p_{kk}}{(\xi-\lambda_{i_p})^2(\xi-\lambda_k)}
=
\mathcal{O}\left(
	\frac{1}{2\pi}\int_\gamma \frac{p_{kk}}{(\xi-\lambda_{i_p})^2(\overline{\xi}-\lambda_k)}
\right)
\]
so that we can write
\begin{equation}\label{eq:intcontour2}
\left\vert \partial_z f(\lambda_{i_p})\right\vert = \left\vert\frac{1}{2\pi}\int_{\Gamma}\frac{g(\xi)}{(\xi-\lambda_{i_p})^2}\mathrm{d}\xi\right\vert
=
\mathcal{O}\left(\left\vert\int_\Gamma\frac{\mathrm{Im} g(\xi)}{(\xi-\lambda_{i_p})^2}\mathrm{d}\xi\right\vert \right).
\end{equation}
We will first control the part of the contour closest to the real axis. Consider 
$$\Gamma_1=\{ z=E+\mathrm{i}\eta\in\Gamma,\,\vert \eta \vert < N^\varepsilon/N\},$$ as in \cite{bourgade2018random} and bounding $p_{kk}$ by $\widehat{I}$, we obtain
$$\left\vert \int_{\Gamma_1}\frac{\mathrm{Im}g(\xi)}{(\xi - \lambda_{i_p})^2}\D \xi\right\vert=\mathcal{O}\left(N^\varepsilon\frac{\widehat{I}N}{\ell^2}\right).$$
Now for the rest of the contour, note that we can add the missing eigenvalues to the total sum in $g$ up to adding an error of order $N^\varepsilon\widehat{I}/\ell.$ Finally we just have to bound
\begin{multline}\label{eq:intcontour}
\frac{N}{\ell}\int_{\Gamma\setminus \Gamma_1}\left\vert \mathrm{Im}\sum_{k=1}^N\frac{p_{kk}}{N(\xi-\lambda_k)}\right\vert\vert\mathrm{d}\xi\vert
\\=\frac{N}{\ell}\int_{\Gamma\setminus \Gamma_1}\left\vert\mathrm{Im}\frac{1}{N}\sum_{\alpha\in I}G_{\alpha\alpha}(\xi)-\frac{\mathrm{Im}m_t(\xi)}{\mathrm{Im}m_t(z_0)}\frac{1}{N}\sum_{\alpha\in I}\mathrm{Im}g_\alpha(t,z_0)\right\vert\vert\mathrm{d}\xi\vert
\end{multline}
where we used the definition of $p_{kk}$ and of $\sigma_t$ and defined $z_0=\gamma_{t,i_0}+\mathrm{i}\eta_0,$ with $\eta_0\ll t$ is the center of our window of size $w\ll t$ with positive imaginary part. Now, using the entrywise local law from Theorem \ref{entrylocal} and expanding between $z_0$ and $\xi$ since $\vert \xi-z_0\vert\leqslant w,$ we have
\begin{align*}
\left\vert\mathrm{Im}\frac{1}{N}\sum_{\alpha\in I}\right.&\left.G_{\alpha\alpha}(\xi)-\frac{\mathrm{Im}m_t(\xi)}{\mathrm{Im}m_t(z_0)}\frac{1}{N}\sum_{\alpha\in I}\mathrm{Im}g_\alpha(t,z_0)\right\vert\\
&\leqslant \left\vert \mathrm{Im}\frac{1}{N}\sum_{\alpha\in I}\left(G_{\alpha\alpha}(\xi)-g_\alpha(t,\xi)\right)\right\vert+\left\vert\frac{1}{N}\sum_{\alpha\in I}\left(\mathrm{Im}g_\alpha(t,\xi)-\frac{\mathrm{Im}m_t(\xi)}{\mathrm{Im}m_t(z_0)}\mathrm{Im}g_\alpha(t,z_0)\right)\right\vert\\
&\leqslant N^\varepsilon\frac{\widehat{I}}{\sqrt{N\vert\mathrm{Im}\xi\vert}}+N^\varepsilon\widehat{I}\frac{w}{t}.
\end{align*}
Injecting this bound in the contour integral, we get the following bound.
$$\eqref{eq:intcontour}\leqslant N^\varepsilon\widehat{I}\frac{w}{t}+N^\varepsilon\frac{\widehat{I}}{\sqrt{N}}.$$ 
Finally, putting all the contributions together and coming back to \eqref{eq:intcontour2}, we obtain
\begin{equation}
\vert \partial_{z}f(\lambda_{i_p})\vert \leqslant N^\varepsilon \frac{N}{\ell}\left(\frac{\widehat{I}}{\ell}+\widehat{I}\frac{w}{t}\right).
\end{equation}

Consider now the contribution from \eqref{eq:defg2}, first see that for $G\in\mathcal{G}_{\bm{\xi}}^{(2)}$, there exists $q_1$ and $q_2$ in $\{1,\dots m\},$ such that
$$\prod_{e\in\mathcal{E}(G)}p(e)=p_{ki_{q_1}}p_{ki_{q_2}}\times Q_2((p(e))_{e\in\mathcal{E}(G)})$$
with $Q_2$ a monic monomials of degree $n-2$.
Then using  Lemma \ref{lem:holderdeformed}, we can bound
$$\mathds{E}\left[\sum_{G\in\mathcal{G}'}Q_2((p(e))_{e\in\mathcal{E}(G)})\middle\vert\bm{\lambda}\right]=\mathcal{O}\left(\left(S_{J_{out}}^{(0,\tau)}\right)^{\frac{(n-2)}{n}	}\right).$$

Besides, we can bound the term with the cross-edges in the following way
$$\sum_{\substack{k,\,\vert k-i_p\vert>\ell\\\tilde{\eta}_k=0}}\frac{p_{i_{q_1}k}p_{i_{q_2}}}{N(\lambda_{i_p}-\lambda_{k})^2}\leqslant \frac{N}{\ell^2}\sum_{k=1}^N(p_{i_{q_1}k}^2+p_{i_{q_2}k}^2)=\frac{N}{\ell^2}\sum_{\alpha\in I}\left(u_{q_1}(\alpha)^2+u_{q_2}(\alpha)^2\right)\leqslant\widehat{I}\frac{N}{\ell^2}.$$
Putting everything together, we get
$$\left\vert((U_\mathscr{B}(0,\tau)-U_\mathscr{S}(0,\tau))\Phi_t(\bm{\xi})\right\vert\leqslant N^\varepsilon \frac{N\tau}{\ell}\left(S_{J_{\mathrm{out}}}^{(0,\tau)}+\widehat{I}\frac{w}{t}\left(S_{J_{\mathrm{out}}}^{(0,\tau)}\right)^{\frac{n-1}{n}}+\frac{\widehat{I}}{\ell}\left(S_{J_{\mathrm{out}}}^{(0,\tau)}\right)^{\frac{n-2}{n}}\right)$$
which is exactly the result wanted.
\end{proof}
We will first prove the following proposition in order to deduce Theorem \ref{theo:resultQUE}.
\begin{proposition}\label{prop:mainque}
For any $\varepsilon$ and $N$ large enough the following holds. For any intervals $J_{\mathrm{in}}\subset\mathcal{A}_{r}^\kappa$ and $J_{\mathrm{out}}\subset\{i,\,d(i,J_{\mathrm{in}})\leqslant N^{-\varepsilon}Nw\}$ we have
\begin{multline}
S_{J_{\mathrm{in}}}^{(0,\tau)}\leqslant N^\varepsilon\left(\frac{\ell}{Nw}+\frac{N\tau}{\ell}+\frac{1}{N\tau}\right)S_{J_{\mathrm{out}}}^{(0,\tau)}
+
N^\varepsilon\left(\frac{\widehat{I}}{\sqrt{N\tau}}+\widehat{I}\frac{w}{t}\right)(S_{J_{\mathrm{out}}}^{(0,\tau)})^{\frac{n-1}{n}}
\\+
N^\varepsilon\left(\frac{\widehat{I}}{N\tau}+\widehat{I}\frac{N\tau}{\ell^2}\right)(S_{J_{\mathrm{out}}}^{(0,\tau})^{\frac{n-2}{n}}.
\end{multline}
\end{proposition}
\begin{proof}
We will use the short-range dynamics and its finite speed of propagation property in order to localize the maximum principle in $J_{\mathrm{in}}$. We will then use the local laws and Lemma \ref{lem:shortlong} in order to get a Gronwall type bound. Define the following averaging operator. For $f$ a function on configurations and $\bm{\xi}$ a configuration,
\begin{equation}
\mathrm{Av}(f)=\frac{3N^\varepsilon}{Nw}\sum_{\frac{1}{3}\frac{Nw}{N^\varepsilon}<a<\frac{2}{3}\frac{Nw}{N^\varepsilon}}\mathrm{Flat}_a(f).
\end{equation}
Note that, if $\bm{\xi}$ is not included in $J_{\mathrm{out}}$, by definition of the flattening operator, $(\mathrm{Av}(f))(\bm{\xi})=0$. The purpose of these operators is to change the initial condition in order to remove the particles far from the initial interval $J_{\mathrm{in}}$. Note also that we can write, for any $\bm{\xi}$,
$$\mathrm{Av}(f)(\bm{\xi})=a_{\bm{\xi}}f(\bm{\xi}),$$
with $a_{\bm{\xi}}\in[0,1]$. Note that we have the elementary bound $\left\vert a_{\bm{\xi}}-a_{\bm{\xi}}\right\vert\leqslant {CN^\varepsilon}/{Nw}$. Define now the following dynamics
\begin{equation}
\left\{\begin{array}{ll}
			\partial_s \Gamma_s=\mathscr{S}(s)\Gamma_s,\quad 0\leqslant s\leqslant \tau \\
			\Gamma_0(\bm{\xi})=(\mathrm{Av}\,\Phi_t)(\bm{\xi}).

		\end{array}\right.
\end{equation}

Now, if one takes a configuration $\bm{\xi}$ supported on $J_{\mathrm{in}}$, it suffices to show the bound in Proposition \ref{prop:mainque} for $\Gamma$. Indeed
\begin{align*}
\left\vert F_\tau(\bm{\xi})-\Gamma_\tau(\bm{\xi})\right\vert&\leqslant\left\vert \left(U_\mathscr{B}(0,\tau)\Phi_t-U_\mathscr{S}(0,\tau)\Phi_t\right)(\bm{\xi})\right\vert+\left\vert U_\mathscr{S}(0,\tau)(\Phi_t-\mathrm{Av}\Phi_t)(\bm{\xi})\right\vert\\
&\leqslant N^\varepsilon\frac{CN\tau}{\ell}\left(S_{J_{\mathrm{out}}}^{(0,\tau)}+\widehat{I}\frac{w}{t}\left(S_{J_{\mathrm{out}}}^{(0,\tau)}\right)^{\frac{n-1}{n}}+\frac{\widehat{I}}{\ell}\left(S_{J_{\mathrm{out}}}^{(0,\tau)}\right)^{\frac{n-2}{n}}\right)+\exp(-cN^\varepsilon)
\end{align*}
where we bounded the first term by using Lemma \ref{lem:shortlong} and the second term using the finite speed of propagation. Indeed, since $\bm{\xi}$ is supported on $J_\mathrm{in}$, $(\mathrm{Id}-\mathrm{Av})\Phi_t$ vanishes for any configuration supported on $J_{\mathrm{out}}$. Note that we can use Lemma \ref{lem:shortlong} since we will take $\ell/N\ll w$. In the rest of the proof, we will prove the bound from Proposition \ref{prop:mainque} for $\Gamma_\tau$. If we already have for some $C>0$,
$$\Gamma_\tau(\bm{{\eta_m}}):=\sup_{\bm{\xi},\,\,\mathcal{N}(\bm{\xi})=n}\Gamma_\tau(\bm{\xi})\leqslant N^{-C}$$
then we have nothing to prove by the argument above and the definition of $F_\tau$. However, if this supremum is greater than $N^{-C}$, then by the finite speed of propagation of $\mathscr{S}$, we know that $\bm{{\eta_m}}$ will be supported in, for instance, $\left\{i,\,d(i,J_{\mathrm{in}})\leqslant\frac{3Nw}{4N^\varepsilon}\right\}$.

Consider now, a parameter $\eta$ that we will choose later and denote also $m\leqslant n$ the number of sites with at least a particle and $j_1,\dots,j_m$ those sites. Then, we can write
\begin{align}
\partial_\tau &\Gamma_\tau(\bm{\eta_m})=\sum_{0<\vert j-k\vert\leqslant \ell}\frac{2\eta_{m,k}(1+2\eta_{m,j})\left(\Gamma_\tau\left(\bm{\eta_m^{k,j_p}}\right)-\Gamma_\tau\left(\bm{\eta_m}\right)\right)}{N(\lambda_k(t+\tau)-\lambda_j(t+\tau))^2}\\
&\leqslant \frac{C}{N\eta}\sum_{\substack{1\leqslant p\leqslant m\\k,\,0<\vert j_p-k\vert\leqslant\ell}}\frac{\eta(\Gamma_\tau\left(\bm{\eta_m^{k,j_p}}\right)-\Gamma_\tau(\bm{\eta_m}))}{(\lambda_k-\lambda_{j_p})^2+\eta^2}\\
&\label{eq:maxprincque}\leqslant-\frac{C}{N\eta}\sum_{\substack{1\leqslant p\leqslant m\\k,\,0<\vert j_p-k\vert\leqslant\ell}}\mathrm{Im}\left(\frac{\Gamma_\tau\left(\bm{\eta_m^{k,j_p}}\right)}{\lambda_k-z_{j_p}}\right)-\frac{1}{N\eta}\Gamma_\tau(\bm{\eta_m})\sum_{\substack{1\leqslant p\leqslant m\\k,\,0<\vert j_p-k\vert<\ell}}\mathrm{Im}\left(\frac{1}{z_{j_p}-\lambda_k}\right)
\end{align}
with $z_{j_p}=\lambda_{j_p}+i\eta$. For the second term, see that for $p\in[\![1,m]\!]$, if we choose $\eta$ to be smaller than $\ell/N$,
$$\#\left\{k,\,0<\vert j_p-k\vert\leqslant \ell\right\}\geqslant C\ell\geqslant CN\eta\geqslant C\#\left\{k,\,\vert\lambda_{j_p}-\lambda_{k}\vert\leqslant \eta\right\}\geqslant C'N\eta$$
where we used, in the two last inequalities, the rigidity of the eigenvalues. Now we can write
\begin{align*}
\sum_{\substack{1\leqslant p\leqslant m\\k,\,0<\vert j_p-k\vert\leqslant\ell}}\mathrm{Im}\left(\frac{1}{z_{j_p}-\lambda_k}\right)\geqslant \sum_{\substack{1\leqslant p\leqslant m\\k,\,0<\vert \lambda_{j_p}-\lambda_k\vert\leqslant \eta}}\frac{\eta}{(\lambda_{j_p}-\lambda_{k})^2+\eta^2}\geqslant CN.
\end{align*}

We now need to control the first term in \eqref{eq:maxprincque}. To do so, we will split it in the three following terms:
\begin{align}
\label{eq:term1}\phantom{+}&\mathrm{Im}\left(\sum_{k,\,0<\vert j_p-k\vert\leqslant\ell}\frac{\left(U_\mathscr{S}(0,\tau)\mathrm{Av}\Phi_t\right)(\bm{\eta_m^{j_p,k}})-\left(\mathrm{Av}U_\mathscr{S}(0,\tau)\Phi_t\right)(\bm{\eta_m^{j_p,k}})}{N(z_{j_p}-\lambda_k)}\right)\\
\label{eq:term2}+&\mathrm{Im}\left(\sum_{k,\,0<\vert j_p-k\vert\leqslant\ell}\frac{\left(\mathrm{Av}U_\mathscr{S}(0,\tau)\Phi_t\right)(\bm{\eta_m^{j_p,k}})-\left(\mathrm{Av}U_\mathscr{B}(0,\tau)\Phi_t\right)(\bm{\eta_m^{j_p,k}})}{N(z_{j_p}-\lambda_k)}\right)\\
\label{eq:term3}+&\mathrm{Im}\left(\sum_{k,\,0<\vert j_p-k\vert\leqslant\ell}\frac{\left(\mathrm{Av}U_\mathscr{B}(0,\tau)\Phi_t\right)(\bm{\eta_m^{j_p,k}})}{N(z_{j_p}-\lambda_k)}\right)
\end{align} 

To bound the first term, we will use the finite speed of propagation property of $\mathscr{S}$ from Lemma \ref{finitespeed}. Indeed, we can write
$$\left(U_\mathscr{S}(0,\tau)\mathrm{Av}\Phi_t\right)(\bm{\eta_m^{j_p,k}})-\left(\mathrm{Av}U_\mathscr{S}(0,\tau)\Phi_t\right)(\bm{\eta_m^{j_p,k}})=\frac{3N^\varepsilon}{Nw}\sum_{\frac{1}{3}\frac{Nw}{N^\varepsilon}<a<\frac{2}{3}\frac{Nw}{N^\varepsilon}}\mathscr{U}_a\left(\bm{\eta_m^{j_p,k}}\right)$$
with 
$$\mathscr{U}_a=U_\mathscr{S}(0,\tau)\mathrm{Flat_a}\Phi_t-\mathrm{Flat}_a U_\mathscr{S}(0,\tau)\Phi_t.$$
Fix $a$ and consider three cases, if $\bm{\eta_m^{j_p,k}}$ is supported on $\left\{i,\,d(i,J_\mathrm{In})>a+N^\varepsilon \ell \right\}$ then by definition of $\mathrm{Flat_a}$ we have
$$\mathrm{Flat}_aU_\mathscr{S}(0,\tau)\Phi_t=0,$$
and by Lemma \ref{finitespeed} we have
$$\left\vert U_\mathscr{S}(0,\tau)\mathrm{Flat}_a\Phi_t\right\vert\leqslant \exp\left(-\frac{cN^\varepsilon}{2}\right).$$
Now if $\bm{\eta_m^{j_p,k}}$ is supported on $\left\{i,\,d(i,J_{\mathrm{In}})\leqslant a-N^\varepsilon \ell\right\}$ then again by definition of $\mathrm{Flat_a}$,
\begin{align*}
\mathrm{Flat}_a \left(U_\mathscr{S}(0,\tau)\Phi_t\right)\left(\bm{\eta_m^{j_p,k}}\right)=\left(U_\mathscr{S}(0,\tau)\Phi_t\right)\left(\bm{\eta_m^{j_p,k}}\right).
\end{align*}
Thus
\begin{align*}
\left\vert\mathscr{U}_a\left(\bm{\eta_{m}^{j_p,k}}\right)\right\vert\leqslant \left\vert U_\mathscr{S}(0,\tau)\left(\Phi_t-\mathrm{Flat}_a\Phi_t\right)\left(\bm{\eta_m^{j_p,k}}\right)\right\vert\leqslant \exp\left(-\frac{cN^\varepsilon}{2}\right).
\end{align*}
Finally, if $\bm{\eta_m^{j_p,k}}$ is supported on $\left\{i,\,d(i,J_{\mathrm{In}})\leqslant a+\ell N^\varepsilon\right\}$, first note that there can only be $2n\ell N^\varepsilon$ such $a$, then one can see that we can use the finite speed of propagation if we remove particle away from $a$ at distance $2\ell N^\varepsilon$ for instance, then
\begin{align*}
\left\vert\mathscr{U}_a\left(\bm{\eta_{m}^{j_p,k}}\right)\right\vert&\leqslant\left\vert \mathrm{Flat}_aU_\mathscr{S}(0,\tau)\Phi_t\right\vert+\left\vert\mathrm{Flat}_{a}U_\mathscr{S}(0,\tau)\mathrm{Flat}_{a+2\ell N^\varepsilon}\Phi_t\right\vert\\
&+\left\vert\mathrm{Flat_a}U_\mathscr{S}(0,\tau)\left(\Phi_t-\mathrm{Flat}_{a+2\ell N^\varepsilon}\right)\right\vert\\
&\leqslant \norme{\mathrm{Flat}_a\Phi_t}_\infty+\norme{\mathrm{Flat}_{a+2\ell N^\varepsilon}}_\infty+\exp\left(-\frac{cN^\varepsilon}{2}\right)\\
&\leqslant 2S_{J_{\mathrm{out}}}^{(0,\tau)}+\exp\left(-\frac{cN^\varepsilon}{2}\right),
\end{align*}
where we used that $U_\mathscr{S}$ is a contraction in $\Vert\Vert_\infty$.
Finally we can bound \eqref{eq:term1},
\begin{equation}
\eqref{eq:term1}\leqslant N^\varepsilon\frac{\ell}{Nw}S_{J_{\mathrm{out}}}^{(0,\tau)}+\exp\left(-\frac{cN^\varepsilon}{2}\right)
\end{equation}
where we used the fact that
\begin{equation}\label{eq:boundstiel}\left\vert\frac{1}{N}\mathrm{Im}\left(\sum_{k, 0<\vert j_p-k\vert<\ell}^N\frac{1}{z_{j_p}-\lambda_k}\right)\right\vert\leqslant \left\vert \varsigma\left(\tau,z_{j_p}\right)\right\vert\leqslant N^\varepsilon
\end{equation}
where the Stieltjes transform $\varsigma$ is defined in \eqref{eq:defstielt}.
For $\eqref{eq:term2},$ we will use Lemma \ref{lem:shortlong}. Indeed, first note that in the short-range regime, the set of $k$ such that $\vert j_p-k\vert\leqslant \ell$ is included in $\mathcal{A}_{r}^{2\kappa}.$ Then we can bound
\begin{equation}
\eqref{eq:term2}\leqslant N^\varepsilon\frac{N\tau}{\ell}\left(S_{J_{\mathrm{out}}}^{(0,\tau)}+\widehat{I}\frac{w}{t}\left(S_{J_{\mathrm{out}}}^{(0,\tau)}\right)^{\frac{n-1}{n}}+\frac{\widehat{I}}{\ell}\left(S_{J_{\mathrm{out}}}^{(0,\tau)}\right)^{\frac{n-2}{n}}\right)
\end{equation}
where we used the fact that $\mathrm{Av}$ is a contraction and $\eqref{eq:boundstiel}$.

Finally, in order to bound the third term $\eqref{eq:term3}$, we will use the local law for $\Phi_t$. First write
\begin{align}
\mathrm{Im}&\left(\sum_{k,\,0<\vert j_p-k\vert\leqslant \ell}\frac{\left(\mathrm{Av}U_\mathscr{B}(0,\tau)\Phi_t\right)\left(\bm{\eta_m^{j_p,k}}\right)}{N(z_{j_p}-\lambda_k)}\right)=\mathrm{Im}\left(\sum_{k,\,0<\vert j_p-k\vert\leqslant \ell}\frac{a_{\bm{\eta_m^{j_p,k}}}F_\tau\left(\bm{\eta_m^{j_p,k}}\right)}{N(z_{j_p}-\lambda_k)}\right)\nonumber\\
&=\mathrm{Im}\left(\sum_{k,\,0<\vert j_p-k\vert\leqslant \ell}\frac{a_{\bm{\eta_m}}F_\tau\left(\bm{\eta_m^{j_p,k}}\right)}{N(z_{j_p}-\lambda_k)}\right)+\mathrm{Im}\left(\sum_{k,\,0<\vert j_p-k\vert\leqslant \ell}\frac{\left(a_{\bm{\eta_m^{j_p,k}}}-a_{\bm{\eta_m}}\right)F_\tau\left(\bm{\eta_m^{j_p,k}}\right)}{N(z_{j_p}-\lambda_k)}\right)\nonumber\\
\label{eq:term31}&=\mathrm{Im}\left(\sum_{k,\,0<\vert j_p-k\vert\leqslant \ell}\frac{a_{\bm{\eta_m}}F_\tau\left(\bm{\eta_m^{j_p,k}}\right)}{N(z_{j_p}-\lambda_k)}\right)+\mathcal{O}\left(N^\varepsilon\mathds{1}_{\vert j_p-k\vert\leqslant \ell}\,{\left\vert a_{\bm{\eta_m}}-a_{\bm{\eta_m^{j_p,k}}}\right\vert}S_{J_\mathrm{out}}^{(0,\tau)}\right).
\end{align}
But we have the bound, for $\vert j_p-k\vert\leqslant \ell$,
$$\left\vert a_{\bm{\eta_m}}-a_{\bm{\eta_m^{j_p,k}}}\right\vert\leqslant \frac{N^\varepsilon d\left(\bm{\eta_m^{j_p,k}},\bm{\eta_m}\right)}{Nw}\leqslant \frac{N^\varepsilon\ell}{Nw}.$$

In order to bound the first term, see first that we can remove the contributions of $k\in\{j_1,\dots,j_p\}$ writing
$$\mathrm{Im}\left(\sum_{k,\,0<\vert j_p-k\vert\leqslant \ell}\frac{a_{\bm{\eta_m}}F_\tau\left(\bm{\eta_m^{j_p,k}}\right)}{N(z_{j_p}-\lambda_k)}\right)=\mathrm{Im}\left(\sum_{\substack{k,\,0<\vert j_p-k\vert\leqslant \ell\\k\notin\{j_1,\dots,j_p\}}}\frac{a_{\bm{\eta_m}}F_\tau\left(\bm{\eta_m^{j_p,k}}\right)}{N(z_{j_p}-\lambda_k)}\right)+\mathcal{O}\left(\frac{N^\varepsilon}{N\eta}S_{J_{\mathrm{out}}}^{(0,\tau)}\right).$$
Recall now the definition of $\Phi_t$ from \eqref{lem:eq:perfobsdeformed} and write,
\begin{equation}\label{eq:debutrec}
\sum_{\substack{k,\,0<\vert j_p-k\vert\leqslant \ell\\k\notin\{j_1,\dots,j_p\}}}\frac{F_\tau\left(\bm{\eta_m^{j_p,k}}\right)}{N(z_{j_p}-\lambda_k)}=\frac{1}{\mathcal{M}\left(\bm{\eta_m^{j_p,k}}\right)}\mathds{E}\left[\sum_{\substack{k,\,0<\vert j_p-k\vert\leqslant \ell\\k\notin\{j_1,\dots,j_p\}}}\sum_{G\in\mathcal{G}_{\bm{\eta_m^{j_p,k}}}}\frac{\prod_{e\in\mathcal{E}(G)}p(e)}{N(z_{j_p}-\lambda_k)}\middle\vert\bm{\lambda}\right].
\end{equation}


First consider the contribution of \eqref{eq:defg1} in the sum in \eqref{eq:debutrec}, denote $e_k=\{(k,1),(k,2)\}$ and write
\begin{align}
\sum_{\substack{k,\,0<\vert j_p-k\vert\leqslant \ell\\k\notin\{j_1,\dots,j_p\}}}\sum_{G\in\mathcal{G}^{(1)}_{\bm{\eta_m^{j_p,k}}}}\frac{\prod_{e\in\mathcal{E}(G)}p(e)}{N(z_{j_p}-\lambda_k)}&=\sum_{\substack{k,\,0<\vert j_p-k\vert\leqslant \ell\\k\notin\{j_1,\dots,j_p\}}}\sum_{G\in\mathcal{G}_{\bm{\eta_m^{j_p,k}}}^{(1)}}\left(\prod_{e\in\mathcal{E}(G)\setminus\{e_k\}}p(e)\right)\frac{p_{kk}}{N(z_{j_p}-\lambda_k)}\nonumber\\
\label{eq:n-1}&=\left(\sum_{G\in\mathcal{G}_{\bm{\eta_m\setminus j_p}}}\prod_{e\in\mathcal{E}(G)}p(e)\right)\sum_{\substack{k,\,0<\vert j_p-k\vert\leqslant \ell\\k\notin\{j_1,\dots,j_p\}}}\frac{p_{kk}}{N(z_{j_p}-\lambda_k)}.
\end{align}
To control the last term in \eqref{eq:n-1}, we can use the local law. First write,
\begin{equation}
\sum_{\substack{k,\,0<\vert j_p-k\vert\leqslant \ell\\k\notin\{j_1,\dots,j_p\}}}\frac{p_{kk}}{N(z_{j_p}-\lambda_k)}=\sum_{k=1}^N\frac{p_{kk}}{N(z_{j_p}-\lambda_k)}+\sum_{k,\,\vert j_p-k\vert>\ell}\frac{p_{kk}}{N(z_{j_p}-\lambda_k)}+\mathcal{O}\left(N^\varepsilon\frac{\widehat{I}}{N\eta}\right),
\end{equation}
where we used the bound $\vert p_{kk}\vert\leqslant N^\varepsilon \widehat{I}$.
Now, recall the definition of $p_{kk}$ from \eqref{eq:pii} so that we have
\begin{align*}
&\left\vert\mathrm{Im}\sum_{k=1}^N\frac{p_{kk}}{N(z_{j_p}-\lambda_k)}\right\vert=\left\vert\mathrm{Im}\frac{1}{N}\sum_{\alpha\in I}G_{\alpha\alpha}(z_{j_p})-\frac{\mathrm{Im}m_t(z_{j_p})}{\mathrm{Im}m_t(z_0)}\frac{1}{N}\sum_{\alpha\in I}\mathrm{Im}g_\alpha(t,z_0)\right\vert\\
&\leqslant \left\vert \mathrm{Im}\frac{1}{N}\sum_{\alpha\in I}\left(G_{\alpha\alpha}(z_{j_p})-g_\alpha(t,z_{j_p})\right)\right\vert+\left\vert\frac{1}{N}\sum_{\alpha\in I}\left(\mathrm{Im}g_\alpha(t,z_{j_p})-\frac{\mathrm{Im}m_t(z_{j_p})}{\mathrm{Im}m_t(z_0)}\mathrm{Im}g_\alpha(t,z_0)\right)\right\vert\\
&\leqslant N^\varepsilon\left(\frac{\widehat{I}}{\sqrt{N\eta}}+\widehat{I}\frac{w}{t}\right)
\end{align*} 
where $z_0:=\gamma_{t,i_0}+\mathrm{i}\eta$. Note that we used the fact that $\bm{\eta_m}$ is supported in $J_{\mathrm{out}}$, so that $\vert z_0-z_{j_p}\vert=\vert \gamma_{t,i_0}-\lambda_{j_p}\vert\leqslant N^\varepsilon w.$ We can then use Lemma \ref{lem:holderdeformed} and bound
$$\sum_{G\in\mathcal{G}_{\bm{\eta_m\setminus j_p}}}\prod_{e\in\mathcal{E}(\mathcal{G}}p(e)=\mathcal{O}\left(\sup_{\bm{\xi}\subset J_{\mathrm{out}},\,\mathcal{N}(\bm{\xi})=n-1}\vert \Phi_t(\bm{\xi})\vert\right)=\mathcal{O}\left((S_{J_\mathrm{out}}^{(0,\tau)})^{\frac{n-1}{n}}\right)$$

Now, consider the contribution of \eqref{eq:defg2} in the sum from \eqref{eq:debutrec}. Note that for any graph $G$ in $\mathcal{G}^{(2)}_{\bm{\xi}}$, there exists $q$ and $q'$ in $\{1,\dots,\,m\}$, $a\in\{1\,\dots,\eta_{i_q}\}$ and $b\in\{1,\dots,\eta_{i_{q'}}\}$, such that $e_q:=\{(k,1),(q,a)\}$ and $e_{q'}:=\{(k,2),(q',b)\}$ are edges in $G$. We can then write
\begin{align}
\sum_{k,\,0<\vert j_p-k\vert\leqslant\ell}\sum_{G\in\mathcal{G}^{(2)}_{\bm{\eta_m^{j_p,k}}}}\frac{\prod_{e\in\mathcal{E}(G)}p(e)}{N(z_{j_p}-\lambda_k)}&=\sum_{k,\,0<\vert j_p-k\vert\leqslant\ell}\sum_{G\in\mathcal{G}^{(2)}_{\bm{\eta_m^{j_p,k}}}}\left(\prod_{e\in\mathcal{E}(G)\setminus{\{e_q,e_{q'}}\}}p(e)\right)\frac{p_{i_q,k}p_{i_{q'},k}}{N(z_{j_p}-\lambda_k)},\nonumber\\
\label{eq:g2bound}&=\sum_{q,q'=1}^m\left(\sum_{G\in\mathcal{H}_{\bm{\eta_m\setminus{j_p}}}^{q,q'}}\prod_{e\in \mathcal{E}(G)}p(e)\right)\sum_{k,\,0<\vert j_p-k\vert\leqslant \ell}\frac{p_{i_q,k}p_{i_{q'},k}}{N(z_{j_p}-\lambda_k)},
\end{align}
where we defined the set of graphs $\mathcal{H}_{\bm{\xi}}^{q,q'}$ to be the set of perfect matching of the complete graph on the set of vertices  
$\mathcal{V}_{\bm{\xi}}$ where we removed a single particle at the site $i_q$ and $i_{q'}$. Note that for any graph $G\in\mathcal{H}_{\bm{\eta_m\setminus j_p}}^{q,q'}$, $\prod_{e\in\mathcal{E}(G)}p(e)$ is a monomial of degree $n-2$.\\
Now, we can bound the imaginary part of the second sum in \eqref{eq:g2bound},
\begin{align}
\label{eq:g2bound2}\left\vert\mathrm{Im}\left(\sum_{\substack{k,\,0<\vert j_p-k\vert\leqslant \ell\\k\notin\{j_1,\dots,j_p\}}}\frac{p_{i_q,k}p_{i_{q'},k}}{N(z_{j_p}-\lambda_k)}\right)\right\vert\leqslant \frac{C}{N\eta}\sum_{k=1}^N\left(p_{i_q,k}^2+p_{i_{q'},k}^2\right)=\mathcal{O}\left(\frac{N^\varepsilon\widehat{I}}{N\eta}\right).
\end{align}
For the last inequality, we used the following identity on eigenvectors
$$\sum_{k=1}^Np_{i,k}^2=\sum_{\alpha\in I}u_i(\alpha)^2$$
and that for any $\varepsilon>0$,using the entrywise local law from Theorem \ref{entrylocal} on a diagonal entry of the resolvent,
$$u_k(\alpha)^2\leqslant N^{-1+\varepsilon}\mathrm{Im}\left(G(\tau,\lambda_k+\mathrm{i}N^{-1+\varepsilon})_{\alpha,\alpha}\right)\leqslant \frac{N^\varepsilon}{Nt}$$
Again, we can bound the other term from \eqref{eq:g2bound} using Lemma \ref{lem:holderdeformed},
$$\sum_{G\in\mathcal{H}_{\bm{\eta_m\setminus{j_p}}}^{q,q'}}\prod_{e\in \mathcal{E}(G)}p(e)=\mathcal{O}\left(N^\varepsilon\sup_{\bm{\xi}\subset J_{\mathrm{out}},\,\mathcal{N}(\bm{\xi})=n-2}\vert \Phi_t(\bm{\xi})\vert\right)=\mathcal{O}\left(\left(S_{J_{\mathrm{out}}}^{(0,\tau)}\right)^{\frac{n-2}{n}}\right)$$

Finally, putting all these estimates together, we get the Gronwall-type inequality,

\begin{multline}
\partial_\tau \Gamma_\tau(\bm{\eta_m})\leqslant-\frac{1}{\eta}\Gamma_\tau(\bm{\eta_m})+\mathcal{O}\left( \frac{N^\varepsilon}{\eta}\left(\left(\frac{\ell}{Nw}+\frac{N\tau}{\ell}+\frac{1}{N\eta}\right)S_{J_{\mathrm{out}}}^{(0,\tau)}\right.\right.\\\left.\left.+\left(\frac{\widehat{I}}{\sqrt{N\eta}}+\widehat{I}\frac{w}{t}\right)(S_{J_{\mathrm{out}}}^{(0,\tau)})^{\frac{n-1}{n}}+\left(\frac{\widehat{I}}{N\eta}+\frac{N\tau}{\ell^2}\right)(S_{J_{\mathrm{out}}}^{(0,\tau})^{\frac{n-2}{n}}\right)\right)
\end{multline}

In order to get a proper bound using Gronwall's lemma, we need to take $\eta\ll\tau$ but to get the best estimates possible, we also have to take $\eta$ as large as possible. Hence, considering $\eta=N^{-\varepsilon}\tau$ we have the bound
\begin{multline}
S_{J_{\mathrm{in}}}^{(0,\tau)}\leqslant N^\varepsilon\left(\frac{\ell}{Nw}+\frac{N\tau}{\ell}+\frac{1}{N\tau}\right)S_{J_{\mathrm{out}}}^{(0,\tau)}+N^{2\varepsilon}\left(\frac{\widehat{I}}{\sqrt{N\tau}}+\widehat{I}\frac{w}{t}\right)(S_{J_{\mathrm{out}}}^{(0,\tau)})^{\frac{n-1}{n}}\\
+N^\varepsilon\left(\frac{\widehat{I}}{N\tau}+\widehat{I}\frac{N\tau}{\ell^2}\right)(S_{J_{\mathrm{out}}}^{(0,\tau})^{\frac{n-2}{n}}
\end{multline}
which gives the Proposition \ref{prop:mainque}.
\end{proof}

Now that we have the bound from Proposition \ref{prop:mainque}, we are able to get a bound on the $p_{ij}$ using Lemma \ref{lem:holderdeformed}. To do so, we can use a sequence of set of indices with decreasing size and apply recursively Proposition \ref{prop:mainque}. We will also need to choose the right parameters $\ell$, $w$ and $\tau$.
\begin{proof}[Proof of Theorem \ref{theo:resultint}]
Consider first any $\varepsilon$ small enough, such that if we write $t=N^{-1+\delta}$ (recall that $t\in\mathcal{T}_\omega$ so that $t\gg N^{-1}$) we have $\varepsilon<4\delta/3.$ and a large $D>0$. Then we can take the following parameters :
\begin{equation}\label{eq:deftauw}
w=N^\varepsilon \tau\quad \text{and}\quad \ell=N\sqrt{\tau w},
\end{equation}
note that we have the right bounds between these parameters: $N^{-1}\ll \tau\ll \ell/N\ll w\ll t$, and define the following sequence of sets of indices, defined implicitely,
$$\left\{\begin{array}{ll}
			J_0=\mathcal{A}_{w}^\kappa(i_0),\\
			J_i=\{i:\,d(i,J_{i+1})\leqslant N^{-\varepsilon}Nw\}.
		\end{array}		
\right.$$
From Proposition \ref{prop:mainque} we have the following bound holding with overwhelming probability,
\[
S_{J_{i+1}}^{(0,\tau)}\leqslant N^{-\varepsilon/2}S_{J_i}^{(0,\tau)}
+
\left(\frac{\widehat{I}}{\sqrt{N\tau}}
+\widehat{I}\frac{N^\varepsilon\tau}{t}\right)(S_{J_i}^{(0,\tau)})^{\frac{n-1}{n}}\\+\widehat{I}\frac{1}{N\tau}(S_{J_i}^{(0,\tau)})^{\frac{n-2}{n}}.
\]

Now see that as long as we have
$$(S_{J_i}^{(0,\tau)})^{1/n}\geqslant CN^{3\varepsilon/2}\Xi(\tau)$$
with $\Xi$ given in \eqref{eq:defxi}, we obtain the recursive bound
$$S_{J_{i+1}}^{(0,\tau)}\leqslant N^{-\varepsilon/2}S_{J_i}^{(0,\tau)}.$$
But if we take a very large $i$ so that the previous bound cannot hold, for instance $i=\lceil 3\varepsilon^{-1}\rceil$, then it means that for such a $i$ we have the bound
$$(S_{J_i}^{(0,\tau)})^{1/n}\leqslant CN^{3\varepsilon/2}\Xi.$$
Now using the definition of $p_{ii}$, we have for $i\in\mathcal{A}_{w}^\kappa(i_0),$
\begin{equation}\label{eq:endproof1}
\left\vert\sum_{\alpha\in I}\left(u_i(\alpha)^2-\frac{1}{N}\sigma_t^2(\alpha,i)\right)\right\vert\leqslant \vert p_{ii}\vert +\widehat{I}\frac{w}{t}\leqslant\vert p_{ii}\vert + \Xi(\tau).
\end{equation}
Finally, using Markov's inequality, taking for instance $n=\lfloor 3D/\varepsilon\rfloor$, and using Lemma \ref{lem:holderdeformed} to bound the $p_{ij}$ by $S^{(0,\tau)}$ we have
\begin{equation}
\label{eq:endproof2}\mathds{P}\left(\vert p_{ii}\vert+\vert p_{ij}\vert\geqslant N^\varepsilon\Xi(\tau)\right)\leqslant N^{-D}.
\end{equation}
The result then follows from combining \eqref{eq:endproof1} and \eqref{eq:endproof2}.
\end{proof}

\section{Approximation by a Gaussian divisible ensemble}

\subsection{Continuity of the Dyson Brownian motion}\label{sec:cont}

\indent In Subsection \ref{subsec:moment} we showed that the moments of the eigenvectors of the matrix $H_\tau$ are asymptotically those of a Gaussian random variable with variance $\sigma_{t}^2$. Now, since $\tau$ is a small time, recall that $N^{-1}\ll \tau\ll t$, we can use the continuity of the Dyson Brownian motion to show that $H_\tau$ and $H_0=W_t$ have the same local statistics. In order to state a proper continuity lemma we need to have a dynamics with constant second moments and vanishing expectation.\\
\indent  First see that the variance of the centered model is
$$\mathds{E}\left[\left(W_{t,ij}-D_{ij}\right)^2\right]=\frac{t}{N}.$$
Consider, for $0\leqslant s\leqslant \tau$, the following variance-preserving dynamics on symmetric matrices.
\begin{align}
\label{eq:deftildeh}\mathrm{d}\left(\tilde{H}(s)-D\right)&=\frac{\mathrm{d}B}{\sqrt{N}}-\frac{1}{2t}\left(\tilde{H}(s)-D\right)\mathrm{d}s,\\
\tilde{H}(0)&=W_t=D+\sqrt{t}W.
\end{align}

The following lemma gives us a continuity argument between $\tilde{H}(\tau)$ and $W_t$. It is similar to Lemma A.1 in \cite{bourgade2017eigenvector} or Lemma 4.3 in \cite{huang2015bulk}. We will later use this lemma on the resolvent entries.
\begin{lemma}\label{lem:cont}
Denote $\partial_{\alpha\beta}=\partial_ {\tilde{H}_{\alpha\beta}}$. Take $F$ a smooth function of the matrix entries satisfying
\begin{equation}\label{eq:boundcont}
\mathds{E}\left[\sup_{\bm{\theta},\,0\leqslant s\leqslant \tau}\frac{1}{N}\sum_{\alpha\leqslant \beta}\left(\frac{N\vert \left(\tilde{H}(s)-D\right)_{\alpha\beta}\vert^3}{t}+\vert \left(\tilde{H}(s)-D\right)_{\alpha\beta}\vert\right)\left\vert\partial_{\alpha\beta}^3F(\bm{\theta}\tilde{H}_s)\right\vert\right]\leqslant M
\end{equation}
where $(\bm{\theta}H)_{\alpha\beta}=\theta_{\alpha\beta}H_{\alpha\beta}$ with $\theta_{kl}=1$ for $\{k,l\}\neq\{\alpha,\beta\}$ and $\theta_{\alpha\beta}\in[0,1].$
Then
$$\mathds{E}[F(\tilde{H}(\tau))]-\mathds{E}[F(\tilde{H}(0))]=\mathcal{O}\left(\tau\right)M.$$
\end{lemma}

\begin{proof}
By It\^o's formula we have
$$\partial_s\mathds{E}[F(\tilde{H}(s))]=-\frac{1}{2N}\sum_{\alpha\leqslant \beta}\frac{N}{t}\mathds{E}[\left(\tilde{H}(s)-D\right)_{\alpha\beta}\partial_{\alpha\beta}F(\tilde{H}_s)]-\mathds{E}[\partial^2_{\alpha\beta}F(\tilde{H}_s)].$$
Using Taylor expansions, we can write, forgetting the dependence in time for  clarity,
\begin{align*}
\mathds{E}[\left(\tilde{H}-D\right)_{\alpha\beta}\partial_{\alpha\beta}F(\tilde{H})]&=\mathds{E}[\left(\tilde{H}-D\right)_{\alpha\beta}\partial_{\alpha\beta}F_{\tilde{H}_{\alpha\beta}=D_{\alpha\beta}}]+\mathds{E}[\left(\tilde{H}-D\right)_{\alpha\beta}^2\partial_{\alpha\beta}^2F_{\tilde{H}_{\alpha\beta}=D_{\alpha\beta}}]\\
&+\mathcal{O}\left(\mathds{E}\left[\sup_{\bm{\theta}}\left\vert \left(\tilde{H}-D\right)_{\alpha\beta}^3\partial_{\alpha\beta}^3F(\bm{\theta}\tilde{H})\right\vert\right]\right).\\
&=\frac{t}{N}\partial^2_{\alpha\beta}F_{\tilde{H}_{\alpha\beta}=D_{\alpha\beta}}+\mathcal{O}\left(\mathds{E}\left[\sup_{\bm{\theta}}\left\vert \left(\tilde{H}-D\right)_{\alpha\beta}^3\partial_{\alpha\beta}^3F(\bm{\theta}\tilde{H})\right\vert\right]\right).
\end{align*}
and
\begin{align*}
\mathds{E}[\partial^2_{\alpha\beta}F(\tilde{H}_s)]=\mathds{E}[\partial^2_{\alpha\beta}F_{\tilde{H}_{\alpha\beta}=D_{\alpha\beta}}]+\mathcal{O}\left(\mathds{E}\left[\sup_{\bm{\theta}}\left\vert \left(\tilde{H}-D\right)_{\alpha\beta}\partial^3_{\alpha\beta}F(\bm{\theta}\tilde{H}_s)\right\vert\right]\right).
\end{align*}
Putting everything together the claim follows.
\end{proof}
This continuity property of the Dyson Brownian motion gives us a control over the eigenvalues and eigenvectors of $\tilde{H}(0)$ and $\tilde{H_s}.$
\begin{corollary}\label{coro:conti}
Let $\kappa\in(0,1)$ and $m\in\mathbb{N}$. Let $\Theta:\mathbb{R}^{2m}\rightarrow\mathbb{R}$ be a smooth function satisfying
$$\sup_{k\in[\![0,5]\!],x\in\mathbb{R}^{2m}}\vert\Theta^{(k)}(x)(1+\vert x\vert)^{-C}<\infty,$$
for some $C>0$. Denote $\tilde{u}_1(s),\dots,\tilde{u}_N(s)$ the eigenvectors of $\tilde{H}(s)$ associated with the eigenvalues $\tilde{\lambda}_1(s),\dots,\tilde{\lambda}_N(s)$. Define, for a small $\mathfrak{a}$ the time domain
$$\mathcal{T}_\mathfrak{a}'=\left[\frac{N^\mathfrak{a}}{N},N^{-\mathfrak{a}}\sqrt{\frac{t}{N}}\right],$$
then for any $\tau\in\mathcal{T}_\mathfrak{a}'$, there exists $\mathfrak{p}>0$ depending on $\Theta,$ $\mathfrak{a}$, $\kappa$ and $r$ such that
\begin{equation}\label{eq:contineigvect}
\sup_{I\subset \mathcal{I}_{r}^\kappa, \,\vert I\vert=m,\,\Vert\mathbf{q}\Vert=1}\left\vert\left(\mathds{E}^{\tilde{H}_s}-\mathds{E}^{\tilde{H}_0}
\right)\Theta\left(\left(N(\tilde{\lambda}_k-\gamma_{k,t}),\frac{N}{\sigma_t^2(\mathbf{q},k,\eta)}\langle\mathbf{q},\tilde{u}_k\rangle^2\right)_{k\in I}\right)\right\vert\leqslant N^{-\mathfrak{p}}.
\end{equation}

\end{corollary}

\begin{proof}
We prove this corollary using the continuity estimate from Lemma \ref{lem:cont}. To do so, we use the techniques introduced in \cite{knowles2013eigenvector} in order to change estimates of the resolvent below microscopic scales, in other words control $G(E+\I \eta)$ for $\eta\ll N^{-1}$, into estimates on eigenvectors. Indeed, it has been shown that such a control of the resolvent combined with an estimate on the number of eigenvalues in a very small interval allows us via integrating the resolvent over such an interval to gain estimates on eigenvectors. We can then split the proof into two results we need to show:

\begin{itemize}
\item[$(i)$] A level repulsion estimate on the eigenvalues for both matrix ensembles of the following form: for $E\in\mathcal{I}_r^\kappa$ and a small $\xi>0$ there exists $\mathfrak{d}>0$ such that
$$\mathds{P}\left(\left\vert\left\{\lambda_i\in[E-N^{-1-\xi},E+N^{-1-\xi}]\right\}\right\vert\geqslant 2\right)\leqslant N^{-\xi-\mathfrak{d}}.$$ 
\item[$(ii)$] Comparison of the resolvent below microscopic scales: for any smooth function $F$ of polynomial growth, there exists a $\mathfrak{c}>0$ and a $\xi>0$ such that for all $N^{-1-\xi}<\eta<t,$ 
\begin{equation}\label{eq:controlresolv}
\sup_{\substack{\Vert\mathbf{q}\Vert_{2}=1,\\\,E_1,\dots,E_m\in\mathcal{I}^\kappa_{r}}}\left\vert\left(\mathds{E}^{\tilde{H}_\tau} -\mathds{E}^{W_t}\right)F\left(\left(\frac{1}{\mathrm{Im}\left(\sum_{\alpha=1}^N q_\alpha^2g_\alpha(t,z_k)\right)}\langle\mathbf{q},G(z_k)\mathbf{q}\rangle\right)_{k=1}^m\right)\right\vert
\leqslant {N^{-\mathfrak{c}}}
\end{equation}
where $z_k=E_k+\mathrm{i}\eta$.
\end{itemize}

We first prove $(i)$ for the eigenvalues of $W_t.$ This property can be deduced from gap universality, note that gap universality for Gaussian perturbation of size $t\in\mathcal{T}_\omega$ has been shown in \cite{landon2017convergence} (a stronger level repulsion estimate can also be found in \cite{landon2017convergence}*{Section 5}). In \cite{landon2017convergence}*{Subsection 2.4}, Landon$-$Yau explains that they can deduce universality for deformed Wigner ensembles. However, they state the result for an initial condition such that $r^2\gg t$. It has been confirmed by the authors that it is a simple typographical error and should be read as $r\gg t.$ We will nonetheless give an idea of the proof of $(i)$ for the sake of completeness.\\
\indent  As said earlier, we will first apply Lemma \ref{lem:cont} to $$F(\tilde{H}_s)=\frac{1}{N}\mathrm{Tr}(\tilde{H}_s-z)^{-1}\quad\text{for $z$ in}\quad \left\{z=E+\mathrm{i}\eta, \,E\in\mathcal{I}_{r}^\kappa,\,N^{-1-\xi}<\eta<t\right\}$$
for $\xi>0$ arbitrarily small. See that in Lemma \ref{lem:cont}, we need to bound a functional of the form $F(\theta \tilde{H}_s)$ for $\theta$ a perturbation of two entries of the matrix. Since we will only need bounds such as Theorem \ref{theo:isolocal} or Corollary \ref{coro:aniso} with a possible $N^\varepsilon$ for any small positive $\varepsilon$ window, such bounds still hold for the perturbated matrix (see \cite{erdos2017dynamical}*{Section 15}). So we will explain the bound for the third derivative of $F$ applied directly to $\tilde{H}_s$.
Note that by definition of $\tilde{H}$, we have $\vert (\tilde{H}-D)_{\alpha\beta}\vert\prec\sqrt{\frac{t}{N}}$ so that we can bound the left hand side of \ref{eq:boundcont} by
\begin{equation}\label{eq:boundcont2}
\frac{1}{N}\sqrt{\frac{t}{N}}\sum_{\alpha\leqslant \beta}\vert\partial^3_{\alpha\beta}F(\tilde{H}_s)\vert.
\end{equation}
Taking the third derivative of $F$ with respect to an entry, we obtain, writing $G=(\tilde{H}_s-z)^{-1}$ for simplicity
$$\partial_{\alpha\beta}^3F(\tilde{H}_s)=-\frac{1}{N}\sum_{\gamma=1}^N\sum_{\bm{\alpha},\bm{\beta}}G_{\gamma\alpha_1}G_{\beta_1\alpha_2}G_{\beta_2\alpha_3}G_{\beta_3\gamma}$$
where $\{\alpha_\ell,\beta_\ell\}=\{\alpha,\beta\}$ for $\ell=1,2,3$.
To bound the sum in the previous equation, we will need the following high probability bound for the off-diagonal entries of the resolvent which follows directly from Theorem \ref{entrylocal}
\begin{align}
\label{eq:offdiaglocal}\left\vert G_{\alpha\beta}(z)\right\vert&\prec \frac{1}{\sqrt{N\eta}}\sqrt{\vert g_\alpha(t,z)g_\beta(t,z)\vert}.
\end{align}
Note that these bounds holds for $\eta\gg N^{-1}$, we will first consider such $\eta$.

\paragraph{}Finally we can bound \eqref{eq:boundcont2},
\begin{align}
\label{eq:boundtrace1}\frac{1}{N}\sqrt{\frac{t}{N}}\sum_{\alpha\leqslant \beta}\left\vert \partial^3_{\alpha\beta}F(\tilde{H}_s)\right\vert&\prec\frac{N^{2\xi}}{N^2}\sqrt{\frac{t}{N}}\sum_{\alpha\leqslant \beta}\sum_{\gamma=1}^N\sum_{\bm{\alpha},\bm{\beta}}\left\vert g_\gamma\right\vert\left(\left\vert g_{\alpha_1}g_{\beta_1}g_{\alpha_2}g_{\beta_2}g_{\alpha_3}g_{\beta_3}
\right\vert\right)^{1/2} .
\end{align}
Now, from Lemma 7.5 of \cite{landon2017convergence}, we have
\begin{equation}\label{eq:sommegk}
\frac{1}{N}\sum_{\gamma=1}^N\vert g_\gamma(t,z)\vert \leqslant C\log N
\end{equation} 
where the constant $C$ only depend on $D$ our diagonal matrix.
Besides, in the last product of \eqref{eq:boundtrace1}, there are, by definition of $\bm{\alpha}$ and $\bm{\beta}$, three occurrences of $g_\alpha$ and three occurrences of $g_\beta$. Thus,
\begin{align}
\eqref{eq:boundtrace1}\leqslant \frac{CN^{2\xi}\log N}{N}\sqrt{\frac{t}{N}}\sum_{\alpha\leqslant \beta}\vert g_\alpha(t,z)g_\beta(t,z)\vert^{3/2}
&\leqslant \frac{CN^{2\xi}\log N}{Nt}\sqrt{\frac{t}{N}}\left(\sum_{\alpha=1}^N\vert g_\alpha(t,z)\vert\right)^2\\
&\leqslant N^{2\xi}\log^3 N\sqrt{\frac{N}{t}}
\end{align}
where we used \eqref{eq:sommegk} in the first inequality, the fact that $\vert g_\beta\vert\leqslant Ct^{-1}$ in the second and \eqref{eq:sommegk} again in the final inequality. In order to go below microscopic scales, recall that we used local laws that holds down to mesoscopic scales, we can use the following identity, for $y\leqslant \eta$,
$$\mathrm{Im}\left(\frac{1}{N}\mathrm{Tr}G(E+\mathrm{i}y)\right)\leqslant \frac{\eta}{y}\mathrm{Im}\left(\frac{1}{N}\mathrm{Tr}G(E+\mathrm{i}\eta)\right).$$

Finally, using Lemma \ref{lem:cont}, we get,
\begin{equation}
\sup_{E\in\mathcal{I}^\kappa_{r}}\left\vert\left(\mathds{E}^{\tilde{H}_\tau}
-\mathds{E}^{W_t}\right)\left[\frac{1}{N}\mathrm{Tr}G(z)\right]\right\vert\leqslant N^{5\xi}\tau\sqrt{\frac{N}{t}}\leqslant N^{-\mathfrak{e}}
\end{equation} 
for some ${\mathfrak{e}>0}$ by taking $\tau\in\mathcal{T}_\mathfrak{a}'$ for $\mathfrak{a}>5\xi$.
We can easily generalize this result to a product of trace, indeed taking 
$$F(\tilde{H}_s)=\prod_{k=1}^m F_k\quad \text{with}\quad F_k=\frac{1}{N}\mathrm{Tr}G(z_k),$$
we can take the third derivative and write
\begin{multline}
\label{eq:thirdderiv}\partial_{\alpha\beta}^3F=\sum_{k_1=1}^m\partial_{\alpha\beta}^3F\prod_{k\neq k_1}F_k+3\sum_{k_1=1}^m\sum_{k_2\neq k_1}\partial_{\alpha\beta}^2F_{k_1}\partial_{\alpha\beta}F_{k_2}\prod_{k\neq k_1,k_2}F_k\\
+\sum_{k_1=1}^m\sum_{k_2\neq k_1}\sum_{k_3\neq k_1,k_2}\partial_{\alpha\beta}F_{k_1}\partial_{\alpha\beta}F_{k_2}\partial_{\alpha\beta}F_{k_3}\hspace{-1em}\prod_{k\neq k_1,k_2,k_3}\hspace{-1em}F_k.
\end{multline}
Then, using the first and second derivative of $F_k$,
\begin{align*}
\partial_{\alpha\beta}F_k&=\frac{1}{N}\sum_{k=1}^N\sum_{\substack{\alpha,\,\beta\\\{\alpha,\beta\}=\{i,j\}}}G_{k\alpha}G_{\beta k},\\
\partial_{\alpha\beta}^2F_k&=\frac{1}{N}\sum_{k=1}^N\sum_{\substack{\bm{\alpha},\bm{\beta}\\\{\alpha_k,\beta_k\}=\{i,j\}}}G_{k\alpha_1}G_{\beta_1\alpha_2}G_{\beta_2 k},
\end{align*}
we can bound \eqref{eq:thirdderiv} in a similar way and finishing the bound by Lemma \ref{lem:cont}. Again, now that we have any polynomial of fixed degree, we can also extend to any smooth function  $F$ with polynomial growth. 

Now, a consequence of these uniform bounds in $\mathrm{Re}(z)$ between $\tilde{H}_0=W_t$ and $\tilde{H}_\tau$ for $\tau\in\mathcal{T}_\mathfrak{a}'$ for some small $\mathfrak{a}$ gives us a comparison of the gap distribution between these two matrix ensembles (see \cite{erdos2012bulk} for instance). Namely, there exists $c_1>0$ such that for any $O$ a smooth test function of $n$ variables and any index $i$ such that $\gamma_{i,t}\in\mathcal{I}^\kappa_{r},$ we have for $N$ large enough and $i_1,\dots,i_n$ indices such that $i_k\leqslant N^{c_1}$, 
\begin{equation}\label{eq:gapuniv0}
\left\vert\left(\mathds{E}^{W_t}-\mathds{E}^{\tilde{H}_\tau}\right)\left[
O\left(N\rho^{(N)}_{t}(\gamma_{i,t})(\lambda_i-\lambda_{i,i+i_1}),\dots,N\rho^{(N)}_{t}(\gamma_{i,t})(\lambda_i-\lambda_{i+i_n})\right)\right]\right\vert\leqslant N^{-c_1}.
\end{equation}   

But $\tilde{H}_s$ is a matrix with a small Gaussian component following the conditions of \cite{landon2017convergence}, so that we have, for this matrix ensemble, gap universality. Hence, combining this gap universality with the continuity of the Green's function, we obtain gap universality of the matrix ensemble $D+\sqrt{t}W$. In other words, there exists $c_2>0$ such that, taking the same assumptions as for \eqref{eq:gapuniv0}, we can write,
\begin{multline}\label{eq:gapuniv}
\left\vert\mathds{E}^{W_t}\left[O\left(N\rho^{(N)}_{t}(\gamma_{i,t})(\lambda_i-\lambda_{i,i+i_1}),\dots,N\rho^{(N)}_{t}(\gamma_{i,t})(\lambda_i-\lambda_{i+i_n})\right)\right]\right.\\\left.-\mathds{E}^{\mathrm{GOE}}\left[O\left(N\rho^{(N)}_{sc}(\mu_i)(\lambda_i-\lambda_{i,i+i_1}),\dots,N\rho^{(N)}_{sc}(\mu_i)(\lambda_i-\lambda_{i+i_n})\right)\right]\right\vert\leqslant N^{-c_2}
\end{multline} 
where $\rho_{sc}$ is the density of Wigner's semicircular law and $\mu_i$ its quantiles defined by
\begin{equation}\label{eq:defquant}
\rho_{sc}(x)=\sqrt{4-x^2}\mathds{1}_{[-2,2]},\quad \int_{-\infty}^{\mu_i}\mathrm{d}\rho_{sc}(E)=\frac{i}{N}.
\end{equation}

Finally, \eqref{eq:gapuniv} combined with Theorem \ref{theo:rigidity} gives us the level repulsion estimate $(i)$ for the matrix $W_t$, indeed consider $E\in \mathcal{I}^\kappa_{r}$, and $\ell$ the index such that 
$$\left\vert \gamma_{\ell,t}-E\right\vert\leqslant \min_{k\in \mathcal{I}_{r}^\kappa}\left\vert \gamma_{k,t}-E\right\vert,$$
then, for any $\tilde{\varepsilon}>0$, we have
\begin{align*}
\mathds{P}\left(\left\vert \left\{i,\,\lambda_i\in[E-N^{-1-\xi},E+N^{-1-\xi}]\right\}\right\vert\geqslant 2\right)&\leqslant \sum_{\vert k-\ell\vert \leqslant N^{\tilde{\varepsilon} }}\mathds{P}^{W_t}\left(\left\vert \lambda_k-\lambda_{k+1}\right\vert<N^{-1-\xi}\right)\\
&\leqslant \sum_{|k-\ell|\leqslant N^{\tilde{\varepsilon}}}\mathds{P}^{\mathrm{GOE}}\left(\left(\vert \lambda_k-\lambda_{k+1}\right\vert<N^{-1-\xi}\right)+N^{-c_{\varepsilon}+\tilde{\varepsilon}}\\
&\leqslant N^{-2\xi+\tilde{\varepsilon}}+N^{-c_\varepsilon+\tilde{\varepsilon}}\\
&\leqslant N^{-\xi-\mathfrak{\delta}}
\end{align*}
for some $\delta>0$ by taking $\tilde{\varepsilon}$ and $\xi>0$ small enough. Note that we used rigidity in the first inequality, gap universality in the second and a level repulsion estimate for GOE matrix which for instance can be found in \cite{erdos2015gap} .
  
\paragraph{}In order to get the resolvent estimate $(ii)$, we will use Lemma \ref{lem:cont}. To do so, we will first explain how to get the bound $M$ for  
$$F(\tilde{H}_s)=\frac{1}{\mathrm{Im}\left(\sum_{\gamma=1}^Nq_\gamma^2g_\gamma(t,z)\right)}\langle\mathbf{q},(\tilde{H}-z)^{-1}\mathbf{q}\rangle$$
 for $z\in\mathbb{C}$ down to below microscopic scales. To get the right bound, we will first need to use local laws which hold down to mesoscopic scales $\eta=N^{-1+\xi}$. \\
 Now for the third derivative of $F$, first write
\begin{equation}\label{eq:thirderiv}
\vert\partial_{\alpha\beta}^3F(H_s)\vert=\left\vert\frac{1}{\mathrm{Im}\left(\sum_{\gamma=1}^Nq_\gamma^2g_\gamma(t,z)\right)}\sum_{1\leqslant a,b\leqslant N}\sum_{\bm{\alpha,\beta}}q_aG_{a\alpha_1}G_{\beta_1\alpha_2}G_{\beta_2\alpha_3}G_{\beta_3b}q_b\right\vert
\end{equation}
where $\{\alpha_k,\beta_k\}=\{i,j\}$ for $k=1,2,3$. In order to bound the four terms coming up in the previous equation we will need Corollary \ref{coro:aniso}.
Writing \eqref{eq:anisolocal} for $\mathbf{v}=\mathbf{q}$ and $\mathbf{w}=\mathbf{e_\alpha}$, we obtain
$$\langle \mathbf{q},G\mathbf{e_\alpha}\rangle=q_\alpha g_\alpha(t,z)+\mathcal{O}_\prec\left( \frac{1}{\sqrt{N\eta}}\sqrt{\mathrm{Im}\left(\sum_{\gamma=1}^N q_\gamma^2g_\gamma(t,z) \right)\mathrm{Im}\left(g_\alpha(t,z)\right)}\right).$$

Note that since we want a bound holding down to microscopic scales, the error terms has to be taken into account. In particular, we will consider $\eta\gtrsim N^{-1-\xi}$ so that we can bound every $(N\eta)^{-1/2}$ by $N^{\xi/2}$. In the following computations, we will not bound the errors coming cross terms for simplicity, they can be bounded in a similar way.

\paragraph{}We can divide the sum in \eqref{eq:thirderiv} in three parts. The first case consists in $\{\beta_1,\alpha_2\}=\{\beta_2,\alpha_3\}=\{\alpha,\beta\}$. In this case, note that, necessarily, $\{\alpha_1,\beta_3\}=\{\alpha,\beta\}$ and write
\begin{align}
\sum_{1\leqslant a,b\leqslant N}q_aG_{a\alpha_1}&G_{\beta_1\alpha_2}G_{\beta_2\alpha_3}G_{\beta_3b}q_b=\langle\mathbf{q},G\mathbf{e_{\bm{\alpha}_1}}\rangle G_{\beta_1\alpha_2}G_{\beta_2\alpha_3}\langle \mathbf{e_{\bm{\beta}_3}},G\mathbf{q}\rangle\nonumber\\\nonumber
&\prec\frac{1}{N\eta}\min\left(\vert g_\alpha(t,z)\vert,\,\vert g_\beta(t,z)\vert\right)^2\langle\mathbf{q},G\mathbf{e_{\bm{\alpha}_1}}\rangle\langle \mathbf{e_{\bm{\beta}_3}},G\mathbf{q}\rangle\nonumber\\
\label{eq:firstbothird}&\prec {N^{2\xi}}\left(\min\left(\vert g_\alpha(t,z)\vert,\,\vert g_\beta(t,z)\vert\right)^2\vert q_\alpha g_\alpha(t,z)q_\beta g_\beta(t,z)\vert\right.\\
\label{eq:firstbothird1}&\left.+\min\left(\vert g_\alpha(t,z)\vert,\,\vert g_\beta(t,z)\vert\right)^2\mathrm{Im}\left(\sum_{\gamma=1}^Nq_\gamma^2 g_\gamma(t,z)\right)\sqrt{\vert g_\alpha(t,z)g_\beta(t,z)\vert}\right).
\end{align}
Putting the leading order \eqref{eq:firstbothird} in the sum of \eqref{eq:boundcont2}, we have the bound
\begin{align}
\label{eq:continter0}
\sqrt{\frac{t}{N}}\frac{1}{N\mathrm{Im}\left(\sum_{\gamma=1}^Nq_\gamma^2g_\gamma(t,z)\right)}
&\sum_{1\leqslant \alpha\leqslant \beta\leqslant N}\eqref{eq:firstbothird}\\
&\leqslant\sqrt{\frac{t}{N}}\frac{N^{2\xi}}{N\mathrm{Im}\left(\sum_{\gamma=1}^Nq_\gamma^2g_\gamma(t,z)\right)}\sum_{1\leqslant i\leqslant j\leqslant N}\hspace{-1em}\vert q_\alpha q_\beta\vert\vert g_\alpha(t,z)g_\beta(t,z)\vert^2\\
&\leqslant \sqrt{\frac{t}{N}}\frac{CN^{2\xi}}{N\mathrm{Im}\left(\sum_{\gamma=1}^Nq_\gamma^2g_\gamma(t,z)\right)}\sum_{1\leqslant \alpha\leqslant \beta\leqslant N}\hspace{-1em}\left(q_\alpha^2+q_\beta^2\right)\vert g_\alpha(t,z)g_\beta(t,z)\vert^2\\
&\leqslant\sqrt{\frac{t}{N}}\frac{CN^{2\xi}}{N\mathrm{Im}\left(\sum_{\gamma=1}^Nq_\gamma^2g_\gamma(t,z)\right)}\left(\sum_{\alpha=1}^Nq_\alpha^2\vert g_\alpha(t,z)\vert^2\right)\sum_{\beta=1}^N\vert g_\beta(t,z)\vert^2.
\end{align}
Note that by definition of $g_\alpha(t,z)=(D_\alpha-z-tm_t(z))^{-1}$, the fact that $\mathrm{Im}m_t(z)\asymp 1$ and $\eta\leqslant t$, we can write
\begin{equation}\label{eq:continter1}
\vert g_\alpha(t,z)\vert^2\asymp \frac{1}{t}\mathrm{Im}\left(g_\alpha(t,z)\right).
\end{equation}
Besides we also have from \eqref{eq:sommegk},
\begin{equation}\label{eq:continter2}
\sum_{\alpha=1}^N \vert g_\alpha(t,z)\vert^2\leqslant \frac{C}{t}\sum_{\alpha=1}^N \vert g_\alpha(t,z)\vert\leqslant\frac{CN}{t}\log N.
\end{equation}
Injecting now \eqref{eq:continter1} and \eqref{eq:continter2} in \eqref{eq:continter0}, we get the bound
\begin{align}
\sqrt{\frac{t}{N}}\frac{1}{N\mathrm{Im}\left(\sum_{\gamma=1}^Nq_\gamma^2g_\gamma(t,z)\right)}\sum_{1\leqslant \alpha<\beta\leqslant N}\eqref{eq:firstbothird}&\leqslant C\sqrt{\frac{t}{N}}\frac{N^{2\xi}}{N}\frac{N}{t^2}\log N\\
\label{eq:finalfirst}&\leqslant\frac{N^{3\xi}}{Nt}\sqrt{\frac{N}{t}}.
\end{align}
Looking now at the error term \eqref{eq:firstbothird1} and injecting it in the sum \eqref{eq:boundcont2}, we obtain
\begin{align}
\sqrt{\frac{t}{N}}\frac{1}{N\mathrm{Im}\left(\sum_{\gamma=1}^Nq_\gamma^2g_\gamma(t,z)\right)}\sum_{1\leqslant \alpha<\beta\leqslant N}\eqref{eq:firstbothird1}&\leqslant\sqrt{\frac{t}{N}}\frac{N^{2\xi}}{N}\sum_{1\leqslant \alpha<\beta\leqslant N}\vert g_\alpha(t,z)g_{\beta}(t,z)\vert^{3/2}\nonumber\\
&\leqslant\sqrt{\frac{t}{N}}\frac{N^{2\xi}}{Nt}\left(\sum_{\alpha=1}^N \vert g_\alpha(t,z)\vert\right)^2\nonumber\\
\label{eq:finalfirst1}&\leqslant N^{3\xi}\sqrt{\frac{N}{t}}. 
\end{align}

The second case are the terms where one term is diagonal and the other is an off-diagonal term. More precisely the set $\bm{\alpha}$ and $\bm{\beta}$ such that $\beta_1=\alpha_2$ and $\beta_2\neq \alpha_3$ or $\beta_1\neq\alpha_2$ and $\beta_2=\alpha_3$. Note that necessarily, in that case, $\alpha_1=\beta_3$. For instance consider the term
\begin{align}\label{eq:cont2ndterm}
\langle\mathbf{q},G\mathbf{e_{\bm{\alpha}_1}}\rangle G_{\beta_1\alpha_2}G_{\beta_2\alpha_3}\langle \mathbf{e_{\bm{\beta}_3}},G\mathbf{q}\rangle=\langle\mathbf{q},G\mathbf{e}_i\rangle G_{jj}G_{\alpha\beta}\langle\mathbf{e}_i,G\mathbf{q}\rangle.
\end{align}
Putting all the leading terms from Theorem \ref{entrylocal}, \eqref{eq:offdiaglocal} and \eqref{eq:anisoloca}, we obtain the bound
\begin{align}
\label{eq:cont2ndterm1}\langle\mathbf{q},G\mathbf{e}_\alpha\rangle G_{\beta\beta}&G_{\alpha\beta}\langle\mathbf{e}_\alpha,G\mathbf{q}\rangle\prec\vert q_\alpha\vert^2\vert g_\alpha(t,z)\vert^2\vert g_\beta(t,z)\vert\frac{1}{\sqrt{N\eta}}\min(\vert g_\alpha(t,z)\vert, \vert g_\beta(t,z)\vert)\\
\label{eq:cont2ndterm2}&+\frac{1}{(N\eta)^{3/2}}\mathrm{Im}\left(\sum_{\gamma=1}^Nq_\gamma^2 g_\gamma(t,z)\right)\sqrt{\vert g_\alpha(t,z)g_\beta(t,z)\vert}\vert g_\beta(t,z)\vert\min(\vert g_\alpha(t,z)\vert, \vert g_\beta(t,z)\vert)\\
&\leqslant N^{2\xi}\left(q_\alpha^2\vert g_\alpha(t,z)g_\beta(t,z)\vert^2+\mathrm{Im}\left(\sum_{\gamma=1}^Nq_\gamma^2g_\gamma(t,z)\right)
\vert g_\alpha(t,z)g_\beta(t,z)\vert^{3/2}\right).
\end{align}
Then injecting the bounds \eqref{eq:continter1} and \eqref{eq:continter2} in the sum of \eqref{eq:boundcont2}, one gets
\begin{equation}\label{eq:finalsecond}
\sqrt{\frac{t}{N}}\frac{N^{2\xi}}{N\mathrm{Im}\left(\sum_{\gamma=1}^Nq_\gamma^2g_\gamma(t,z)\right)}\sum_{1\leqslant \alpha<\beta\leqslant N}q_\alpha^2\vert g_\alpha g_\beta\vert^2\leqslant \frac{CN^{3\xi}}{Nt}\sqrt{\frac{N}{t}}
\end{equation}
and for the second term,
\begin{equation}\label{eq:finalsecond1}
\sqrt{\frac{t}{N}}\frac{N^{2\xi}}{N}\sum_{1\leqslant \alpha<\beta\leqslant N}\vert g_\alpha(t,z)g_\beta(t,z)\vert^{3/2}\leqslant N^{3\xi}\sqrt{\frac{N}{t}}.
\end{equation}

The final case consists of $\bm{\alpha}$ and $\bm{\beta}$ such that $\left\{\{\beta_1,\alpha_2\},\{\beta_2,\alpha_3\}\right\}=\{\{\alpha,\alpha\},\{\beta,\beta\}\}.$ Note that, in this case, we necessarily have $\alpha_1\neq\beta_3$. For instance, consider the term
\begin{equation}
\langle\mathbf{q},G\mathbf{e_{\bm{\alpha}_1}}\rangle G_{\beta_1\alpha_2}G_{\beta_2\alpha_3}\langle \mathbf{e_{\bm{\beta}_3}},G\mathbf{q}\rangle=\langle\mathbf{q},G\mathbf{e}_\alpha\rangle G_{\beta\beta}G_{\alpha\alpha}\langle\mathbf{e}_\beta,G\mathbf{q}\rangle.
\end{equation}
Again, taking the leading terms from the local laws from Threom \ref{entrylocal} and Corollary \ref{coro:aniso}, 
\begin{align}
\langle\mathbf{q},G\mathbf{e}_\alpha\rangle G_{\beta\beta}G_{\alpha\alpha}\langle\mathbf{e}_\beta,G\mathbf{q}\rangle\prec\vert q_\alpha q_\beta\vert\vert g_\alpha(t,z)g_\beta(t,z)\vert^2+N^{\xi}\vert g_\alpha(t,z)g_\beta(t,z)\vert^{3/2}\mathrm{Im}\left(\sum_{\gamma=1}^Nq_\gamma^2 g_\gamma(t,z)\right).\label{eq:contthirdcase}
\end{align}
Then using similar bounds as the first case one gets
\begin{equation}\label{eq:finalthird}
\sqrt{\frac{t}{N}}\frac{1}{N\mathrm{Im}\left(\sum_{\gamma=1}^Nq_\gamma^2g_\gamma(t,z)\right)}\sum_{1\leqslant \alpha<\beta\leqslant N}\eqref{eq:contthirdcase}\prec N^{2\xi}\left(\frac{1}{Nt}\sqrt{\frac{N}{t}}+\sqrt{\frac{N}{t}}\right).
\end{equation}

Finally, putting together \eqref{eq:finalfirst}, \eqref{eq:finalfirst1}, \eqref{eq:finalsecond}, \eqref{eq:finalsecond1} and \eqref{eq:finalthird}, we get the bound, for $\eta=N^{-1+\xi}$.
\begin{equation}\label{eq:contoneprod}
\sqrt{\frac{t}{N}}\frac{1}{N}\sum_{1\leqslant \alpha<\beta\leqslant N}\partial_{\alpha\beta}^3F(\tilde{H}_s)\prec N^{3\xi}\sqrt{\frac{N}{t}}.
\end{equation}

 In order to get a bound for $\eta$ below microscopic scales, we can use the following inequality, for any $y\leqslant \eta$, which can be found in \cite{erdos2012bulk}*{Section 8},
$$\vert\langle \mathbf{v},G(E+\mathrm{i}y)\mathbf{w}\rangle\vert\leqslant C\log N\frac{\eta}{y}\mathrm{Im}\langle\mathbf{v},G(E+i\eta)\mathbf{w}\rangle.$$
This bounds allows us to get below microscopic scales for $F$ and its derivatives since they only involves such quantity as $\langle \mathbf{v},G(E+\mathrm{i}y)\mathbf{w}\rangle$.
Thus, uniformly in $E\in\mathcal{I}_{r}^\kappa$ and $N^{-1-\xi}\leqslant\eta\leqslant t$, we have
\begin{equation}
M=\mathcal{O} \left({N^{5\xi}}\sqrt{\frac{N}{t}}\right).
\end{equation}
Using now Lemma \ref{lem:cont}, we can make $W_t$ undergo the dynamics $\tilde{H}_s$ up to a time $\tau\ll N^{-5\xi}\sqrt{\frac{t}{N}}$ with $\xi$ arbitrarily small in order to get the right bound.

\paragraph{}For a product of resolvent entries, one can do similar computations and bounds. Indeed consider $m\geqslant 0$, and
$$F(\tilde{H}_s)=\prod_{k=1}^mF_k(\tilde{H}_s)\quad\text{with}\quad
F_k(\tilde{H}_s)= \langle\mathbf{q},G(z_k)\mathbf{q}\rangle,$$
then one can write the third derivative of $F$ as \eqref{eq:thirdderiv} and using the fact that
\begin{align}
\partial_{\alpha\beta}F_k&=-\hspace{-1em}\sum_{\{\mathbf{\alpha},\mathbf{\beta}\}=\{i,j\}}\langle\mathbf{q},G\mathbf{e}_{\alpha}\rangle\langle\mathbf{e}_{\beta},G\mathbf{q}\rangle,\\
\partial_{\alpha\beta}^2F_k&=\sum_{\bm{\alpha},\bm{\beta}}\langle\mathbf{q},G\mathbf{e}_{\alpha_1}\rangle G_{\beta_1,\alpha_2}\langle\mathbf{e}_{\beta_2},G\mathbf{q}\rangle
\end{align}
where $\{\alpha_i,\beta_i\}=\{i,j\}$ and using the same type of bounds as for \eqref{eq:contoneprod}, we obtain the result \eqref{eq:controlresolv} since the extension to any smooth function with polynomial growth is also clear.
\end{proof}
\subsection{Reverse heat flow}\label{sec:reverse}
In Subsection \ref{subsec:perfect}, we showed Theorem \ref{theo:resultint}, which corresponds to our main result for the matrix $H_\tau=W_t+\sqrt{\tau}\mathrm{GOE}$ for $N^{-1}\ll \tau\ll t$ with a general Wigner matrix $W$ in the definition of $W_t$. Thus, the overwhelming probability bound holds for the eigenvectors of this matrix $H_\tau$ giving us a strong form of quantum unique ergodicity for the deformed Gaussian divisible ensemble. In order to remove the small Gaussian component in the matrix, we will use the reverse heat flow technique from \cites{erdos2010peche, erdos2011universality} which allows us to obtain an error as small as we want in total variation between two matrix ensembles. In order to use this technique, we need the smoothness assumption on the matrix $W$ given by Definition \ref{def:smoothwig}. We first introduce some notation for this section.\\
\indent As before, we let $\nu$ denote the distribution of the entries of $W$, and $\varphi$ denote the density of $\nu$ with respect to $\rho$, the Gaussian distribution with mean zero and variance one, that is, $\D \nu=\varphi\D \varrho$. The reverse heat flow technique gives the existence of a probability distribution $\tilde{\nu}_s$ for any $s$ small enough such that making  $\tilde{\nu}_s$ undergo the Ornstein-Uhlenbeck process of generator
$$A:=\frac{1}{2}\frac{\partial^2}{\partial x^2}-\frac{x}{2}\frac{\partial}{\partial x}$$ 
approaches the distribution $\nu$ in total variation.\\
\indent This process on all the matrix entries induces the Dyson Brownian motion process on the eigenvalues. Thus the following proposition tells us that there exists a distribution of a matrix from the Gaussian divisible process of the form 
$$\widetilde{W}_s=\sqrt{1-s}\,\widetilde{W}+\sqrt{s}\mathrm{GOE}$$
that approximates as close as polynomially possible a smooth Wigner matrix $W$. The precise statement is written in the following proposition.
\begin{proposition}[\cite{erdos2011universality}]\label{prop:reverseheat}
Let $K$ be a positive integer and $\nu=\varphi \varrho$ a distribution smooth in the sense that it follows the conditions (ii) and (iii) of Definition \ref{def:smoothwig}. Then there exists $s_K$ a small positive constant depending on $K$ such that for any $0<s\leqslant s_K$, there exists a probability density $\psi_s$ with mean zero and variance one such that we have the inequality
\begin{equation}
\int\left\vert e^{sA}\psi_s-\varphi\right\vert\mathrm{d}\varrho\leqslant Cs^K
\end{equation} 
for some positive constant $C$ depending only on $K$. Besides we also have the inequality for the joint probability of all matrix entries in the following sense,
\begin{equation}
\int\left\vert e^{sA^{\otimes N^2}}\psi_s^{\otimes N^2}-\varphi^{\otimes N^2}\right\vert\mathrm{d}\varrho\leqslant C N^2s^K
\end{equation}
\end{proposition}

Now, see that this proposition holds for any fixed $K$ so that, taking $s=N^{-\varepsilon}$ for some small $\varepsilon$ we can choose a large $K$ only depending on $\varepsilon$ (and not on $N$) so that we can obtain any polynomial bound between the two matrix ensembles. This property allows us to get overwhelming probability bounds on the eigenvectors since the total variation distance of the distribution of the eigenvector entries is smaller than the total variation distance between the joint probability of the matrix entries.

\section{Proofs of main results}\label{secloc}\label{sec:proof}
Now that we have the result for the Gaussian divisible ensemble $H_\tau$ with $N^{-1}\ll \tau\ll t$ by Section \ref{localmaxprinc}, combining it with the continuity argument from the last subsection, we are able to prove Theorem \ref{theo:resultGaus} and Corollary \ref{QUE}. These two results are a consequence of the following proposition showing the convergence of moments for the eigenvectors of $W_t$.
\begin{proposition}\label{prop:phi}~\\
Let $\kappa\in(0,1)$ and $m$ an integer, for a set of indices $I\subset\mathcal{A}^\kappa_{r}$, such that $\vert I\vert=m$, we have for any deterministic unit vector $\mathbf{q}=\mathbf{q}_N,$
\begin{equation}\label{eq:phi}
\mathds{E}\left[P\left(\left({\frac{N}{\sigma_t(\mathbf{q},k)^2}}\vert \langle\mathbf{q},u_k\rangle\vert^2\right)_{k\in I}\right)\right]\xrightarrow[N\longrightarrow\infty]{}\mathds{E}\left[P\left(\left(\mathcal{N}_k^2\right)_{k=1}^m\right)\right]
\end{equation}
with $(\mathcal{N}_k)_k$ a family of independent normal random variables.
\end{proposition}

See now the proof of Theorem \ref{theo:resultGaus} and Corollary \ref{QUE} given by Proposition \ref{prop:moment}.
\begin{proof}[Proof of Theorem \ref{theo:resultGaus}.]
Proposition \ref{prop:phi} exactly gives us that the joint moments of the renormalized eigenvectors converge to those of independent normal random variables which is the result of Theorem \ref{theo:resultGaus}.
\end{proof}

\begin{proof}[Proof of Corollary \ref{QUE}]
By Proposition \ref{prop:moment} and Corollary \ref{coro:conti}, we have the following inequality, for some $\varepsilon>0$, 
\begin{equation}\label{eq312}
\mathds{E}\left[\vert u_k(\alpha)^2\vert \right]=\frac{1}{N}\sigma_t^2(\alpha,k)+\mathcal{O}\left(\frac{N^{-\varepsilon} }{Nt}\right).
\end{equation}
By Markov's inequality, we can write 
\begin{align}
\mathds{P}\left(\frac{Nt}{\vert A\vert}\right.\left.\middle\vert\sum_{\alpha\in A}\vert u_k(\alpha)\vert^2-\frac{1}{N}\sum_{\alpha\in A}\sigma_t^2(\alpha,k)\middle\vert>c\right)&\leqslant \frac{N^2t^2}{c^2\vert A\vert^2}\mathds{E}\left[\left\vert\sum_{\alpha\in A}\vert u_k(\alpha)\vert^2-\frac{1}{N}\sum_{\alpha\in A}\sigma_t^2(\alpha,k)\right\vert^2\right],\nonumber\\
&\label{eq314}\leqslant\frac{N^2t^2}{c^2\vert A\vert^2}\left(\mathfrak{A}-2\mathfrak{B}+\mathfrak{C}\right)
\end{align}
We now need to evaluate the three terms in the last inequality using \eqref{eq312}, first we have
\begin{align*}
\mathfrak{A}:=\mathds{E}\left[\left(\sum_{\alpha\in A}\vert u_k(\alpha)\vert^2\right)^2\right]=\frac{1}{N^2}\left(\sum_{\alpha\in A}\sigma_t^2(\alpha,k)\right)^2+\frac{2}{N^2}\sum_{\alpha \in A}\sigma_t^4(\alpha,k)+\mathcal{O}\left(\frac{N^{-\varepsilon}\vert A\vert^2}{N^2t^2}\right).
\end{align*}
Likewise,
\begin{align*}
\mathfrak{B}:=\frac{1}{N}\sum_{\beta\in A}\sigma_t^2(\beta,N)\mathds{E}\left[\sum_{\alpha\in A}\vert u_k(\alpha)\vert^2\right]=\frac{1}{N^{2}}\sum_{\alpha,\beta\in A}\sigma_t^2(\alpha,k)\sigma^2_t(\beta,k)+\mathcal{O}\left(\frac{N^{-\varepsilon }\vert A\vert^2}{N^2t^2}\right).
\end{align*}
Finally, $\mathfrak{C}$ is just a deterministic term,

$$\mathfrak{C}:=\left(\frac{1}{N}\sum_{\alpha\in A}\sigma_t^2(\alpha,k)\right)^2.$$
Putting all three terms together, we get
\begin{equation}
\mathfrak{A}-2\mathfrak{B}+\mathfrak{C}=\frac{2}{N^2}\sum_{\alpha\in A}{\sigma_t^4(\alpha,k)}+\mathcal{O}\left(\frac{N^{-\varepsilon} \vert A\vert^2}{N^2t^2}\right)\leqslant C\left(\frac{\vert A\vert}{N^2t^2}+\frac{N^{-\varepsilon} \vert A\vert^2}{N^2t^2}\right)
\end{equation}
The claim then follows from injecting the last inequality in \eqref{eq314}. 
\end{proof}

We finish now with the proof of Proposition \ref{prop:phi}.
\begin{proof}[Proof of Proposition \ref{prop:phi}]
By Corollary \ref{coro:conti}, we know that for some $\tau\in\mathcal{T}_\mathfrak{a}'$ there exists $\varepsilon>0$ such that 
\begin{equation}\label{eq:boundcontin}
\left\vert\mathds{E}\left[P\left(\left(\frac{N}{\sigma_t(\mathbf{q},k)^2}\vert\langle\mathbf{q},u_k\rangle\vert^2\right)_{k\in I}\right)\right]-\mathds{E}\left[P\left(\left(\frac{N}{\sigma_t(\mathbf{q},k)^2}\vert\langle\mathbf{q},u_k^{\tilde{H}_\tau}\rangle\vert^2\right)_{k\in I}\right)\right]\right\vert\leqslant N^{-\varepsilon}
\end{equation}
and by Proposition \ref{prop:moment}, we know that, recalling the definition of $H_s$, for some $\tau'\ll t$ there exists a $\varepsilon'>0$ such that
\begin{equation}\label{eq:boundgaussdiv}
\left\vert\mathds{E}\left[P\left(\left(\frac{N}{\sigma_t(\mathbf{q},k)^2)}\vert\langle\mathbf{q},u_k^{H_{\tau'}}\rangle\vert^2\right)_{k\in I}\right)\right]-\mathds{E}\left[P\left(\left(\mathcal{N}_k^2\right)_{k=1
}^m\right)\right]\right\vert\leqslant N^{-\varepsilon'}.
\end{equation}

Now, we need to see that $\tilde{H}_\tau$ defined in \eqref{eq:deftildeh} has the same law as $H_{\tau'}$ for some $\tau'\ll t$. Note that we can write the law of the entries of $\tilde{H}_\tau$ as
\begin{equation}
\tilde{H}_{ij}(\tau)\overset{d}{=}D_{ij}+e^{-\frac{\tau}{2t}}\sqrt{t}W_{ij}+\sqrt{t\left(1-e^{-\frac{\tau}{t}}\right)}\frac{1}{\sqrt{N}}\mathcal{N}^{(ij)},
\end{equation}
where $\left(\mathcal{N}^{(ij)}\right)_{i\leqslant j}$ is a family of independent standard Gaussian random variables. Doing the scaling 
\begin{equation}
\begin{split}
t'&=te^{-\tau/t}=\mathcal{O}(t),\\
\tau'&=\sqrt{t\left(1-e^{-\frac{\tau}{t}}\right)}=\mathcal{O}(\tau)
\end{split}
\end{equation}
one can write
$$\tilde{H}_\tau=D+\sqrt{t'}W+\sqrt{\tau'}\mathrm{GOE}.$$
Finally, we can apply Proposition \ref{prop:moment} to $\tilde{H}$ so that \eqref{eq:boundgaussdiv} applies and combining it with \eqref{eq:boundcontin} we get the convergence of moments for the eigenvectors of $W_t$. 
\end{proof}

 Combining Theorem \ref{theo:resultint} and Proposition \ref{prop:moment}, we are now able to prove Theorem \ref{theo:resultQUE}.
\begin{proof}[Proof of Theorem \ref{theo:resultQUE}]
Let $\varepsilon$ and $D$ two positive constants and consider $s=\tau/t,$. There exists then a large $K$, which does not depend on $N$, such that by Proposition \ref{prop:reverseheat} there exists a matrix $\widetilde{W}$ such that the total variation distance between the distribution of $W$ and $\sqrt{1-s}\widetilde{W}+\sqrt{Ns}\mathrm{GOE}$ is smaller than $N^{-D}$.\\ Denote $u_1,\dots,\,u_N$ the $L^2-$normalized eigenvectors of $W_t=D+\sqrt{t}W$ and $\tilde{u}_1,\dots,\,\tilde{u}_N$ the normalized eigenvectors of $\widetilde{W}_t(s)=D+\sqrt{t(1-s)}\,\widetilde{W}+\sqrt{ts}\mathrm{GOE}$. Now, since we have in the overwhelming probability bound \eqref{eq:resultQUE} the $N^\varepsilon$ degree of liberty, we can do the scaling $t'=\sqrt{t(1-s)}$ as $s\ll 1$ and still get \eqref{eq:resultQUE} for the deformed Gaussian divisible ensemble $\widetilde{W}_t(s),$ thus one can write
\begin{align*}
\mathds{P}&\left(\left\vert\sum_{\alpha\in I}\left(u_k(\alpha)^2-\frac{1}{N}\sigma_t^2(\alpha,k)\right)\right\vert\geqslant N^\varepsilon\Xi(\tau)\right)
\\
&\leqslant \mathds{P}\left(\left\vert \sum_{\alpha\in I}\left(\tilde{u}_k(\alpha)^2-\frac{1}{N}\sigma_t^2(\alpha,k)\right)\right\vert\geqslant N^{\varepsilon/2}\Xi(\tau)\right)
+\mathds{P}\left(\left\vert \sum_{\alpha\in I}\left(u_k(\alpha)^2-\tilde{u}_k(\alpha)\right)\right\vert\geqslant N^{\varepsilon/2}\Xi(\tau)\right)\\
&\leqslant N^{-D}
\end{align*}
where for the last inequality we used the quantum unique ergodicity proved in Theorem \ref{theo:resultint} for the deformed Gaussian divisible ensemble of which $\tilde{\mathbf{u}}$ are the eigenvectors and Proposition $\ref{prop:reverseheat}$ in order. 

Now, in order to get the error $\Xi$ we now need to optimize the error
\[
\Xi(\tau_0)=\frac{\widehat{I}}{\sqrt{N\tau_0}}+\widehat{I}\frac{\tau_0}{t}=\frac{\widehat{I}}{(Nt)^{1/3}}=\Xi\quad\text{with}\quad \tau_0=\left(\frac{t^2}{N}\right)^{1/3}.
\]

We can do the same thing for the quantity $\sum_{\alpha\in I}u_k(\alpha)u_l(\alpha)$ and get the final result.
\end{proof}

\end{document}